\documentclass[oneside,reqno]{amsart}

\usepackage[shortcuts]{extdash}

\usepackage{hyperref}

\usepackage[totalheight=22.5 true cm, totalwidth=16 true cm]{geometry}

\usepackage{xypic}

\usepackage{enumerate}
\usepackage{amsthm}
\usepackage{amsmath}
\usepackage{amssymb}
\usepackage{amsfonts}

\newtheorem{theo}{Theorem}[section]
\newtheorem{lemma}[theo]{Lemma}
\newtheorem{prop}[theo]{Proposition}
\newtheorem{cor}[theo]{Corollary}
\newtheorem{defi}[theo]{Definition}
\newtheorem{rem}[theo]{Remark}

\theoremstyle{definition}
\newtheorem{ex}[theo]{Example}

\newcommand{\f}{\phi}

\newcommand{\ida}{\mathfrak{a}}

\newcommand{\p}{\mathfrak{p}}

\newcommand{\spec}{\operatorname{Spec}}

\newcommand{\G}{\mathcal{G}}

\newcommand{\Gl}{\operatorname{GL}}

\newcommand{\Aut}{\operatorname{Aut}}
\newcommand{\Hom}{\operatorname{Hom}}

\newcommand{\A}{\mathbb{A}}
\newcommand{\sdim}{\sigma\text{-}\dim}

\newcommand{\ord}{\operatorname{ord}}

\newcommand{\X}{\mathcal{X}}

\newcommand{\N}{\mathcal{N}}

\newcommand{\s}{\sigma}
\newcommand{\de}{\delta}
\newcommand{\ds}{\delta\sigma}

\newcommand{\I}{\mathbb{I}}

\newcommand{\hs}{{{}^\sigma\!}}

\newcommand{\pis}{\pi_0^\sigma}
\newcommand{\ks}{$k$\=/$\s$}

\renewcommand{\sc}{{\sigma o}}
\newcommand{\Gm}{\mathbb{G}_m}

\newcommand{\hsi}{^{\sigma^i}\!}
\newcommand{\hsn}{^{\sigma^n}\!}

\newcommand{\Ga}{\mathbb{G}_a}

\newcommand{\lcm}{\operatorname{lcm}}
\newcommand{\kn}{k^{[n]}}
\newcommand{\sred}{{\sigma\text{-}\operatorname{red}}}

\newcommand{\ZZ}{\mathbb{Z}}
\newcommand{\rank}{\operatorname{rank}}

\newcommand{\po}{{\rm{(P1$^u$)}}}
\newcommand{\pt}{{\rm{(P3$^u$)}}}
\newcommand{\pw}{{\rm{(P2$^u$)}}}
\newcommand{\poweak}{{\rm{(P1)}}}
\newcommand{\ptweak}{{\rm{(P3)}}}
\newcommand{\pwweak}{{\rm{(P2)}}}

\usepackage{todonotes} 

\title{Trivializing Torsors for Difference Algebraic Groups}

\date{\today}

\author{Michael Wibmer}
\address{Michael Wibmer, School of Mathematics, University of Leeds, \url{https://sites.google.com/view/wibmer}}
\email{m.wibmer@leeds.ac.uk}

\author{Annette Bachmayr}
\address{Annette Bachmayr, Department of Mathematics, RWTH Aachen University}
\email{bachmayr@mathematik.rwth-aachen.de}

\thanks{This work was supported by DFG (Deutsche Forschungsgemeinschaft): SFB-TRR 195}

\date{\today}


\begin{document}
\maketitle

\begin{center} \emph{Dedicated to the memory of Zo\'{e} Chatzidakis} \end{center}

\medskip

\begin{abstract}
We study torsors for groups defined by algebraic difference equations.
Our main result provides necessary and sufficient conditions on the base difference field for all such torsors to be trivial. 
We also present an application to the uniqueness of difference Picard-Vessiot extensions for linear differential equations.
\end{abstract}


\section{Introduction}

It is a classical question in algebraic geometry and field arithmetic to understand conditions on the base field $k$  that guarantee that $H^1(k,G)$ is trivial or small for large classes of linear algebraic groups $G$ over $k$. (See e.g., \cite{Serre:GaloisCohomology}). For example, perfect fields $k$ of cohomological dimension at most one (which include  $C_1$-fields) are characterized by the property that $H^1(k,G)$ is trivial for all connected linear algebraic groups over $k$ and this is equivalent to $H^1(k,G)$ being trivial for all semisimple algebraic groups $G$ over $k$. It is also known that $H^1(k,G)$ is finite for every linear algebraic group $G$ over $k$ if $k$ is perfect and bounded, i.e., $k$ has only finitely many algebraic extensions of any fixed degree. There has been a sustained interest in obtaining analogs of these results for differential algebraic groups (\cite{Pillay:ThePicardVessiotTheoryConstrainedCohomology}, \cite{MinchenkoOvchinnikov:TrivialityOfDifferentialGaloisCohomologyOfLinearDifferentialAlgebraicGroups}, \cite{ChatzidakisPillay:GeneralizedPVExtensionsAndDifferentialGaloisCohomology},
\cite{LeonSanchezPillay:DifferentialGaloisCohomologyAndParameterizedPicardVessiotExtensions}, \cite{LeonSanchezMeretzkyPillay:MoreOnGaloisCohomologyDefinabilityAndDifferentialAlgebraicGroups},  \cite{MeretzkyPillay:PicardVessiotExtensionsLinearDifferentialAlgebraicGroupsAndTheirTorsors}). One of the most profound results in the area is the following theorem due to Anand Pillay.

\begin{theo}[\cite{Pillay:ThePicardVessiotTheoryConstrainedCohomology}]  \label{theo: Pillay}
	For an ordinary differential field $k$ of characteristic zero the following statements are equivalent:
	\begin{enumerate}
		\item For every linear differential algebraic group $G$ over $k$ and every $G$-torsor $X$ we have $X(k)\neq\emptyset$.
		\item We have $H^1_\delta(k,G)=0$ for every linear differential algebraic group $G$ over $k$.
		\item The field $k$ is algebraically closed and for every $m\geq 1$ and every $A\in k^{m\times m}$ the linear differential equation $\de(Y)=AY$ has a solution in $\Gl_m(k)$.
	\end{enumerate}	
\end{theo}
Here (ii) refers to Kolchin's constrained cohomology and the equivalence of (i) and (ii) is well known (\cite[Chapter VII]{Kolchin:differentialalgebraicgroups}). So the striking part of Theorem \ref{theo: Pillay} is that (iii) implies (i) or~(ii). As a finite Galois extension defines a torsor for a finite constant group scheme, we see that a field $k$ of characteristic zero is algebraically closed if and only if all torsors for all linear algebraic groups over $k$ are trivial. Moreover, for a differential field $k$ and $A\in k^{m\times m}$, the differential algebraic variety $X=\{Y\in\Gl_m|\ \de(Y)=AY\}$ over $k$ is a torsor for the linear differential algebraic group $\Gl_m^\delta=\{g\in\Gl_m|\ \de(g)=0\}$.  Thus, Theorem~\ref{theo: Pillay} can be reformulated as follows: If for all $m\geq 1$ all torsors for $\Gl_m^\delta$ are trivial and all torsors for linear algebraic groups are trivial, then all torsors for all linear differential algebraic groups are trivial.

Generalizations of Theorem \ref{theo: Pillay} to partial differential fields are discussed in \cite{ChatzidakisPillay:GeneralizedPVExtensionsAndDifferentialGaloisCohomology} and \cite{MinchenkoOvchinnikov:TrivialityOfDifferentialGaloisCohomologyOfLinearDifferentialAlgebraicGroups}. Another generalization of Theorem \ref{theo: Pillay}, that can be seen as a relative version, not requiring any assumptions on the ordinary base differential field, can be found in \cite{MeretzkyPillay:PicardVessiotExtensionsLinearDifferentialAlgebraicGroupsAndTheirTorsors}.

The question of an appropriate difference analog of Theorem \ref{theo: Pillay} was raised in \cite[Section~3]{Pillay:ThePicardVessiotTheoryConstrainedCohomology} and in \cite{ChatzidakisPillay:GeneralizedPVExtensionsAndDifferentialGaloisCohomology} it is shown that the \emph{immediate} analog of Theorem \ref{theo: Pillay} fails for difference algebraic groups. The counterexample provided is related to the difference algebraic subgroup $G=\{g\in\Gm|\ x^{-2}\s(x)=1\}$ of the multiplicative group~$\Gm$ and the $G$-torsor $X=\{x\in\Gm|\  x^{-2}\s(x)=a\}$, where $a\in k^\times$. More generally, for $\alpha_0,\ldots,\alpha_m\in\ZZ$ with $\alpha_m\neq 0$ and $a\in k^\times$, we have a difference algebraic subgroup $G=\{g\in\Gm|\ g^{\alpha_0}\s(g)^{\alpha_1}\ldots\s^m(g)^{\alpha_m}=1\}$ of $\Gm$ and a $G$-torsor $X=\{x\in\Gm|\ x^{\alpha_0}\s(x)^{\alpha_1}\ldots\s^m(x)^{\alpha_m}=a\}$.

The main result of this paper is an analog of Theorem \ref{theo: Pillay} for difference algebraic groups.
In essence, it states that the solvability of these multiplicative equations $y^{\alpha_0}\s(y)^{\alpha_1}\ldots\s^m(y)^{\alpha_m}=a$ is the only obstruction to the validity of the immediate difference analog of Theorem \ref{theo: Pillay}. However, as difference algebraic geometry is somewhat more intricate than differential algebraic geometry, this has to be taken with a grain of salt. Specifically, while differential algebraic geometry is more or less synonymous with the model theory of differential fields, the connections are more subtle in the difference case.

Before precisely stating our main result, we need to explain other obstructions to the validity of the immediate difference analog of Theorem \ref{theo: Pillay}, that are not related to difference algebraic subgroups of $\Gm$ and can be resolved by passing to higher powers of the endomorphism $\s$. We note that the philosophy that things can be smoothed out by passing to higher powers of $\s$ is also present in other work (e.g., \cite{ChatzidakiHrushovskiPeterzil:ModelTheoryofDifferenceFieldsIIPeriodicIdelas} or \cite{Wibmer:Chevalley}).

Some torsors for linear difference algebraic groups do not have a point in any difference field. For example, the difference variety $X=\{x\in \Gm|\ x^2=1,\ \s(x)=-x\}$ is a torsor for the linear difference algebraic group $G=\{g\in\Gm|\ g^2=1,\ \s(g)=g\}$ (under left multiplication) but $X(k)=\emptyset$ for every difference field $k$. In other words, this torsor will never be trivial, no matter over which base difference field it is considered.
From a model theoretic perspective, where one is confined to working inside a monster model of the theory of difference fields, the torsor $X$ does not exist, i.e., it collapses to the empty set. However, this perspective is unsatisfactory from the point of view of difference algebraic geometry, because the system of equations $x^2=1,\ \sigma(x)=-x$ is consistent, in the sense that it has a solution in a difference ring. For example, $(1,-1,1,-1,1,\ldots)$ is a solution in the ring of sequences (equipped with the shift).

The problem that not every torsor has a point in a difference field can be resolved by looking for solutions slightly beyond the base difference field $k$. In the above example, we have $(1,-1)\in X(k^{[2]})$, where, for an integer $n\geq 1$, we let $\kn=k\times\ldots\times k$ denote the product of $n$ copies of $k$, considered as a difference algebra over $k$ via \label{page: kn}
$k\to\kn,\ a\mapsto (\s^{n-1}(a),\ldots,\s(a),a)$ and
$$\s\colon\kn\to\kn,\ (a_1,\ldots,a_n)\mapsto (\s^n(a_n),a_1,\ldots,a_{n-1}).$$ 


A similar obstruction to the immediate difference analog of Theorem \ref{theo: Pillay} occurs if we look at linear difference algebraic groups of the form $G=\{g\in \G|\ \s(g)=\psi(g) \}$, where $\G\leq\Gl_m$ is an algebraic subgroup defined over $\mathbb{Q}$ and $\psi\colon \G\to \G$ is a morphism of algebraic groups. For example, we can take $G=\{g\in\Gl_m|\ \s(g)=(g^T)^{-1}\}$. 
It is known that all $G$-torsors are of the form $X=\{x\in\Gl_m|\ \s(x)= (x^T)^{-1}A\}$ \label{page: X} for some $A\in\Gl_m(k)$. The equation $x^T\s(x)=A$ is non-linear and there appears to be no reason why we should be able to find a solution in the base difference field $k$, just because $k$ is algebraically closed and we can solve linear difference equations in $k$. However, if we pass from $\s$ to $\s^2$, then every $x\in X$ satisfies $\s^2(x)=x(A^T)^{-1}\s(A)$ (a linear difference equation over $(k,\s^2)$) and we have a difference variety
$X_2=\{x\in\Gl_m|\ \s^2(x)=x(A^T)^{-1}\s(A)\}$
over the difference field $(k,\s^2)$. Thus, if every linear difference equation over $(k,\s^2)$ is solvable, then $X_2(k)\neq\emptyset$. Furthermore, $X_2(k)\neq\emptyset$ turns out to be equivalent to $X(k^{[2]})\neq\emptyset$ (Remark \ref{rem: points k vs kn}).
In Section \ref{section: The induction principle} below we will explain systematically how a difference variety $X$  over $(k,\sigma)$ gives rise to a difference variety $X_n$ over $(k,\s^n)$ for every $n\geq 1$. We will also see that for a linear difference algebraic group $G$, the cohomology sets $H^1_{\s^n}(k,G_n)$ 
naturally form a direct system. The main result of this paper is the following difference analog of Theorem \ref{theo: Pillay}.

\enlargethispage{5mm}
\begin{theo} \label{theo: main 2nd version}
	For an ordinary difference field $k$ of characteristic zero the following statements are equivalent:
\begin{enumerate}
	\item[\rm{(P1)}] For every linear difference algebraic group $G$ over $k$ and for every $G$-torsor $X$, there exists an integer $n\geq 1$ such that $X(\kn)\neq\emptyset$. 
	\item[\rm{(P2)}] We have $\lim\limits_{n\to \infty} H^1_{\s^n}(k,G_n)=0$ for every linear difference algebraic group $G$ over $k$.
		\item[\rm{(P3)}]		
	\begin{enumerate}
		\item[\rm{(a)}] The field $k$ is algebraically closed.
		\item[\rm{(b)}] For every $m\geq 1$ and every $A\in\Gl_m(k)$ there exists an integer $n\geq 1$ such that the equation $\s(Y)=AY$ has a solution in $\Gl_m(\kn)$. 	
		\item[\rm{(c)}] 
		For all $\alpha_1,\ldots,\alpha_m\in \mathbb{Z}$ with $\alpha_m\neq 0$ and $a\in k^\times$, there exists an $n\geq 1$ such that the equation $y^{\alpha_0}\s(y)^{\alpha_1}\ldots\s^m(y)^{\alpha_m}=a$
		has a solution in $\kn$.
	\end{enumerate}
\end{enumerate}
\end{theo}


We note that we do not make any assumptions on the constants $k^\s=\{a\in k|\ \s(a)=a\}$ of $k$. The constants of an algebraically closed differential field are automatically algebraically closed. However, this need not be the case in the difference world. Condition \ptweak{} (b) is similar to what is called ``strongly PV-closed'' or ``linearly closed'' in \cite{ChatzidakisPillay:GeneralizedPVExtensionsAndDifferentialGaloisCohomology}, \cite{ChalupnikKowalski:DifferenceSheavesAndTorsors} and \cite{KowalskiPillay:OnAlgebraicSigmaGroups}. A difference field $k$ is strongly PV-closed if for every $m\geq 1$ and every $A\in\Gl_m(k)$ the equation $\s(Y)=AY$ has a solution in $\Gl_m(k)$. Our condition \ptweak{} (b) is weaker than being strongly PV-closed and is more in line with the Picard-Vessiot theory of linear difference equations (\cite{SingerPut:difference}), where one allows zero-divisors in the solution rings and usually assumes the constants to be algebraically closed. The constants of a strongly PV-closed difference field have absolute Galois group $\widehat{\ZZ}$ (because for every $m\geq 1$ the equation $\s^m(y)=y$ has to have $m$ solutions in $k$ that are linearly independent over the constants). On the other hand, our condition \ptweak{} (b) does not obviously imply any restrictions on the constants (although it follows from \ptweak{} (a) that the absolute Galois group of the constants is procyclic).

The multiplicative difference equations in \ptweak{} (b) are intimately related to toric difference varieties studied in \cite{GaoHuangWangYuan:ToricDifferenceVariety}, \cite{GaoHuangYuan:BinomialDifferenceIdeal} \cite{Wang:ToricPdifferenceVariety}, \cite{Wang:FiniteBasisForRadicalWell-mixedDifferenceIdealsGeneratedByBinomials}.

Condition {\rm \ptweak{} (b)} is equivalent to the special case of {\rm (P1)} where $G=\Gl_m^\s=\{g\in\Gl_m|\ \s(g)=g\}$, while condition {\rm \ptweak{} (c)} is equivalent to the special case of {\rm (P1)} for one-relator subgroups of $\Gm$, i.e., for difference algebraic subgroups of $\Gm$ of the form $G=\{g\in\Gm|\ g^{\alpha_0}\s(g)^{\alpha_1}\ldots\s^m(g)^{\alpha_m}=1\}$.



For a linear difference algebraic group $G$ over a difference field $k$, we say that $k$ is \emph{$G$\=/trivial} if for all $G$-torsors $X$ there exists an $n\geq 1$ such that $X(k^{[n]})\neq \emptyset$. With this nomenclature our main result can be reformulated as follows: If $k$ is  $G$-trivial for all $G$ where $G$ is a linear algebraic group, a difference algebraic group of the form $G=\Gl_m^\s=\{g\in\Gl_m|\ \s(g)=g\}$ (for some $m\geq 1$), or a one-relator subgroup of $\Gm$, then $k$ is  $G$-trivial for all linear difference algebraic groups $G$ over $k$. In this sense, the one-relator subgroups of $\Gm$ are the only obstruction to the immediate difference analog of Theorem~\ref{theo: Pillay}. 

Note that the integer $n$ in the first condition of Theorem \ref{theo: main 2nd version} depends on both the difference algebraic group $G$ and the $G$-torsors $X$. Similarly, in the third condition, the integer $n$ depends on $A\in \Gl_m(k)$ or $a\in k^\times$. We can replace both conditions by a stronger condition where an $n$ has to exist uniformly for all $G$-torsors, and where $n$ is not allowed to depend on $A$ and $a$. As an additional result we show that these stronger conditions are still equivalent:

\begin{theo} \label{theo: main}
	For an ordinary difference field $k$ of characteristic zero the following statements are equivalent:
\begin{enumerate}
	\item[\rm{(P1$^{u}$)}] For every linear difference algebraic group $G$ over $k$ there exists an integer $n\geq 1$ such that $X(\kn)\neq\emptyset$ for every $G$-torsor $X$. 
	\item[\rm{(P2$^u$)}] For every linear difference algebraic group $G$ over $k$, there exists an integer $n\geq 1$ such that the natural map $H^1_{\s}(k,G)\to H^1_{\s^n}(k,G_n)$ is trivial.
	
	\item[\rm{(P3$^u$)}] 
	\begin{enumerate}
		\item[\rm{(a)}] The field $k$ is algebraically closed.
		\item[\rm{(b)}] For every $m\geq 1$, there exists an integer $n\geq 1$ such that for every $A\in\Gl_m(k)$ the equation $\s(Y)=AY$ has a solution in $\Gl_m(\kn)$. 	
		\item[\rm{(c)}] For all $\alpha_1,\ldots,\alpha_m\in \mathbb{Z}$ with $\alpha_m\neq 0$, there exists an $n\geq 1$ such that for all $a\in k^\times$ the equation $y^{\alpha_0}\s(y)^{\alpha_1}\ldots\s^m(y)^{\alpha_m}=a$
		has a solution in $\kn$.
	\end{enumerate}
\end{enumerate}
\end{theo}

There appears to be no previous difference analogs of Theorem \ref{theo: Pillay} in the literature. However, the connection between $H^1(k,\Gl_m^\s)$ and solutions of linear difference equations was already noted in \cite[Section 7.3]{ChalupnikKowalski:DifferenceSheavesAndTorsors}, where also the counterexample from \cite{ChatzidakisPillay:GeneralizedPVExtensionsAndDifferentialGaloisCohomology} is interpreted in cohomological terms.

In Section \ref{sec: Torsors over sclosed sfields} we show that every $\s$-closed difference field satisfies the equivalent conditions of Theorem~\ref{theo: main}. While we establish our main result (\ptweak{}$\Rightarrow$\poweak{}) only in characteristic zero, we think that it is interesting to note that an algebraically closed field $k$ of positive characteristic $p$, equipped with a Frobenius endomorphism $\s_q\colon k\to k,\ a\mapsto a^q$ (with $q$ a power of $p$), satisfies property \pt{}. Indeed, that $(k,\s_q)$ satisfies \pt{} (b) with $n=1$ is a special case of the Lang-Steinberg Theorem and that $(k,\s_q)$ satisfies \pt{} (c) with $n=1$ follows from $k$ being algebraically closed.

\medskip

Our main result has applications in the difference Galois theory of linear differential equations (\cite{DiVizioHardouinWibmer:DifferenceGaloisTheoryOfLinearDifferentialEquations}), a special case of the parameterized Picard-Vessiot theory, where one considers a linear differential equation and an auxiliary operator such as a derivation or an endomorphism that is acting on the coefficients of the linear differential equation. A typically example is a linear differential equation with coefficients that depend on a parameter and the auxiliary operator is the derivation with respect to the parameter or a shift operator on the parameter.
Our application concerns the uniqueness of the solution rings in the context of this Galois theory, known as $\s$\=/Picard-Vessiot rings.
This is similar to the application in \cite[Cor. 1.4]{Pillay:ThePicardVessiotTheoryConstrainedCohomology} of Theorem \ref{theo: Pillay} to the uniqueness of the solution rings in the case when the auxiliary operator is a derivation. However, our situation is slightly different because $\s$-Picard-Vessiot rings are generally not unique, even when the constants are $\s$-closed. In general, two $\s$-Picard-Vessiot rings for the same equation only become isomorphic when passing to higher powers of $\s$. More precisely, in \cite[Cor. 1.17]{DiVizioHardouinWibmer:DifferenceGaloisTheoryOfLinearDifferentialEquations} it is shown that if $K$ is a $\de\s$-field such that the $\de$-constants $K^\de$ are a $\s$-closed $\s$-field and $R_1,R_2$ are two $\s$-Picard-Vessiot rings for the equation $\de(y)=Ay$ (for some $A\in K^{m\times m}$), then there exists an integer $n\geq 1$ such that $R_1$ and $R_2$ are isomorphic as $K$-$\de\s^n$-algebras. We note that the integer $n$ is allowed to depend on $R_1$ and $R_2$. In particular, the need to pass to higher powers of $\s$ also naturally arises in the parameterized Picard-Vessiot theory and this phenomenon is well-aligned with our philosophy underlying the statements of Theorems \ref{theo: main} and \ref{theo: main 2nd version}. 

We improve the above uniqueness result in two ways. Firstly, we show that the same uniqueness result also holds if we only assume that $K^\de$ satisfies the equivalent conditions of Theorem~\ref{theo: main 2nd version}. Secondly, we show that if $K^\de$ satisfies the equivalent conditions of Theorem \ref{theo: main} (e.g., $K^\de$ is $\s$\=/closed), then the integer $n$ can be chosen uniformly for all pairs $R_1,R_2$ of $\s$-Picard-Vessiot rings.

\medskip

We conclude the introduction with an outline of the content. In Section \ref{sec: Basic definition} we set the stage by recalling definitions and introducing our notation. In Section \ref{section: The induction principle} we formally introduce, for every difference variety $X$ over $(k,\s)$, the difference varieties $X_n$ over $(k,\s^n)$ and we establish the important induction principle. It states that if $N$ is a normal difference algebraic subgroup of a difference algebraic group $G$, such that $\lim\limits_{n\to \infty} H^1_{\s^n}(k,N)=0$ and $\lim\limits_{n\to \infty} H^1_{\s^n}(k,(G/N)_n)=0$, then $\lim\limits_{n\to \infty} H^1_{\s^n}(k,G_n)=0$. This principle allows us to use decomposition theorems for difference algebraic groups to reduce the proof to certain types of mostly quite explicitly described classes of difference algebraic groups. In Section \ref{section: The induction principle} we also show that \poweak{} and \pwweak{} are equivalent, which is fairly straight forward once the appropriate definitions are in place. The fact that \poweak{} implies \ptweak{} is also fairly immediate and established in Section \ref{sec: first observations} along with some reformulations of property \ptweak{}.
The most difficult part of the proof of Theorem \ref{theo: main} is the implication \ptweak{}$\Rightarrow$\poweak{}. This relies on several structure and decomposition theorems for difference algebraic groups. We first treat special classes of difference algebraic groups (e.g., order zero, diagonalizable, difference dimension theory zero) and build up to the general case. In Section \ref{sec: Torsors over sclosed sfields} we show that a $\s$-closed difference field satisfies the equivalent conditions of Theorem \ref{theo: main}. Finally, in Section \ref{sec: Applicationt to PV} we present the application to the parameterized Picard-Vessiot theory.

\medskip

The authors are grateful to Zo\'{e} Chatzidakis and Anand Pillay for helpful discussions during the preparation of the article.


\section{Basic definitions and constructions}
\label{sec: Basic definition}


In this section we recall the basic definitions from difference algebra and some results about difference algebraic groups. The reader with a solid grounding in difference algebra may prefer to skip this section.


We start by recalling some definitions from difference algebra. Standard references for difference algebra are \cite{Cohn:difference} and \cite{Levin:difference}. Our main references for difference algebraic groups are \cite{Wibmer:FinitenessPropertiesOfAffineDifferenceAlgebraicGroups} and \cite{Wibmer:almostsimple}. All rings are assumed to be commutative and unital. A \emph{difference ring} (or \emph{$\s$-ring} for short) $R$ is a ring together with a ring endomorphism $\s\colon R\to R$. A morphism of $\s$-rings is a morphism of rings that commutes with $\s$.

Let $k$ be a $\s$-ring. A \emph{\ks-algebra} is a $\s$-ring $R$ together with a $k$-algebra structure such that $k\to R$ is a morphism of $\s$-rings. A morphism of \ks-algebras $R\to S$ is a morphism of $\s$-rings that also is a morphism of $k$-algebras. We will write $\Hom_{k,\s}(R,S)$ for the set of morphisms of \ks-algebras.

 If $R$ and $S$ are \ks-algebras, then $R\otimes_k S$ is a \ks-algebra via $\s(r\otimes s)=\s(r)\otimes \s(s)$. A \ks\=/algebra $R$ is \emph{finitely $\s$-generated} if there exists a finite subset $B\subseteq R$ such that $R=k[B,\s(B),\ldots]$. We write $R=k\{B\}$ for a $k$-$\s$-algebra $R$ if $B\subseteq R$ is such that $R=k[B,\s(B),\ldots]$.

An ideal $\ida$ of a $\s$-ring $R$ is a \emph{$\s$-ideal} if $\s(\ida)\subseteq\ida$. If $\s^{-1}(\ida)=\ida$, then $\ida$ is a \emph{reflexive} $\s$-ideal. If $f\s(f)\in\ida$ implies that $f\in\ida$, then $\ida$ is a \emph{perfect} $\s$-ideal. For a subset $B$ of $R$, the $\s$-ideal $[B]=(B,\s(B),\s^2(B),\ldots)\subseteq R$ is the $\s$-ideal \emph{$\s$-generated by $B$}.

Let $k$ be a $\s$-field. A \emph{$\s$-polynomial} over $k$ in the $\s$-variables $y_1,\ldots,y_n$ is a polynomial in the variables $y_1,\ldots,y_n,\s(y_1),\ldots,\s(y_n),\s^2(y_1),\ldots$ over $k$. The \ks-algebra of all $\s$-polynomials is denoted by $k\{y\}=k\{y_1,\ldots,y_n\}$.

Informally speaking, a \emph{$\s$-variety} $X$ over $k$ (or a \emph{\ks-variety}) is the set of solutions of a set $F\subseteq k\{y\}$ of $\s$-polynomials. To make this precise, we can define $X$ to be the functor from the category of \ks-algebras to the category of sets that assigns to every \ks-algebra $R$ the $R$-valued solutions of $F$. A morphism of \ks-varieties is then simply a natural transformation of functors. Alternatively, we can fix a sufficiently large\footnote{The precise requirement here is that $\mathcal{U}$ contains an isomorphic copy of every finitely $\s$-generated \ks-algebra.} \ks-algebra $\mathcal{U}$ and define a \ks-variety to be the set of $\mathcal{U}$-valued solutions of $F$. Then a morphism of \ks-varieties is simply a map that is given by $\s$-polynomials in the coordinates. A \emph{$\s$-algebraic group\footnote{It may seems more accurate to add the word affine or linear here. However, since we will only consider affine $\s$-varieties we refrain from doing so.}} (over $k$) is a group object in the category of $\s$-varieties (over $k$). A morphism of $\s$-algebraic groups (over $k$) is a morphism of $\s$-varieties (over $k$) that respects the group structure.

In practice it does not matter if we work with functors or inside $\mathcal{U}$. In fact, we will suppress the $R$ or the $\mathcal{U}$ from the notation most of the time. For example,
we may write $\f\colon \Gm\to \Gm,\ g\mapsto\s(g)g$ to specify a morphism $\f$ of $\s$-algebraic group from the multiplicative group to the multiplicative group. This can either be read as the transformation of functors such that $\f_R\colon \Gm(R)\to\Gm(R),\ g\mapsto\s(g)g$ for every \ks-algebra $R$ or as the map $\Gm(\mathcal{U})\to\Gm(\mathcal{U}),\ g\mapsto\s(g)g$.

 For a \ks-variety $X$
$$\I(X)=\{f\in k\{y\}|\ f(x)=0 \ \forall \ x\in X\}$$
is a $\s$-ideal of $k\{y\}$. The \ks-algebra
$k\{X\}=k\{y\}/\I(X)$ is called the \emph{$\s$-coordinate ring} of $X$.
The functor $X\rightsquigarrow k\{X\}$ induces an equivalence of categories between the category of \ks-varieties and the category of finitely $\s$-generated \ks-algebras. Moreover $X\simeq \Hom_{k,\s}(k\{X\},-)$ as functors. (See \cite[Section 1.2]{Wibmer:FinitenessPropertiesOfAffineDifferenceAlgebraicGroups}.) Allowing ourselves the freedom of calling any functor isomorphic to a \ks-variety (as defined above) a \ks-variety, it follows that a functor from the category of \ks\=/algebras to the category of sets is a \ks-variety if and only if it is representable by a finitely $\s$-generated \ks-algebra. 
The category of $\s$-algebraic groups over $k$ is equivalent to the category of finitely $\s$-generated \ks-Hopf algebras, i.e., Hopf algebras over $k$, that are finitely $\s$-generated \ks-algebras such that the Hopf algebra structure maps are morphisms of \ks-algebras.

A $\s$-ring $R$ is called \emph{$\s$-reduced} if $\s\colon R\to R$ is injective. The quotient $R_\sred=R/(0)^*$ of $R$ by the reflexive closure $(0)^*=\{f\in R\ |\ \exists \ n\geq 1:\  \s^n(f)=0\}$ of the zero ideal is $\s$-reduced. A \ks-variety $X$ is \emph{$\s$-reduced} if $k\{X\}$ is $\s$-reduced. If the zero ideal of $k\{X\}$ is perfect, then $X$ is called \emph{perfectly $\s$-reduced}. If $k\{X\}$ is $\s$-reduced and an integral domain, then $X$ is called \emph{$\s$-integral}.

A subfunctor $Y$ of a \ks-variety $X$ is a \emph{$\s$-closed $\s$-subvariety} if it is defined by a $\s$-ideal $\I(Y)\subseteq k\{X\}$, i.e., 
$$
\xymatrix{
Y(R) \ar@{^(->}[r] \ar_-\simeq[d] & X(R) \ar[d]^\simeq \\
\Hom_{k,\s}(k\{X\}/\I(Y),R) \ar@{^(->}[r] & \Hom_{k,\s}(k\{X\},R)	
}
$$ 
commutes for every \ks-algebra $R$. A morphism $\f\colon X\to Y$ of \ks-varieties is a \emph{$\s$-closed embedding} if it induces an isomorphism between $X$ and a $\s$-closed $\s$-subvariety of $Y$; equivalently, the dual map $\f^*\colon k\{Y\}\to k\{X\}$ is surjective (\cite[Lemma 1.6]{Wibmer:FinitenessPropertiesOfAffineDifferenceAlgebraicGroups}).

If $\f\colon X\to Y$ is a morphism of \ks-varieties, we denote with $\f(X)$ the smallest $\s$-closed $\s$\=/subvariety of $Y$ such that $\f$ factors through the inclusion $\f(X)\hookrightarrow Y$ (cf. \cite[Lemma 1.5]{Wibmer:FinitenessPropertiesOfAffineDifferenceAlgebraicGroups}).

If $K$ is a $\s$-field extension of $k$ and $X$ a $\s$-variety over $k$, we denote the base change of $X$ from $k$ to $K$ by $X_K$. So $X_K(R)=X(R)$ for every $K$-$\s$-algebra $R$. Then $X_K$ is a $\s$-variety over $K$ with $\s$-coordinate ring $K\{X_K\}=k\{X\}\otimes_k K$. 

If $G$ is a $\s$-algebraic group and $H$ a $\s$-closed $\s$-subvariety such that $H(R)$ is a subgroup of $G(R)$ for every \ks-algebra $R$, then $H$ is a called a \emph{$\s$-closed subgroup} of $G$ and we will write $H\leq G$. If moreover $H(R)$ is a normal subgroup of $G(R)$ for every \ks-algebra $R$, then $H$ is \emph{normal} in $G$ and we indicate this by $H\unlhd G$.

The forgetful functor from \ks-algebras to $k$-algebras has a left adjoint $S\rightsquigarrow [\s]_k S$ which we need to describe explicitly. For a $k$-algebra $S$ and $i\geq 1$ we set ${\hsi S}=S\otimes_k k$ where the tensor product is formed by using $\s^i\colon k\to k$. More generally, if $X$ is an object over $k$, we denote by ${\hsi X}$ the base change of $X$ via $\s^i\colon k\to k$. We set $S[i]=S\otimes_k{\hs S}\otimes_k\ldots\otimes_k {\hsi S}$ and 
let $[\s]_k S$ denote the direct limit (i.e., the union) of the $S[i]$'s. We define the structure of a \ks-algebra on $[\s]_k S$ by setting 
$$\s(f_0\otimes(f_1\otimes\lambda_1)\otimes\ldots\otimes(f_i\otimes\lambda_i))=1\otimes (f_0\otimes 1)\otimes(f_1\otimes\s(\lambda_1))\otimes\ldots\otimes (f_i\otimes\s(\lambda_i))\in S[i+1]$$
for $f_0,\ldots,f_i\in S$ and $\lambda_1,\ldots,\lambda_i\in k$.

\begin{ex}
Consider the univariate polynomial ring $S=k[y]$. Then $$S[i]=k[y, 1\otimes(y\otimes 1), \dots, 1\otimes \dots \otimes (y\otimes 1)]$$ and $\s\colon S[i]\to S[i+1]$ shifts each generator to the next one. Hence $[\s]_kS=k\{y\}$, the difference polynomial ring in one variable.
\end{ex}

For the sake of brevity and to be consistent with our nomenclature for difference algebraic geometry, we use the term \emph{variety over $k$} to mean affine scheme of finite type over $k$.
If $\X$ is a variety over $k$, then the functor $R\rightsquigarrow[\s]_k\X(R)=\X(R)$ is a $\s$-variety over $k$. Indeed, $k\{[\s]_k\X\}=[\s]_k k[\X]$ where $k[\X]$ is the coordinate ring of $\X$. (Cf. \cite[Section 1.3]{Wibmer:FinitenessPropertiesOfAffineDifferenceAlgebraicGroups}.)

By an \emph{algebraic group} over $k$ we mean an affine group scheme of finite type over $k$. By a \emph{closed subgroup} of an algebraic group we mean a closed subgroup scheme. 
Let $\G$ be an algebraic group over $k$. 
Then $[\s]_k\G$ is a $\s$-algebraic group over $k$. If confusion is unlikely we will simply write $\G$ for $[\s]_k\G$. In particular, by a $\s$-closed subgroup of $\G$, we mean a $\s$-closed subgroup of $[\s]_k \G$.

Let $G$ be a $\s$-closed subgroup of an algebraic group $\G$. Then $G$ is defined by a $\s$-ideal $\I(G)\subseteq [\s]_k k[\G]$. For $i\geq 0$ the intersection $\I(G)\cap k[\G][i]\subseteq k[\G][i]=k[\G\times\ldots\times {\hsi\G}]$ defines a closed subgroup $G[i]$ of $\G\times\ldots\times {\hsi\G}$. We call $G[i]$ the \emph{$i$-th order Zariski closure} of $G$ in $\G$. For $i=0$ we simply speak of the \emph{Zariski closure}. If $G[0]=\G$, we say that $G$ is \emph{Zariski-dense} in $\G$.

In \cite[Theorem 3.7]{Wibmer:FinitenessPropertiesOfAffineDifferenceAlgebraicGroups} it is shown that there exist integers $d,e\geq 0$ such that
$\dim(G[i])=d(i+1)+e$ for $i\gg 0$. Moreover every $\s$-algebraic group $G$ is isomorphic to a $\s$-closed subgroup of an algebraic group $\G$ (\cite[Prop. 2.16]{Wibmer:FinitenessPropertiesOfAffineDifferenceAlgebraicGroups}), e.g., $\G$ can be chosen as $\Gl_m$ for some $m$,
and the integer $d$ does not depend on the choice of $\G$ or the choice of such an embedding. We call $d$ the \emph{$\s$-dimension} of $G$ and denote it by $\sdim(G)$. Furthermore, if $\sdim(G)=0$, then the integer $\ord(G)=e\geq 0$ only depends on $G$ and is called the \emph{order} of $G$. If $\sdim(G)>0$, we set $\ord(G)=\infty$.

\begin{ex}
  As an example, consider $\G=\mathbb{G}_a$ and $G=\{g \in \G \mid \s^2(g)-\s(g)-1=0\}$. Then $G[0]=\G$ and $G[1]=\G\times {^\s \G}$ as they do not see the defining equation. But $G[2]$ is the closed subgroup of $\G \times {^\s \G} \times {^{\s^2}\G}$ defined by $\s^2(x)-\s(x)-1=0$, so it has dimension $2$. The dimension of $G[3]$ equals $2$, since it is defined by the two equations $\s^2(x)-\s(x)-1=0$ and $\s^3(x)-\s^2(x)-1=0$ inside  $\G \times {^\s \G} \times {^{\s^2}\G}\times {^{\s^3}\G} $. Similarly, it follows that $d_i=2$ for all $i\geq 2$. Hence $\sdim(G)=0$ and $\ord(G)=2$. 
\end{ex}

Let $G$ be a $\s$-algebraic group and let $N\unlhd G$ be a normal $\s$-closed subgroup. The quotient $G/N$ can be defined by a universal property (\cite[Def. 3.1]{Wibmer:almostsimple}). It exists (Ibid., Theorem~3.3) and is well-behaved with respect to the $\s$-dimension and the order (Ibid. Corollary 3.13): Namely, $\sdim(G)=\sdim(N)+\sdim(G/N)$ and $\ord(G)=\ord(N)+\ord(G/N)$. Furthermore, the isomorphism theorems hold for $\s$-algebraic groups (Ibid., Section 5).

A morphism $\f\colon G\to H$ of $\s$-algebraic groups is called a \emph{quotient map} if $\f(G)=H$. (For other equivalent characterizations of quotient maps see Proposition 4.10 in \cite{Wibmer:almostsimple}.) In particular, $\f$ is a quotient map if and only if $H\simeq G/N$ for a normal $\s$-closed subgroup $N$ of $G$. Any morphism of $\s$\=/algebraic groups factors uniquely as a quotient map followed by a $\s$-closed embedding (\cite[Theorem~4.14]{Wibmer:almostsimple}).

The kernel $\ker(\f)$ of a morphism $\f\colon G\to H$ of $\s$-algebraic groups is given by $\ker(\f)(R)=\ker(\f_R)$ for every \ks-algebra $R$. It is a normal $\s$-closed subgroup of $G$. Indeed, $\I(\ker(\f))\subseteq k\{G\}$ is the ideal of $k\{G\}$ generated by $\f^*(\mathfrak{m}_H)$, where $\mathfrak{m}_H$ is the kernel of the counit $k\{H\}\to k$.
A sequence $1 \to N \xrightarrow{\alpha} G\xrightarrow{\beta} H\to 1$ of $\s$-algebraic groups is called \emph{exact} if $\beta$ is a quotient map and $\alpha$ is a $\s$-closed embedding that identifies $N$ with the kernel of $\beta$. 

A $\s$-algebraic group $G$ is called \emph{$\s$-\'{e}tale} if every element of $k\{G\}$ satisfies a separable polynomial over $k$. 
Recall that a finite dimensional $k$-algebra $S$ is \'{e}tale if every element of $S$ satisfies a separable polynomial over $k$. For a $k$-algebra $R$, let $\pi_0(R)$ denote the union of all \'{e}tale $k$-subalgebras of $R$. 

Let $G$ be a $\s$-algebraic group. Among all morphisms from $G$ to a $\s$-\'{e}tale $\s$-algebraic group, there exists a universal one $G\to \pi_0(G)$. See \cite[Prop. 6.13]{Wibmer:FinitenessPropertiesOfAffineDifferenceAlgebraicGroups}. Indeed $k\{\pi_0(G)\}=\pi_0(k\{G\})$. The kernel $G^o$ of $G\to\pi_0(G)$ is called the \emph{identity component} of $G$. If $G=G^o$, then $G$ is called \emph{connected}. 
See \cite[Lemma 6.7]{Wibmer:almostsimple} for other characterizations of connectedness.

Let $G$ be a $\s$-algebraic group over $k$. A \emph{$G$-torsor} is a \ks-variety $X$ together with a morphism $G\times X\to X$ of \ks-varieties such that $G(R)\times X(R)\to X(R)$ is a simply transitive group action of the group $G(R)$ on the set $X(R)$ for every \ks-algebra $R$. A morphism of $G$-torsors is a $G$\=/equivariant morphism of \ks-varieties. A $G$-torsor isomorphic to $G$ with left-translation is called \emph{trivial}. A $G$-torsor $X$ is trivial if and only if $X(k)\neq\emptyset$ (\cite[Remark 1.3]{BachmayrWibmer:TorsorsForDifferenceAlgebraicGroups}).   

\medskip  

{\bf Throughout this article $k$ denotes a $\s$-field.} All varieties, $\s$-varieties, algebraic groups and $\s$-algebraic groups are assumed to be over $k$ unless the contrary is indicated.

\section{The induction principle} \label{section: The induction principle}

Let $1\to N\to G\to H\to 1$ be an exact sequence of $\s$-algebraic groups over $k$. The principle alluded to in the title of this section is the statement that  $\lim_{n\to\infty}H^1_{\s^n}(k,G_n)=0$ if $\lim_{n\to\infty}H^1_{\s^n}(k,N_n)=0$ and  $\lim_{n\to\infty}H^1_{\s^n}(k,H_n)=0$. The induction principle is crucial for the proof of \ptweak{}$\Rightarrow$\poweak{} because, after decomposing a general $\s$-algebraic group into more manageable pieces, it allows us to restrict our considerations to those more manageable pieces. In this section we also show that \poweak{} and \pwweak{} are equivalent and similarly, their uniform variants \po{} and \pw{} are equivalent. 

\medskip

{\bf Throughout Section \ref{section: The induction principle} we denote with $k$ a $\s$-field (of arbitrary characteristic).}

 \subsection{From $\s$ to $\s^n$} \label{subsection: From s to sn}
In this subsection we introduce the $\s^n$-varieties $X_n$ for every $\s$-variety $X$. Throughout this subsection we fix an integer $n\geq 1$.
We use the prefix $k$-$\s^n$ or $\s^n$ instead of $k$-$\s$ or $\s$ to indicate that we are working over the difference field $(k, \s^n)$ instead of $(k,\s)$.

 We need to be able to pass from \ks-varieties to $k$-$\s^n$-varieties and vice versa.
 The most conceptual way to do this is to consider the left adjoint and the right adjoint of the functor $(R,\s)\rightsquigarrow (R,\s^n)$ from \ks-algebras to $k$-$\s^n$-algebras which we will now construct. 
 Let $(S,\tau)$ be a $k$-$\s^n$-algebra. We define
 $${_1S}=S\otimes_k{\hs S}\otimes_k\ldots\otimes_k{^{\s^{n-1}\!}S}$$ and $\s\colon {_1S}\to {_1S}$ by $$\s(a_0\otimes(a_1\otimes\lambda_1)\otimes\ldots\otimes(a_{n-1}\otimes\lambda_{n-1}))=\tau(a_{n-1})\s(\lambda_{n-1})\otimes(a_0\otimes 1)\otimes (a_1\otimes\s(\lambda_1))\otimes\ldots\otimes (a_{n-2}\otimes\s(\lambda_{n-2})),$$ where $a_0,\ldots,a_{n-1}\in S$ and $\lambda_1,\ldots,\lambda_{n-1}\in k$. Then ${_1S}$ is a \ks-algebra and the inclusion into the first factor $S\to {_1S}$ is a morphism of $k$-$\s^n$-algebras. If $R$ is a \ks-algebra
 and ${_1S}\to R$ is a morphism of \ks-algebras, then $S\to {_1S} \to R$ is a morphism of $k$-$\s^n$-algebras. Conversely, if $\psi\colon S\to R$ is a morphism of $k$-$\s^n$-algebras, then $$ {_1S}\to R,\ a_0\otimes(a_1\otimes\lambda_1)\otimes\ldots\otimes(a_{n-1}\otimes\lambda_{n-1})\mapsto \psi(a_0)\s(\psi(a_1))\lambda_1\ldots\s^{n-1}(\psi(a_{n-1}))\lambda_{n-1}$$ is a morphism of \ks-algebras. It is straight forward to check that these two constructions are inverse to each other. So	  
 $$\Hom_{k,\s^n}(S,R)\simeq\Hom_{k,\s}({_1S},R)$$ for every \ks-algebra $R$.
 Thus the functor $S\rightsquigarrow{_1 S}$ (with the obvious definition on morphisms) is the left adjoint of $(R,\s)\rightsquigarrow (R,\s^n)$. (This is why the $1$ is on the left side.)
 
 \begin{ex} \label{ex: left adjoint}
   Consider the difference polynomial ring $S=k\{y\}_{\s^n}=k[y_1,\ldots,y_m,\s^n(y_1),\ldots,\s^n(y_m),\ldots]$ over the difference field $(k,\s^n)$ in the $\s^n$-variables $y_1,\ldots,y_m$. Then $_1(k\{y\}_{\s^n})$ can be identified with $k\{y\}_\s$, the difference polynomial ring over $(k,\s)$ in the $\s$-variables $y_1,\ldots,y_m$. The identification is such that for $1\leq i\leq n-1$ and $1\leq j\leq m$, the variable $\s^i(y_j)\in k\{y\}_\s$ corresponds to the image of $y_j\otimes 1\in {\hsi}S$ under the inclusion ${\hsi}S\to {_1S}$.
   
   More generally, if $S=k\{y\}_{\s^n}/[F]_{\s^n}$ is a quotient of $k\{y\}_{\s^n}$ by the $\s^n$-ideal $[F]_{\s^n}$ of $k\{y\}_{\s^n}$ $\s^n$\=/generated by $F\subseteq k\{y\}_{\s^n}$, then $_1S$ can be identified with $k\{y\}_\s/[F]_\s$, the quotient of $k\{y\}_\s$ by the $\s$-ideal generated by $F\subseteq k\{y\}_{\s^n}\subseteq k\{y\}_\s$.
   %
 \end{ex}

 Next we will construct the right adjoint of $(R,\s)\rightsquigarrow (R,\s^n)$. If $S$ is any $k$-algebra, then we write ${_{\s^{i}}S}$ for the $k$-algebra that equals $S$ as a ring and where we multiply with elements in $k$ via $\s^i \colon k \to k$. In other words, ${_{\s^{i}}S}$ is obtained from $S$ by restriction of scalars via $\s^i \colon k \to k$.

 
 Let $(S,\tau)$ be a $k$-$\s^n$-algebra. We define
 $$S_1={_{\s^{n-1}}S}\times{_{\s^{n-2}}S}\times\ldots\times{_\s S}\times S$$ as a $k$-algebra and we consider $S_1$ as a \ks-algebra via
 $$\s(a_1,a_2,\ldots,a_n)=(\tau(a_n),a_1,\ldots,a_{n-1}).$$ The projection $S_1\to S$ onto the last factor is a morphism of $k$-$\s^n$-algebras. So if $R\to S_1$ is a morphism of \ks-algebras, then $R\to S_1\to S$ is a morphism of $k$-$\s^n$-algebras. Conversely, if $\psi\colon R\to S$ is a morphism of $k$-$\s^n$-algebras, then $$R\to S_1,\ a\mapsto (\psi(\s^{n-1}(a)),\ldots,\psi(\s(a)),\psi(a))$$ is a morphism of \ks-algebras. One can check that these two constructions are inverse to each other. So	 
 $$\Hom_{k,\s^n}(R,S)\simeq\Hom_{k,\s}(R,S_1)$$ for every \ks-algebra $R$.
 In other words, $S\rightsquigarrow S_1$ is the right adjoint of $(R,\sigma)\rightsquigarrow (R,\s^n)$. (This is why the one is on the right.) Note that with $\kn$ defined as on page \pageref{page: kn}, we have $\kn=(k,\s^n)_1$. More generally, for a \ks-algebra $R$ we define the \ks-algebra $R^{[n]}$ by
 $$R^{[n]}=(R,\s^n)_1={_{\s^{n-1}}R}\times \ldots\times  R.$$ One can verify directly that for $l\geq 1$ the map
  $$R^{[n]}\to R^{[nl]},\ (a_1,\ldots,a_n)\mapsto (\s^{n(l-1)}(a_1),\ldots,\s^{n(l-1)}(a_n),\ldots,\s^n(a_1),\ldots,\s^n(a_n),a_1,\ldots,a_n)$$
  is a morphism of \ks-algebras.

 
 If $X$ is a \ks-variety, then we can define a functor $X_n$ from $k$-$\s^n$-algebras to sets by $$X_n(S)=X(S_1)$$
 for any $k$-$\s^n$-algebra $S$. 
 Conversely, if $Z$ is a $k$-$\s^n$-variety, then we can define a functor $Z_1$ from \ks\=/algebras to sets by
 $$Z_1(R)=Z(R,\s^n)$$
 for any \ks-algebra $R$.
 
 \begin{lemma} \mbox{ } \label{lem: X1 Xn}
 	\begin{enumerate}
 		\item If $X$ is a \ks-variety, then $X_n$ is a $k$-$\s^n$-variety. Indeed, the $\s^n$-coordinate ring $k\{X_n\}$ of $X_n$ is $(k\{X\},\s^n)$. 
 		\item If $Z$ is a $k$-$\s^n$-variety, then $Z_1$ is a \ks-variety. Indeed, $k\{Z_1\}={_1 (k\{Z\})}$.
 		
 	\end{enumerate}
 \end{lemma}
 \begin{proof}  	
 	(i) For a $k$-$\s^n$-algebra $S$ we have $X_n(S)=\Hom_{k,\s}(k\{X\},S_1)=\Hom_{k,\s^n}(k\{X\},S)$. 
 	 As $(k\{X\},\s)$ is finitely $\s$-generated, $(k\{X\},\s^n)$ is finitely $\s^n$-generated. So $X_n$ is indeed a $k$\=/$\s^n$\=/variety.
 	 
 	(ii) For a \ks-algebra $R$ we have $Z_1(R)=\Hom_{k,\s^n}(k\{Z\},R)=\Hom_{k,\s}({_1 (k\{Z\})},R)$. As $k\{Z\}$ is finitely $\s^n$-generated, $_1(k\{Z\})$ is finitely $\s$-generated. So $Z_1$ is indeed a \ks-variety.
 \end{proof}
 
 So $X\rightsquigarrow X_n$ is a functor from the category of \ks-varieties to the category of $k$-$\s^n$-varieties and $Z\rightsquigarrow Z_1$ is a functor from the category of $k$-$\s^n$-varieties to the category of \ks-varieties. 
 For a \ks-variety $X$ and a $k$-$\s^n$-variety $Z$ we have
 $$\Hom(X,Z_1)=\Hom_{k,\s}({_1(k\{Z\})}, k\{X\})=\Hom_{k,\s^n}(k\{Z\},k\{X\})=\Hom(X_n,Z).$$ 
 So the functors $X\rightsquigarrow X_n$ and $Z\rightsquigarrow Z_1$ are adjoint. Note that we can identify $(X_n)_m$ with $X_{nm}$, a fact that we will use frequently.

\begin{ex} \label{ex: Xn}
	Note that the $\s$-polynomial ring $k\{y\}_\s$ over $k$ in the $m$ $\s$-variables $y_1,\ldots,y_m$ is, as a $k$-$\s^n$-algebra, the $\s^n$-polynomial ring over $(k,\s^n)$ in the $nm$ $\s^n$-variables $$(z_1,\ldots,z_{nm})=(y_1,\ldots,y_m,\s(y_1),\ldots,\s(y_m),\ldots,\s^{n-1}(y_1),\ldots,\s^{n-1}(y_m)).$$ 
	We may therefore write $k\{y\}_\s=k\{z\}_{\s^n}$.	
	Furthermore, if $F\subseteq k\{y\}$ $\s$-generates the $\s$-ideal $I$ of $k\{y\}_\s$, then $F,\s(F),\ldots,\s^{n-1}(F)\subseteq k\{y\}_\s=k\{z\}_{\s^n}$ $\s^n$-generates the $\s^n$-ideal $I$ of $k\{z\}_{\s^n}$. Thus, if $X$ is the \ks-variety defined by $F\subseteq k\{y\}_\s$, then $X_n$ is the $k$-$\s^n$-variety defined by $F,\s(F),\ldots,\s^{n-1}(F)\subseteq k\{z\}_{\s^n}$.
	
	For a concrete example, let us consider, as on page \pageref{page: X}, the \ks-variety $X=\{x\in\Gl_m|\ \s(x)=(x^T)^{-1}A\}$, where $A\in\Gl_m(k)$. As $X$ is defined by $\s(x)-(x^T)^{-1}A=0$, $X_2$ is defined by $\s(x)-(x^T)^{-1}A=0$ and $\s^2(x)-(\s(x)^T)^{-1}\s(A)=0$, where $x$ and $y=\s(x)$ are to be considered as $\s^2$-variables.
	Thus	
	\begin{align*}
	X_2&=\{(x,y)\in\Gl_m\times\Gl_m|\ y=(x^T)^{-1}A,\ \s^2(x)=(y^T)^{-1}\s(A) \} \\
	&\simeq\{x\in\Gl_m|\ \s^2(x)=x(A^T)^{-1}\s(A)\}
	\end{align*}
	as claimed on page \pageref{page: X}.	
%
%
%
%
%
\end{ex}

\begin{ex} \label{ex: varieties and Xn}
	If $X$ comes from an affine variety $\X$ over $k$, then the $\s^n$-variety $X_n$ corresponds to $\X\times{\hs \X}\times\ldots\times{^{\s^{n-1}}\!\X}$. In other words,
	$$([\s]_k\X)_n\simeq[\s^n]_k(\X\times\ldots\times{^{\s^{n-1}}\!\X}).$$
\end{ex}
\begin{proof}
	This follows from Example \ref{ex: Xn} but can also be seen in a more functorial fashion as follows: For every $k$-$\s^n$-algebra $S$ we have
\begin{align*}
	([\s]_k\X)_n(S) & =\X({_{\s^{n-1}}}S\times\ldots\times S)\simeq\Hom_k(k[\X],{_{\s^{n-1}}}S\times\ldots\times S) \\
	& \simeq\Hom_k(k[\X],{_{\s^{n-1}}}S)\times\ldots\times \Hom_k(k[\X],S) \\
	&\simeq\Hom_k({^{\s^{n-1}}}(k[\X]),S)\times\ldots\times \Hom_k(k[\X],S) \\
	& = \Hom_k({^{\s^{n-1}}}(k[\X])\otimes_k\ldots\otimes_k k[\X],S)\simeq (\X\times\ldots\times{^{\s^{n-1}}\!\X})(S).
	\end{align*}
\end{proof}

\begin{ex}
	If $Z$ is the $k$-$\s^n$-variety defined by $F\subseteq k\{y\}_{\s^n}$. Then $Z_1$ is the \ks-variety defined by $F\subseteq k\{y\}_{\s^n}\subseteq k\{y\}_\s$. Cf. Example \ref{ex: left adjoint}. Thus, concretely, to compute $Z_1$ from $Z$ we simply have to interpret the defining $\s^n$-equations of $Z$ as $\s$-equations.
\end{ex}	

 \begin{rem} \label{rem: points k vs kn}
 	Let $X$ be a \ks-variety. We have $X(\kn)\neq\emptyset$ if and only if $X_n(k)\neq\emptyset$.
 \end{rem}
 \begin{proof}
 	In fact, $X_n(k)=X(\kn)$.   
 \end{proof}

If $G$ is a $\s$-algebraic group and $X$ a $G$-torsor, then $G_n$ is a $\s^n$-algebraic group and $X_n$ is a $G_n$-torsor. Remark \ref{rem: points k vs kn} then shows that $X(\kn)\neq\emptyset$ if and only if $X_n$ is trivial.

%

The following three lemmas relate properties of the $\s$-algebraic group $G$ to properties of the $\s^n$\=/algebraic group $G_n$.

 \begin{lemma} \label{lem: connectedness preserved}
 	Let $G$ be a $\s$-algebraic group over $k$. If $G$ is connected, then $G_n$ is connected.	
 \end{lemma}
 \begin{proof}
 	Recall (e.g., from \cite[Lemma 6.7]{Wibmer:almostsimple}) that $G$ is connected if and only if $\pi_0(k\{G\})=k$. As $\pi_0(k\{G\})$ does not depend on $\s$ the claim follows.
 \end{proof}
 
 \begin{lemma} \label{lem: sdim and order preserved}
 	Let $G$ be a $\s$-algebraic group over $k$. Then $\s^n\text{-}\dim(G_n)=\sdim(G)\cdot n$ and $\ord(G_n)=\ord(G)$. 
 \end{lemma}
 \begin{proof}
 	We consider $G$ as a Zariski-dense, $\s$-closed subgroup of some algebraic group $\G$ via a $\s$\=/closed embedding $G \hookrightarrow\G$.  
    Algebraically this means that we have an inclusion of Hopf algebras $k[\G]\subseteq k\{G\}$ such that $k[\G]$ generates $k\{G\}$ as a \ks-algebra. In particular, if $B$ is a finite subset of $k[\G]$ such that $k[\G]=k[B]$, then $k\{G\}=k\{B\}$.
 	
 	The $\s$-closed embedding $G\hookrightarrow[\s]_k\G$ yields a $\s^n$-closed embedding $G_n\hookrightarrow([\s]_k\G)_n=[\s^n]_k(\G\times\ldots\times {^{\s^{n-1}}\!\G}$ (Example \ref{ex: varieties and Xn}). So we can consider $G_n$ as a $\s^n$-closed subgroup of the algebraic group $\G\times {\hs\G}\times\ldots\times{^{\s^{n-1}}\!\G}$. Algebraically, this corresponds to the inclusion of Hopf algebras $k[B,\ldots,\s^{n-1}(B)]\subseteq k\{G\}$.
 	For $i\geq 0$ let $d_i=\dim(k[B,\ldots,\s^{i}(B)])$. Then, for $i\gg 0$, we have $d_i=\sdim(G)(i+1)+e$ for some integer $e\geq 0$.
 	
 	Similarly, for $i\gg 0$, we have $d_{(i+1)n-1}=\s^n\text{-}\dim(G_n)(i+1)+e'$ for some integer $e'\geq 0$. Therefore
 	$$\sdim(G)((i+1)n)+e=\s^n\text{-}\dim(G_n)(i+1)+e'$$
 	and thus $\s^n$-$\dim(G_n)=\sdim(G)\cdot n$. Moreover, $\sdim(G)=0$ if and only if  $\s^n$-$\dim(G_n)=0$ and in this case $e=e'$, i.e., $\ord(G)=\ord(G_n)$.	 
 \end{proof}

\begin{lemma} \label{lemma: exactness preserved}
	If $1\to N\to G\to H\to 1$ is an exact sequence of $\s$-algebaic groups, then $1\to N_n\to G_n\to H_n\to 1$ is an exact sequence of $\s^n$-algebraic groups for every $n\geq 1$.	
\end{lemma}
\begin{proof}
The exactness of $1\to N\to G\to H\to 1$ can be reformulated in terms of the corresponding morphisms $k\{H\}\to k\{G\}\to k\{N\}$ of \ks-Hopf algebras. Exactness at $N$ means that $k\{G\}\to k\{N\}$ is surjective, exactness at $H$ means that $k\{H\}\to k\{G\}$ is injective and exactness at $G$ means that the kernel of $k\{G\}\to k\{N\}$ agrees with the ideal of $k\{G\}$ generated by the image of $\mathfrak{m}_H$ in $k\{G\}$, where $\mathfrak{m}_H$ is the kernel of the counit $k\{H\}\to k$. These properties do not depend on the power of $\s$ considered. Thus $1\to N_n\to G_n\to H_n\to 1$ is exact. 
\end{proof}	

\subsection{Exact sequences and $H^1$}

In this subsection we establish the induction principle. We also observe that \poweak{}$\Leftrightarrow$\pwweak{} and \po{}$\Leftrightarrow$\pw{}.

Let $G$ be a $\s$-algebraic group over $k$. A cohomology set $H^1(k,G)$ that classifies $G$-torsors was introduced in \cite{BachmayrWibmer:TorsorsForDifferenceAlgebraicGroups}. Far more general and elaborate approaches to cohomology in difference algebraic geometry can be found in \cite{Tomasic:AToposTheoreticViewOfDifferenceAlgebra} and \cite{ChalupnikKowalski:DifferenceSheavesAndTorsors}. For our purposes, the elementary approach from \cite{BachmayrWibmer:TorsorsForDifferenceAlgebraicGroups} is sufficient.

We begin by recalling the definition of the cohomology set $H^1(A/k,G)$, which, for a \ks-algebra $A$, classifies $G$\=/torsors $X$ with $X(A)\neq\emptyset$.
The morphisms of \ks-algebras
$$
\xymatrix{
	A \ar@<0.5ex>[r]  \ar@<-0.5ex>[r]  & A\otimes_k A  \ar@<1ex>[r] \ar[r] \ar@<-1ex>[r]  & A\otimes_k A\otimes_k A \\
}
$$
induce morphisms of groups
$$\xymatrix{
	G(A) \ar@<-0.5ex>_-{\de_2}[r] \ar@<0.5ex>^-{\de_1}[r] & G(A\otimes_k A) \ar@<1ex>^-{\partial_1}[r] \ar[r]|-{\partial_2}  \ar@<-1ex>_-{\partial_3}[r] & G(A\otimes_k A\otimes_k A),
}
$$
where $\delta_1$ is induced from $a\mapsto 1\otimes a$, $\delta_2$ is induced from $a\mapsto a\otimes 1$, $\partial_1$ is induced from $a\otimes b\mapsto 1\otimes a\otimes b$ and so on.

An element $\chi\in G(A\otimes_k A)$ is called a \emph{cocycle} if $\partial_2(\chi)=\partial_1(\chi)\partial_3(\chi)$. Two cocycles $\chi_1$ and $\chi_2$ are called equivalent if there exists a $g\in G(A)$ such that $\chi_1=\delta_1(g)\chi_2\delta_2(g)^{-1}$. The set of all equivalence classes of cocycles is denoted by $H^1_\s(A/k,G)$. It is a pointed set with the usual functorial properties. The distinguished element is the equivalence class of $1\in  G(A\otimes_k A)$.

We set $H^1_\s(k,G)=\lim\limits_{\longrightarrow}H^1_\s(A/k,G)$,
where the limit is taken over the set of isomorphism classes of finitely $\s$-generated \ks-algebras. (See \cite[Section 5]{BachmayrWibmer:TorsorsForDifferenceAlgebraicGroups} for more details.)
Let $\Sigma_\s(G)$ denote the pointed set of all isomorphism classes of $G$-torsors.
The distinguished element in $\Sigma_\s(G)$ is the class of the trivial torsor. According to \cite[Theorem 4.2]{BachmayrWibmer:TorsorsForDifferenceAlgebraicGroups} we have an isomorphism of pointed sets 
$$\Sigma_\s(G)\simeq H^1_\s(k,G).$$
Fix $n\geq 1$. If $X$ is a $G$-torsor, then $X_n$ is a $G_n$-torsor. As $X(k)\subseteq X(\kn)=X_n(k)$, we see that $X_n$ is trivial if $X$ is trivial. Thus we have a map of pointed sets
$$\Sigma_\s(G)\to \Sigma_{\s^n}(G_n).$$

\begin{ex}
Let $G=\{g\in\Gl_m|\ \s(g)=(g^T)^{-1}\}$. For every $A\in\Gl_m(k)$ the \ks-variety $X=\{x\in\Gl_m|\ \s(x)=(x^T)^{-1}A\}$ is a $G$-torsor under left multiplication. We have $G_2=\{g\in\Gl_m|\ \s^2(g)=g\}$ and 	$X_2=\{x\in\Gl_m|\ \s^2(x)=x(A^T)^{-1}\s(A)\}$. See Example \ref{ex: Xn}.

\end{ex}	
We also have a canonical map 
$$H^1_\s(k,G)\to H^1_{\s^n}(k,G_n),$$
which we will describe next. 
Let $X$ be a \ks-variety and $A$ a \ks-algebra. 
The morphism $A\to A^{[n]},\ a \mapsto (\s^{n-1}(a),\ldots,a)$ of \ks-algebras induces a map $X(A)\to X(A^{[n]})=X((A,\s^n)_1)=X_n(A)$.
Under the identifications $X(A)=\Hom_{k,\s}(k\{X\},A)$ and $X_n(A)=\Hom_{k,\s^n}(k\{X\},A)$, this map is simply interpreting a morphism $k\{X\}\to A$ of \ks-algebras as a morphism of $k$\=/$\s^n$\=/algebras.


If $A\to B$ is a morphism of \ks-algebras, then
$$
\xymatrix{
X(A) \ar[r] \ar[d] & X_n(A) \ar[d] \\
X(B) \ar[r] & X_n(B)	
}
$$
commutes. Similarly, if $X\to Y$ is a morphism of \ks-varieties, then
\begin{equation} \label{eqn: X_n commutes}
\xymatrix{
X(A) \ar[r] \ar[d] & X_n(A) \ar[d] \\
Y(A) \ar[r] & Y_n(A)	
}
\end{equation}
commutes.
 Moreover, if $G$ is a $\s$-algebraic group over $k$, then $G(A)\to G_n(A)$ is a morphism of groups. The commutative diagram
$$
\xymatrix{
	G(A) \ar@<0.5ex>[r]  \ar@<-0.5ex>[r] \ar[d] & G(A\otimes_k A)  \ar@<1ex>[r] \ar[r] \ar@<-1ex>[r] \ar[d] & G(A\otimes_k A\otimes_k A) \ar[d] \\
		G_n(A) \ar@<0.5ex>[r] \ar@<-0.5ex>[r] & G_n(A\otimes_k A)  \ar@<1ex>[r] \ar[r] \ar@<-1ex>[r] & G_n(A\otimes_k A\otimes_k A)
}
$$
shows that $G(A\otimes_k A)\to G_n(A\otimes_k A)$ maps cocycles to cocycles and equivalent cocycles to equivalent cocycles. We therefore have a well-defined map $H^1_\s(A/k,G)\to H^1_{\s^n}(A/k,G_n)$, which we may compose with the inclusion $H^1_{\s^n}(A/k,G_n)\to H^1_{\s^n}(k,G_n)$.
If $A\to B$ is a morphism of \ks-algebras, then 
$$
\xymatrix{
	H^1_{\s}(A/k,G) \ar[dd] \ar[rd] &  \\
	& H^1_{\s^n}(k,G_n) \\
		H^1_{\s}(B/k,G) \ar[ru] &
}
$$
commutes. Thus we have an induced map of pointed sets $H^1_\s(k,G)\to H^1_{\s^n}(k,G_n)$. The construction of this map also applies with $(k,\s)$ replaced by $(k,\s^d)$ (for some $d\geq 1$) and so we have maps $H^1_{\s^d}(k,G_d)\to H^1_{(\s^d)^n}(k,(G_{d})_n)=H^1_{\s^{dn}}(k,G_{dn})$. If $X$ is a $k$\=/$\s^n$\=/variety, $(A,\tau)$ a $k$-$\s^n$-algebra and $d,d'\geq 1$, then the composition $X(A,\tau)\to X_{d}(A,\tau^{d})\to (X_d)_{d'}(A,(\tau^d)^{d'})=X_{dd'}(A,\tau^{dd'})$ agrees with $X(A,\tau)\to X_{dd'}(A,\tau^{dd'})$ because all maps are just reinterpreting the same morphism from $X(A,\tau)=\Hom_{k,\s^n}(k\{X\},A)$.
From this it follows that for integers $n,n',n''\geq 1$ with $n|n'$ and $n'|n''$, the composition $H^1_{\s^n}(k,G_n)\to H^1_{\s^{n'}}(k,G_{n'})\to H^1_{\s^{n''}}(k,G_{n''})$ agress with $H^1_{\s^n}(k,G_n)\to H^1_{\s^{n''}}(k,G_{n''})$. The maps $H^1_{\s^n}(k,G_n)\to H^1_{\s^{n'}}(k,G_{n'})$ for $n|n'$ thus form a directed system over the directed set of all integers greater or equal to one, partially ordered by $n\leq n'$ if and only if $n|n'$. We can therefore consider the direct limit $\lim\limits_{n\to \infty} H^1_{\s^n}(k,G_n)$.

\begin{ex} \label{ex: direct limit}
	Let $\G$ be an algebraic group over $k$ and $G=[\s]_k\G$. By \cite[Prop. 3.5]{BachmayrWibmer:TorsorsForDifferenceAlgebraicGroups} we have $H^1_\s(k,G)=H^1(k,\G)$. Furthermore $G_n=[\s^n]_k(\G\times\ldots\times{^{\s^{n-1}}\!\G})$ (Example \ref{ex: varieties and Xn}) and so $H^1_{\s^n}(k,G_n)=H^1(k,\G\times\ldots\times{^{\s^{n-1}}\!\G})=H^1(k,\G)\times\ldots\times H^1(k,{^{\s^{n-1}}\!\G})$. To understand the maps 
	$H^1_{\s^n}(k,G_n)\to H^1_{\s^{n'}}(k,G_{n'})$ for $n|n'$, we first need an explicit description of the map $X_n(A,\tau)\to (X_n)_d(A,\tau^d)$ for $d\geq 1$, $X$ a $k$-$\s^n$-variety and $(A,\tau)$ a $k$-$\s^n$-algebra. Now $X_n(A,\tau)\to (X_n)_d(A,\tau^d)=X_n((A,\tau)^{[d]})$
	is obtained by applying the functor $X_n$ to the morphism $(A,\tau)\to (A,\tau)^{[d]}$ of $k$-$\s^n$-algebras.
	As $X_n(A,\tau)=X((A,\tau)_1)$, this means that we need to apply the functor $X$ to the morphism $(A,\tau)_1\to ((A,\tau)^{[d]})_1$ of \ks-algebras. As $$(A,\tau)\to (A,\tau)^{[d]}={_{(\s^n)^{d-1}}}A\times\ldots\times{_{\s^n}}A\times A,\quad a\mapsto (\tau^{d-1}(a),\ldots,\tau(a),a),$$ 
	we see that 
	$$(A,\tau)_1={_{\s^{n-1}}}A\times\ldots\times A\longrightarrow ((A,\tau)^{[d]})_1={_{\s^n}}({_{(\s^{n(d-1)}}}A\times\ldots\times A)\times\ldots\times {_{\s^{n(d-1)}}}A\times\ldots\times A $$
	is given by $$(a_1,\ldots,a_n)\longmapsto(\tau^{d-1}(a_1),\ldots,a_1,\ldots,\tau^{d-1}(a_n)\ldots,a_n).$$ Thus, after reordering the factors as $_{\s^{nd-1}}A\times \ldots\times {_{\s}A}\times A$, the map is given by $$(a_1,\ldots,a_n)\mapsto (\tau^{d-1}(a_1),\ldots,\tau^{d-1}(a_n),\ldots,a_1,\ldots,a_n).$$ Applying $X=[\s]_k\G$ to this (as in Example \ref{ex: varieties and Xn}), we find the map	$$\G(A)\times\ldots\times{^{\s^{n-1}}\!\G}(A)\longrightarrow \G(A)\times\ldots\times{^{\s^{nd-1}}\!\G}(A),$$ \begin{equation}\label{eq: for H}(g_1,\ldots,g_n)\longmapsto (g_1,\ldots,g_n,\ldots,\tau^{d-1}(g_1),\ldots,\tau^{d-1}(g_n)),
		\end{equation}
	where $\tau^j\colon {\hsi}\G(A)\to {^{\s^{i+n}}\!\G}(A)$ (for $0\leq i\leq n-1$ and $1\leq j\leq d-1$) is given by applying $\tau$ to the coordinates. Without coordinates, if $g\colon {\hsi}(k[\G])\to A$, then $\tau^j(g)\colon^{\s^{n}}\!({^{\s^{i}}\!(k[G])})\to A,\ f\otimes\lambda\mapsto \tau^j(g(f))\lambda$.
	
	If $k\to k'$ is a morphism of fields, we have a map $H^1(k,\G)\to H^1(k',\G_{k'})$, which on isomorphism classes of torsors corresponds to base change. Applying this with $\s^n\colon k\to k$, we obtain maps $\s^n\colon H^1(k,{\hsi\G})\to H^1(k,{^{\s^{i+n}}\!\G})$. 
	
	Base change via $\s^n\colon k\to k$ takes $g\colon {\hsi(k[\G])}\to A$ and returns $\hsn g\colon{^{\s^{i+n}}\!{(k[\G])}}\to{\hsn A}$.
	The map $\tau^j(g)\colon {^{\s^{i+n}}\!{(k[\G])}}\to A$ is then obtained from $\hsn g$ by composing with the morphism ${\hsn A}\to A,\ a\otimes\lambda\mapsto \tau(a)\lambda$ of $k$-algebras. Since a morphism of $k$-algebras does not matter when we pass to cohomology sets, we see that $\tau^j\colon {\hsi}\G(A)\to {^{\s^{i+n}}\!\G}(A)$ induces $\s^n\colon H^1(k,{\hsi\G})\to H^1(k,{^{\s^{i+n}}\!\G})$.
	
	Based on (\ref{eq: for H}), we thus find that for $n|n'$, the map	
	$$H^1_{\s^n}(k,G_n)=H^1(k,\G)\times\ldots\times H^1(k,{^{\s^{n-1}}\!\G})\longrightarrow H^1_{\s^{n'}}(k,G_{n'})=H^1(k,\G)\times\ldots\times H^1(k,{^{\s^{n'-1}}\!\G})$$
	is given by $$(a_1,\ldots,a_n)\longmapsto (a_1,\ldots,a_n,\s^n(a_1),\ldots,\s^n(a_n),\ldots,\s^{n'-n}(a_1),\ldots,\s^{n'-n}(a_n)).$$
	Therefore the limit $\lim\limits_{n\to \infty} H^1_{\s^n}(k,G_n)$ can be identified with the set of all ``periodic'' elements of $\prod_{n\geq 1}H^1(k,{\hsn\G})$, i.e., all elements of the form
	$(a_1,\ldots,a_n,\s^n(a_1),\ldots,\s^n(a_n),\s^{2n}(a_1),\ldots,\s^{2n}(a_n),\ldots)$ for some $n\geq 1$.

\end{ex}

 We will need two more commutative diagrams.

\begin{lemma} \label{lemma: commutes on H1}
	Let $G$ be a $\s$-algebraic group over $k$. Then
	$$
	\xymatrix{
	\Sigma_\s(G) \ar[r] \ar[d] & \Sigma_{\s^n}(G_n) \ar[d] \\
	H^1_\s(k,G) \ar[r] & H^1_{\s^n}(k, G_n)	
	}
	$$
	commutes.	
\end{lemma}
\begin{proof}
	Let $X$ be a $G$-torsor and $A$ a finitely $\s$-generated \ks-algebra such that $X(A)\neq\emptyset$. Choose $x\in X(A)$. The corresponding cocycle $\chi\in G(A\otimes_k A)$ is determined by $f_1(x)=\chi.f_2(x)$, where $f_1,f_2\colon X(A)\to X(A\otimes_k A)$. (See the proof of \cite[Theorem 4.1]{BachmayrWibmer:TorsorsForDifferenceAlgebraicGroups}.)
	Under $G(A\otimes_k A)\to G_n(A\otimes_k A)$ the cocycle $\chi$ maps to the unique cocycle $\widetilde{\chi}\in G_n(A\otimes_k A)$ that satisfies $f_1(\widetilde{x})=\widetilde{\chi}.f_2(\widetilde{x})$, where $\widetilde{x}$ is the image of $x$ under $X(A)\to X_n(A)$.
	
	On the other hand, for the $G_n$-torsor $X_n$ we have $\widetilde{x}\in X_n(A)$ and so the corresponding cocycle in $G_n(A\otimes_k A)$ is exactly $\widetilde{\chi}$. Now the claim follows by taking equivalence classes and passing to the limit.
\end{proof}

\begin{ex}
	The map $H^1_\s(k,G)\to H^1_{\s^n}(k,G_n)$ is in general neither injective nor surjective, i.e., there may exist a $G_n$-torsor that does not come from a $G$-torsor and there may be two non-isomorphic $G$-torsors that become isomorphic as $G_n$-torsors. For a non-surjective example see Example \ref{ex: direct limit}. For a non-injective example consider the $\s$-algebraic group $G=\{g\in\Gm|\ g^2=1,\ \s(g)=g\}$ and the $G$-torsor $X=\{x\in \Gm|\ x^2=1,\ \s(x)=-x\}$ (under left multiplication). Then $X$ is non-trivial but $X_2$ is trivial because $X_2(k)=X(k^{[2]})\neq\emptyset$.
\end{ex}

\begin{lemma} \label{lemma: commutativity of X_n and H}
	Let $G\to H$ be a morphism of $\s$-algebraic groups over $k$. Then
	$$ 
	\xymatrix{
	 H^1_\s(k,G) \ar[r] \ar[d] & H^1_\s(k,H) \ar[d] \\
	 H^1_{\s^n}(k,G_n) \ar[r] & H^1_{\s^n}(k, H_n)
	}
		$$
		commutes.
\end{lemma}
\begin{proof}
	This follows from the commutativity of (\ref{eqn: X_n commutes}).
\end{proof}

\begin{defi} \label{defi: Gtrivial}
	Let $G$ be a $\s$-algebraic group over $k$. To simplify our statements we call $k$ 
	\begin{itemize}
		\item \emph{$G$-trivial} if for every $G$-torsor $X$, there exists an integer $n\geq 1$ such that $X(\kn)\neq\emptyset$. If in addition the integer $n$ can be chosen uniformly for all $X$, then we call $k$ \emph{uniformly $G$-trivial}.
		\item (uniformly) \emph{strongly $G$-trivial} if $(k,\s^n)$ is (uniformly) $G_n$-trivial for all $n\geq 1$.
	\end{itemize}
\end{defi}

Thus $k$ is uniformly strongly $G$-trivial if for every $n\geq 1$ and every $\s$-algebraic group $G$ over $k$, there exists a $d\geq 1$ such that for every $G_n$-torsor we have $X((k,\s^n)^{[d]})=X((k,(\s^n)^d)_1)=X_d(k,(\s^n)^d)\neq\emptyset$. The notion of being (uniformly) strongly $G$-trivial is important for making the induction principle work.

 Property property \poweak{} means that $k$ is $G$-trivial for every $\s$-algebraic group $G$ over $k$, while \po{} means that $k$ is uniformly $G$-trivial for every $\s$-algebraic group $G$ over $k$. 
The following lemma is needed for \poweak{}$\Rightarrow$\pwweak.
 
\begin{lemma}\label{lemma: trivial implies strongly trivial}
	If $k$ is $G$-trivial for every $\s$-algebraic group $G$ over $k$, then $k$ is strongly $G$-trivial for every $\s$-algebraic group $G$ over $k$.
\end{lemma}
\begin{proof}
	Let $G$ be a $\s$-algebraic group over $k$ and let $X$ be a $G_n$-torsor for some $n\geq 1$. Then $X_1$ is a $(G_n)_1$-torsor and by assumption there exists a $d\geq 1$ such that $X_1(k^{[d]})\neq\emptyset$. As $X_1(k^{[d]})=X(k^{[d]},\s^n)$ there exists a morphism $k\{X\}\to k^{[d]}$ of $k$-$\s^n$-algebras. The projection $k^{[d]}\to k$ to the last factor is a morphism of $k$-$\s^d$-algebras. Thus the composition $k\{X\}\to k^{[d]}\to k$ is a morphism of $k$-$\s^{nd}$-algebras. Therefore
	$X_d(k,(\s^n)^d)=X((k,(\s^n)^d)_1)=X((k,\s^n)^{[d]})\neq\emptyset$
	as desired.
\end{proof}	
The following lemma will allow us to pass to higher powers of $\s$ in the proof of \ptweak{}$\Rightarrow$\poweak{}.

\begin{lemma}\label{lem Gn trivial implies G trivial}
Let $G$ be a $\s$-algebraic group over $k$. If $(k,\s^n)$ is (strongly) $G_n$\=/trivial, then $k$ is (strongly) $G$-trivial.  
\end{lemma}

\begin{proof}
Assume that $(k,\s^n)$ is $G_n$-trivial and let $X$ be a $G$-torsor. Then $X_n$ is a $G_n$-torsor, so there exists an integer $d\geq 1$ with $(X_n)_d(k)\neq \emptyset$ (Remark \ref{rem: points k vs kn}). Define $m=nd$. As $(X_n)_d=X_m$ we have $X_m(k)\neq \emptyset$ and thus $X(k^{[m]})\neq \emptyset$ (Remark \ref{rem: points k vs kn}). We conclude that $k$ is $G$-trivial. 

If $(k,\s^n)$ is strongly $G_n$-trivial, then $(k,(\s^n)^d)$ is $(G_n)_d$-trivial for every integer $d\geq 1$. In other words, $(k,(\s^d)^n)$ is $(G_d)_n$-trivial for every integer $d\geq 1$.  By the statement proven in the first paragraph, $(k,\s^d)$ is $G_d$-trivial for every $d\geq 1$. Hence $k$ is strongly $G$-trivial.
\end{proof}

%

\begin{lemma} \label{lemma: Gtriviality and H1}
Let $G$ be a $\s$-algebraic group over $k$, $X$ a $G_n$-torsor and $d\geq 1$. Then $X((k,\s^n)^{[d]})\neq\emptyset$ if and only if the image of $X$ in $H^1_{\s^n}(k,G_n)$ maps to the distinguished element under  $H^1_{\s^{n}}(k,G)\to H^1_{\s^{nd}}(k,G_{nd})$.
\end{lemma}
\begin{proof}
	We have $X((k,\s^n)^{[d]})\neq\emptyset$ if and only if $X_d(k,(\s^n)^d)\neq\emptyset$ by Remark \ref{rem: points k vs kn}. So the claim follows from Lemma \ref{lemma: commutes on H1}. 
\end{proof}

At this point the equivalence of \poweak{} and \pwweak{} is self-evident:

\begin{prop}\label{prop: P2 implies P1}
Properties \poweak{} and \pwweak{} are equivalent. Similarly, properties \po{} and \pw{} are equivalent.
\end{prop}
\begin{proof}
	The equivalence \po{}$\Leftrightarrow$\pw{} follows from Lemma \ref{lemma: Gtriviality and H1}. According to that lemma, property \pwweak{} means that $k$ is strongly $G$-trivial for every $\s$-algebraic group $G$ over $k$. Thus \poweak{}$\Leftrightarrow$\pwweak{} follows from Lemma \ref{lemma: trivial implies strongly trivial}.
\end{proof}

We are now prepared to establish the induction principle. Essentially this is an application of the exact cohomology sequence. Recall that a sequence $M'\to M\to M''$ of pointed sets is exact at $M$ if the kernel of $M\to M''$ (i.e., the inverse image of the distinguished element of $M''$) agrees with the image of $M'\to M$.

\begin{prop}[Induction principle] \label{prop:InductionPrinciple}
Let $1\to N\to G\to H\to 1$ be an exact sequence of $\s$-algebraic groups over $k$.
\begin{enumerate}
    \item If $k$ is (uniformly) $H$-trivial and (uniformly) strongly $N$-trivial, then $k$ is (uniformly) $G$\=/trivial. 
    \item 	If $k$ is (uniformly) strongly $N$-trivial and (uniformly) strongly $H$-trivial, then $k$ is (uniformly) strongly $G$-trivial.
\end{enumerate}
\end{prop}
\begin{proof}
	According to Lemma \ref{lemma: commutativity of X_n and H} we have a commutative diagram for all integers $n,m,l \geq 1$:
	$$
	\xymatrix{
				   & H^1_{\s^n}(k,G_n) \ar[r] \ar^\alpha[d] &  H^1_{\s^n}(k,H_n) \ar[d]   \\
					 H^1_{\s^{nm}}(k,N_{nm})\ar[r] \ar[d] & H^1_{\s^{nm}}(k,G_{nm}) \ar[r] \ar^\beta[d] &  H^1_{\s^{nm}}(k,H_{nm}) \\
					H^1_{\s^{nml}}(k,N_{nml}) \ar[r] & 	H^1_{\s^{nml}}(k,G_{nml}) &
								}
						$$
We prove (i) and (ii) simultaneously and set $n=1$ in case (i) and $n\geq 1$ arbitrary in case (ii). 
Because $1\to N_{nm}\to G_{nm}\to H_{nm}\to 1$ is exact by Lemma \ref{lemma: exactness preserved}, the exact cohomology sequence (\cite[Prop. 5.6]{BachmayrWibmer:TorsorsForDifferenceAlgebraicGroups}) shows that the middle row in the above diagram is exact.

Start with an element in $\gamma \in H^1_{\s^n}(k,G_n)$. As $k$ is (strongly) $H$-trivial, there exists an $m\geq 1$ such that its image in $H^1_{\s^n}(k,H_n)$ lies in the kernel of the vertical map $H^1_{\s^{n}}(k,H_{n})\to H^1_{\s^{nm}}(k,H_{nm})$. Moreover, $m$ can be chosen uniformly for all elements in $H^1_{\s^n}(k,G_n)$ if $k$ is uniformly (strongly) $H$-trivial. By the exactness of the middle row, $\alpha(\gamma)$ is contained in the image of $H^1_{\s^{nm}}(k,N_{nm})\to H^1_{\s^{nm}}(k,G_{nm})$. Since $k$ is strongly $N$-trivial, there exists an $l\geq 1$ such that its preimage in $H^1_{\s^{nm}}(k,N_{nm})$ is mapped to the trivial element in $H^1_{\s^{nml}}(k,N_{nml})$ and we conclude that $\beta(\alpha(\gamma))$ is trivial. Moreover, if $N$ is uniformly strongly $N$-trivial, then $l$ does not depend on the choice of $\gamma$. All claims now follow with Lemma \ref{lemma: Gtriviality and H1}.		
\end{proof}

By induction we obtain from Proposition \ref{prop:InductionPrinciple} (ii):

\begin{cor} \label{cor: strong induction for tower} 
	If a $\s$-algebraic group $G$ has a subnormal series
	$$G=G_0\supseteq G_1\supseteq\ldots\supseteq G_n=1$$ such that $k$ is (uniformly) strongly $G_i/G_{i+1}$-trivial for $i=0,\ldots,n-1$, then $k$ is (uniformly) strongly $G$-trivial. \qed 
\end{cor}

\section{First observations regarding property \ptweak{}}
\label{sec: first observations}

In this section we discuss some equivalent reformulations of property \ptweak{} (b) (and \pt{} (b)). We also show that \poweak{}$\Rightarrow$\ptweak{} and similarly \po{}$\Rightarrow$\pt{}. 

\medskip

{\bf Throughout Section \ref{sec: first observations} we denote with $k$ a $\s$-field (of arbitrary characteristic).}

\medskip
%
%

%

\begin{lemma}\label{lemma: on pt (b)}
	For given integers $m,d\geq 1$ and an algebraic group $\G\leq \Gl_m$ defined over $\mathbb{Q}$  the following are equivalent for every $A\in \G(k)$:
	\begin{enumerate}
		\item There exists an integer $n\geq 1$ and an element $Y\in \G(\kn)$ such that $\s^d(Y)=AY$.
		\item There exists an integer $l\geq 1$ and an element $Y\in \G(k)$ such that $\s^{dl}(Y)=\s^{(l-1)d}(A)\ldots\s^d(A)AY$.
	\end{enumerate}
	Moreover, $n$ can be chosen uniformly for all $A\in \G(k)$ if and only if $l$ can be chosen uniformly for all $A \in \G(k)$. 
\end{lemma}

\begin{proof}
	Let us first show that (i) implies (ii).
	Finding a solution in $\kn$ for some set of difference equations is equivalent to finding a morphism of \ks-algebras $R\to \kn$ for some appropriate \ks-algebra $R$. If $n|\tilde n$, say $\tilde n=nl$, then we have a morphism of \ks-algebras $$\kn\to k^{[\tilde n]},\ (a_1,\ldots,a_n)\mapsto (\s^{n(l-1)}(a_1),\ldots,\s^{n(l-1)}(a_n),\ldots,\s^n(a_1),\ldots,\s^n(a_n),a_1,\ldots,a_n).$$ We can therefore assume that $d|n$, say $n=dl$. By assumption, there exists a $Y\in \G(k^{[n]})$ with $\s^d(Y)=AY$, and thus $\s^{dl}(Y)=\s^{(l-1)d}(A)\ldots\s^d(A)AY$.
	The projection $\kn\to k$ onto the last factor is a morphism of $k$-$\s^n$-algebras. Therefore the equation 
	$\s^{dl}(Y)=\s^{(l-1)d}(A)\ldots\s^d(A)AY$ has a solution in $\G(k)$. 
	
	Next we will show that (ii) implies (i). Let $Z\in\G(k)$ be a solution to $\s^{dl}(Y)=\s^{(l-1)d}(A)\ldots\s^d(A)AY$. Then
	$$Z'=(\s^{d-1}(\s^{(l-2)d}(A)\ldots AZ),\ldots,\s^{(l-2)d}(A)\ldots AZ,\ldots\ldots,\s^{d-1}(AZ),\ldots,AZ,\s^{d-1}(Z),\ldots,Z)\in\G(\kn)$$
	is a solution to $\s^d(Y)=AY$, where $n=dl$.
	
	Note that if $n$ does not depend on $A$, then $l$ does not depend on $A$ and vice versa, so the second statement also holds. 
\end{proof}

\begin{lemma} \label{lemma: on pt (b) was remark}
Condition \ptweak{} (b) is equivalent to each of the following statements:
	
	\begin{enumerate}
		\item[\ptweak{} \rm{(b')}] For every $m\geq 1$ and for every $A\in\Gl_m(k)$ there exists an $l\geq 1$ such that  the equation $\s^{l}(Y)=\s^{(l-1)}(A)\cdots\s(A)AY$ has a solution in $\Gl_m(k)$. 	
		\item[\ptweak{}  \rm{(b$_d$)}]  For every $m\geq 1$ and every $d\geq 1$ and for every $A\in\Gl_m(k)$  there exists an integer $n\geq 1$ such that the equation $\s^d(Y)=AY$ has a solution in $\Gl_m(\kn)$. 		
		\item[\ptweak{} \rm{(b$_d$')}] For every $m\geq 1$ and $d\geq 1$ and for every $A\in\Gl_m(k)$  there exists an $l\geq 1$ such that the equation $\s^{dl}(Y)=\s^{(l-1)d}(A)\ldots\s^d(A)AY$ has a solution in $\Gl_m(k)$. 
	\end{enumerate}	
    
	Similarly, Condition \pt{} \rm{(b)} is equivalent to each of the following statements:
	
	\begin{enumerate}
		\item[ \pt{} (b')] For every $m\geq 1$ there exists an $l\geq 1$ such that for every $A\in\Gl_m(k)$ the equation $\s^{l}(Y)=\s^{(l-1)}(A)\cdots\s(A)AY$ has a solution in $\Gl_m(k)$. 	
		\item[ \pt{}  (b$_d$)]  For every $m\geq 1$ and every $d\geq 1$, there exists an $n\geq 1$ such that for every $A\in\Gl_m(k)$ the equation $\s^d(Y)=AY$ has a solution in $\Gl_m(\kn)$. 
		
		\item[ \pt{} (b$_d$')] For every $m\geq 1$ and $d\geq 1$ there exists an $l\geq 1$ such that for every $A\in\Gl_m(k)$ the equation $\s^{dl}(Y)=\s^{(l-1)d}(A)\ldots\s^d(A)AY$ has a solution in $\Gl_m(k)$. 
	\end{enumerate}	
	
\end{lemma}
\begin{proof} 
	The equivalence of (b) and  (b') follows (in both cases \ptweak{} and \pt{}) from Lemma \ref{lemma: on pt (b)} with $\G=\Gl_m$ and $d=1$. Similarly, the equivalence of (b$_d$) and  (b$_d$') follows from Lemma \ref{lemma: on pt (b)} with $\G=\Gl_m$ in both cases. So it remains to show that (b') implies  (b$_d$') in both cases. We start with the \pt{} case and assume that \pt{} (b') holds and fix integers $m,d\geq 1$. Then by \pt{} (b'), there exists an integer $l_d$ such that for every $B\in \Gl_{md}(k)$, there exists a $Z\in \Gl_{md}(k)$ with $\s^{l_d}(Z)=\s^{(l_d-1)}(B)\cdots \s(B) B Z$. By possibly replacing $l_d$ with $dl_d$, we may assume that $d$ divides $l_d$ and define $l=\frac{l_d}{d}$. We claim that \pt{} (b$_d$') holds for this $l$. Fix a matrix $A \in \Gl_m(k)$. It remains to show that there exists a matrix $Y\in \Gl_m(k)$ with $\s^{dl}(Y)=\hat AY$ where $\hat A=\s^{(l-1)d}(A)\ldots\s^d(A)A$. By assumption, there exists a $Z\in \Gl_{md}(k)$ with $\s^{l_d}(Z)=\hat B Z$, where we define $B$ as block matrix 
	$$ B=\begin{pmatrix} 
		0 & I_m & 0 &  & \\
		\vdots & & \ddots  &  & 0 \\
		0 & &  & & I_m  \\
		 A & 0& \dots  & & 0 \\
	\end{pmatrix} \in \Gl_{md}(k)$$
	and set $\hat B=\s^{(l_d-1)}(B)\cdots \s(B) B$. Recall that $d$ divides $l_d$, so the product $\hat B$ with $l_d$ factors is the block diagonal matrix:
	$$\hat B= \begin{pmatrix}
		\hat A & & & \\
		& \s(\hat A) & & \\
		& & \ddots & \\
		& & & \s^{d-1}(\hat A)
	\end{pmatrix}.$$
	Now the $md$ columns of $Z\in \Gl_{md}(k)$ are $k$-linearly independent, so we can choose $m$ columns $z_1,\dots,z_m$ of $Z$ such that their projections $y_1,\dots,y_m \in k^m$ to the first $m$ coordinates are still $k$\=/linearly independent. Define $Y\in \Gl_m(k)$ as the matrix with columns $y_1,\dots,y_m$. By construction, $\s^{dl}(z_i)=\s^{l_d}(z_i)=\hat B z_i$, so $\s^{dl}(y_i)=\hat A y_i$ for all $i$ and thus $\s^{dl}(Y)=\hat A Y$ as claimed. The proof in the \ptweak{} case is analogous, but here we start with a fixed matrix $A\in \Gl_m(k)$, define $B\in \Gl_{md}(k)$ as above and obtain an integer $l_d$ from \ptweak{}(b') (depending on $B$ and thus on $A$) and a matrix $Z\in \Gl_{md}(k)$ with $\s^{l_d}(Z)=\hat B Z$ (with $\hat B$ defined as above). Again, we may assume that $d$ divides $l_d$ and set $l=l_d$ (which now depends on $A$)  and the same computations as  above yield a matrix $Y\in \Gl_m(k)$ with $\s^{dl}(Y)=\hat A Y$ as claimed in \ptweak{}(b$_d$').
\end{proof}

\begin{lemma} \label{lemma: inversive}
	A $\s$-field satisfying \ptweak{} \rm{(c)} is inversive. 
\end{lemma}
\begin{proof}
	For $a\in k$, consider the multiplicative equation $\s(y)=a$. By \ptweak{}(c), there exists an $n\geq 1$ and a solution $(a_1,\dots,a_n)\in\kn$, i.e., $(\s^n(a_n), a_1,\dots,a_{n-1})=(\s^{n-1}(a),\dots,\s(a),a)$. Thus $\s^{n-1}(\s(a_n)-a)=0$ and so $\s(a_n)=a$, hence $\s\colon k \to k$ is surjective. 
\end{proof}

%


The following lemma shows that property \ptweak{} is preserved when passing to higher powers of~$\s$.

\begin{lemma} \label{lem: pt for powers of s}
	If the difference field $(k,\s)$ satisfies \ptweak{}, then also the difference field $(k,\s^d)$ satisfies \ptweak{} for every $d\geq 1$. Similarly, if  $(k,\s)$ satisfies \pt{}, then $(k,\s^d)$ satisfies \pt{} for every $d\geq 1$. 
\end{lemma}
\begin{proof}
	No need to worry about \ptweak{} (a) and \pt{} (a). 	
	For property (b), we use Lemma \ref{lemma: on pt (b) was remark} to conclude that it suffices to show that (b') holds for the difference field $(k,\s^d)$ in both cases. But that condition equals  (b$_d$') for the difference field $(k,\s)$ which does hold by Lemma \ref{lemma: on pt (b) was remark} in both cases \ptweak{} and \pt{}. 
	
	For (c) let $\alpha_0,\ldots,\alpha_m\in\mathbb{Z}$ with $\alpha\neq 0$ and for $a\in k^\times$ consider the multiplicative $\s^d$-equation
	\begin{equation} \label{eqn: multiplicative system}
	y^{\alpha_0}\s^d(y)^{\alpha_1}\ldots\s^{dm}(y)^{\alpha_m}=a.
	\end{equation}
%
 If \ptweak{}(c) holds, there exists an $n\geq 1$ such that (\ref{eqn: multiplicative system}), when considered as a $\s$-equations has a solution in $\kn$.
	So (\ref{eqn: multiplicative system}), when considered as a system of $\s^d$-equations, has a solution in the $k$-$\s^d$-algebra $(\kn,\s^d)$.
	As in the proof of Lemma \ref{lemma: on pt (b)}, we can assume that $d|n$, say $n=dl$. Then we have a morphism
	$$(\kn,\s^d)\to (k,\s^d)^{[l]},\ (a_1,\ldots,a_n)\mapsto(a_d,a_{2d},\ldots,a_n)$$
	of $k$-$\s^d$-algebras. Thus (\ref{eqn: multiplicative system}), when considered as a system of $\s^d$-equations, has a solution in $(k,\s^d)^{[l]}$. We conclude that $(k,\s^d)$ satisfies \ptweak{} (c). Moreover, if $(k,\s)$ satisfies the stronger condition \pt{} (c), then $n$ does not depend on the choice of $a$ and thus $l$ does not depend on it either and thus $(k,\s^d)$ satisfies the stronger condition \pt{} (c).	
\end{proof}

As indicated in the introduction, to show that \poweak{} implies \ptweak{} it suffices to consider \poweak{} for a few particular $\s$-algebraic groups, namely, algebraic groups, $\Gl_m^\s$ and one-relator subgroups of $\Gm$.

\begin{prop} \label{prop: po implies pt}
	If a $\s$-field $k$ satisfies \poweak{}, then it satisfies \ptweak{}. Similarly, \po{} implies \pt{}.
\end{prop}
\begin{proof}
	To prove (a) in both cases, we need to show that \poweak{} implies that $k$ is algebraically closed. As every finite Galois extension of $k$ defines a torsor for a finite (constant) algebraic group over $k$, it is clear that a field with the property that every torsor for every algebraic group over $k$ is trivial, has to be separably closed.
	Moreover, if $k$ has positive characteristic $p$, then $\G=\{g\in\Gm|\ g^p=1\}$ is a (non-reduced) algebraic group and for every $a\in k^\times$ we have a $\G$-torsor $\X=\{x\in\Gm|\ x^p=a\}$. Thus if all these torsors are trivial, $k$ must be algebraically closed.
	
	So let $\G$ be an algebraic group over $k$ and let $\X$ be a $\G$-torsor. We have to show that $\X(k)\neq\emptyset$.
	Set $G=[\s]_k\G$ and $X=[\s]_k\X$. Then $X$ is a $G$-torsor. So, by assumption, there exists $n\geq 1$ such that $X(\kn)=\X(\kn)\neq\emptyset$. In other words, there exists a morphism of $k$-algebras $k[\X]\to\kn$. Composing with the last projection $\kn\to k$ yields a morphism of $k$-algebras $k[\X]\to k$. So $\X(k)\neq\emptyset$ and $k$ is algebraically closed.
	
	To prove (b), consider the $\s$-algebraic group $G=\Gl_m^\s=\{g\in\Gl_m|\ \s(g)=g\}$.
	For $A\in\Gl_m(k)$ consider the $\s$-variety $X$ given by
	$X=\{x\in\Gl_m|\ \s(x)=Ax\}$. It is straight forward to check that $X$ is a $G$-torsor under the action $(g,x)\mapsto xg^{-1}$. We conclude that \poweak{} implies \ptweak{}(b) and \po{} implies \pt{}(b). 		
	
	To prove (c) consider the $\s$-algebraic group
	$G=\{g\in \Gm |\ g^{\alpha_0}\s(g)^{\alpha_1}\ldots\s^m(g)^{\alpha_m}=1\}.$
	Then the $\s$-variety
	$X=\{x\in \Gm |\ x^{\alpha_0}\s(x)^{\alpha_1}\ldots\s^m(x)^{\alpha_m}=a\}$
	is a $G$-torsor. We conclude that \poweak{} implies \ptweak{}(c) and \po{} implies \pt{}(c). 		
\end{proof}

\section{The proof of \ptweak{}$\Rightarrow$\poweak{}}
\label{sec: The proof}

By now we have seen that \poweak{} and \pwweak{} are equivalent (Proposition \ref{prop: P2 implies P1}) and that \poweak{} implies \ptweak{} (Propsition \ref{prop: po implies pt}) and similarly for their uniform variants.
Thus to complete the proof of our main result, it suffices to prove \ptweak{}$\Rightarrow$\poweak{} (and \pt{}$\Rightarrow$\po{}). This however is the most difficult part.

\medskip

{\bf From now on we assume that $k$ is a $\s$-field of characteristic zero.} 

\medskip


It is an interesting question to ask for a generalization of Theorem \ref{theo: main} to positive characteristic. It may seem plausible to expect that the theorem also holds for a difference field $k$ of positive characteristic $p$ if \pt{} (b) is replaced by 

\medskip

\pt{} (\rm{$b_p$}) \emph{For every $m\geq 1$ and every power $q$ of $p$ there exists an integer $n\geq 1$ such that for every $A\in\Gl_m(k)$ the equation $\s_q(Y)=AY$ has a solution in $\Gl_m(\kn)$.}

\medskip

Here $\s_q$ is defined by $\s_q(y)=\s(y)^q$. Note that the equation $\s_q(Y)=AY$ defines a torsor for the $\s$-algebraic group $G=\Gl_m^{\s_q}=\{g\in\Gl_m|\ \s_q(g)=g\}$. Moreover, any Frobenius difference field, i.e., an algebraically closed field of positive characteristic equipped with a Frobenius endomorphism, satisfies \pt{} (\rm{$b_p$}) with $n=1$ (due to the Lang-Steinberg Theorem).

Many of our intermediate results (and proofs) in this section also work in arbitrary characteristic. However, there are also several steps where we make use of the characteristic zero assumption. At those steps, in positive characteristic, one usually would have to contend with non-reduced algebraic or $\s$\=/algebraic groups and the Frobenius endomorphism. We hope to establish a variant of Theorem \ref{theo: main} in positive characteristic in a future project dedicated to the study of torsors for $\s$-algebraic groups over Frobenius difference fields.

One of the advantages of working in characteristic zero is that every Hopf algebra is reduced. Therefore, a $\s$-algebraic group $G$ is connected if and only if it is integral (\cite[Lemma 6.7]{Wibmer:almostsimple}).

\medskip

In this section we show that if $k$ satisfies \ptweak{} (or \pt{}), then $k$ is (uniformly) strongly $G$-trivial for any $\s$\=/algebraic group $G$ over $k$. 
The proof heavily relies on the induction principle from Section~\ref{section: The induction principle} and the structure theory of $\s$-algebraic groups. Any $\s$-algebraic group can be decomposed into $\s$-algebraic groups belonging to certain more manageable classes of $\s$-algebraic groups and we first show that if $k$ satisfies \ptweak{}, then $k$ is uniformly strongly $G$-trivial for any $G$ belonging to one of those more manageable classes. We start from classes of small (e.g., finite, order zero, $\s$-dimension zero) $\s$-algebraic groups and then build up to the general case. 

\subsection{The proof for $\s$-algebraic groups of order zero} \label{section: order zero}

In this subsection we show that an inversive, algebraically closed $\s$-field is uniformly strongly $G$-trivial for every $\s$-algebraic group $G$ of order zero. The key ingredient for the proof is a decomposition theorem for $\s$-\'{e}tale $\s$-algebraic groups (\cite[Theorem 6.38]{Wibmer:etale}). 
We first show that an algebraically closed $\s$-field is uniformly strongly $G$-trivial for every $\s$-algebraic group $G$ coming from an algebraic group and for finite $\s$-algebraic groups.

\begin{lemma} \label{lemma: algebraic case}
	Let $k$ be an algebraically closed $\s$-field.
	 If $G=[\s]_k\G$ for some algebraic group $\G$ over $k$, then $k$ is uniformly strongly $G$-trivial.
\end{lemma}
\begin{proof}
	Because $G_n$ comes from an algebraic group (Example \ref{ex: varieties and Xn}), it suffices to show that $k$ is uniformly $G$-trivial. 
By \cite[Cor. 3.3]{BachmayrWibmer:TorsorsForDifferenceAlgebraicGroups} every $G$-torsor $X$ is of the form $X=[\s]_k\X$ for some $\G$-torsor $\X$. Thus $X(k)=\X(k)\neq\emptyset$ because $k$ is algebraically closed.
\end{proof}
A $\s$-algebraic group $G$ is called \emph{finite} if $k\{G\}$ is a finite dimensional $k$-vector space.
 
%
%

\begin{lemma} \label{lemma: finite case}
	Let $G$ be a finite $\s$-algebraic group over an algebraically closed $\s$-field $k$. Then $k$ is uniformly strongly $G$-trivial.
\end{lemma}
\begin{proof}
	Because $G_n$ is a finite $\s^n$-algebraic group for every $n\geq 1$, it suffices to show that $k$ is uniformly $G$-trivial.
	
	Let $m=\dim_k(k\{G\})$ denote the dimension of $k\{G\}$ as a $k$-vector space. We claim that $X_n(k)\neq\emptyset$ for every $G$-torsor $X$, where $n=\lcm(1,\ldots,m)$. As $k\{X\}\otimes_k k\{X\}\simeq k\{G\}\otimes_k k\{X\}$, we see that $\dim_k (k\{X\})=m$. So there are at most $m$ prime ideals in $k\{X\}$. Considering the map $\p\mapsto\s^{-1}(\p)$ on $\spec(k\{X\})$ shows that there exists a prime ideal $\p$ in $k\{X\}$ such that $\s^{-l}(\p)=\p$ for some $l$ with $1\leq l\leq m$. As $l$ divides $n=\lcm(1,\ldots,m)$, we also have $\s^{-n}(\p)=\p$. Since $k$ is algebraically closed, the residue field $k(\p)$ of $\p$ equals $k$. The canonical map $k\{X\}\to k(\p)=k$ is a morphism of $k$-$\s^n$-algebras. So $X_n(k)\neq\emptyset$.
\end{proof}

%
%
%

\begin{prop} \label{prop: order zero}
Let $G$ be a $\s$-algebraic group of order zero over an algebraically closed $\s$-field $k$. 
Then $k$ is uniformly strongly $G$-trivial. 
\end{prop}
\begin{proof}

Since we work in characteristic zero, a $\s$-algebraic group of order zero is $\s$-\'{e}tale (\cite[Cor. 4.9]{Wibmer:etale}). According to \cite[Theorem 6.38]{Wibmer:etale} any $\s$-\'{e}tale $\s$-algebraic group has a Babbitt decomposition, i.e., there exists a subnormal series
	$$G\supseteq G_1\supseteq\ldots\supseteq G_n\supseteq 1$$
	such that
	\begin{enumerate}
		\item $G_1=G^\sc$ is the \emph{$\s$-identity component} of $G$ as defined in \cite[Def. 3.16]{Wibmer:etale}. The exact definition of 
		$G^\sc$ is not relevant for our purpose. All we need to know is that the quotient $G/G^\sc=\pis(G)$ is finite (cf. \cite[Prop. 3.13]{Wibmer:etale}).		
		\item $G_i/G_{i+1}$ is isomorphic to a $\s$-algebraic group of the form $[\s]_k\G$ where $\G$ is an \'{e}tale algebraic group for $i=1,\dots,n-1$ and
		\item $G_n$ is \emph{$\s$-infinitesimal}. Again, for our purpose, the exact meaning of $\s$-infinitesimal is not relevant. It suffices to know that $G_n$ is finite and this is the case because a $\s$-closed subgroup of $\s$-\'{e}tale $\s$\=/algebraic group is $\s$-\'{e}tale (\cite[Lemma 4.6]{Wibmer:etale}) and a $\s$-infinitesimal $\s$-\'{e}tale $\s$-algebraic group is finite (\cite[Cor. 6.6]{Wibmer:etale}).
%
	\end{enumerate}
It follows from Lemma\ref{lemma: finite case} that $k$ is uniformly strongly $G/G_1$-trivial and uniformly strongly $G_n$\=/trivial.
Furthermore, $k$ is uniformly strongly $G_i/G_{i+1}$-trivial for $i=1,\ldots,n-1$ by Lemma \ref{lemma: algebraic case}. Thus, according to the induction principle (Corollary \ref{cor: strong induction for tower}), $k$ is uniformly strongly $G$-trivial. 	
\end{proof}
%
The following corollary will allow us to reduce to the case of connected $\s$-algebraic groups in the sequel.
	
\begin{cor} \label{cor: reduce to connected} Assume that $k$ is algebraically closed and let $G$ be a $\s$-algebraic group. If $k$ is (uniformly) strongly $G^o$-trivial, then $k$ is (uniformly) strongly $G$-trivial. 
\end{cor}
\begin{proof}
As $G/G^o$ is $\s$-\'{e}tale it has order zero (\cite[Prop. 4.8]{Wibmer:etale}). So the assertion follows from Proposition \ref{prop: order zero} and Proposition~\ref{prop:InductionPrinciple}(ii) applied to the exact sequence
$1\to G^o\to G\to G/G^o\to 1$.
\end{proof}


\subsection{A Jordan-H\"{o}lder type theorem}
\label{subsec: Jordan Hoelder}

With the order zero case out of the way, the next step is to look at $\s$-algebraic groups of finite positive order, i.e., $\s$-algebraic groups of $\s$-dimension zero. This is the most elaborate part of the proof.
%

We first treat some special cases (diagonalizable $\s$-algebraic groups, abelian $\s$-algebraic groups and Zariski dense $\s$-closed subgroups of simple algebraic groups) and then combine these via a Jordan-H\"{o}lder type decomposition. Since the Jordan-H\"{o}lder type decomposition is also helpful for the diagonalizable case we establish it first.

\begin{defi} \label{defi: almost-simple sdim zero}
	Let $G$ be a $\s$-algebraic group of finite positive order. Then $G$ is \emph{almost-simple} if $\ord(N)=0$ for every proper normal $\s$-closed subgroup $N$ of $G$. 
\end{defi}

As $G^o$ is a normal $\s$-closed subgroup of $G$ with $\ord(G^o)=\ord(G)$ (\cite[Cor. 6.13]{Wibmer:almostsimple}), we see that an almost-simple $\s$-algebraic group is connected.

\begin{lemma} \label{lemma: order drops}
	Let $G$ be a connected $\s$-algebraic of finite positive order. If $H$ is a proper $\s$-closed subgroup of $G$, then $\ord(H)<\ord(G)$.
\end{lemma}
\begin{proof}
Fix an embedding of $G$ into some algebraic group and let $H[i]\leq G[i]$ denote the corresponding Zariski closures. Because $G$ is connected, all the $G[i]$'s are connected (\cite[Lemma~6.12]{Wibmer:almostsimple}). We have $\ord(G)=\dim(G[i])$ and $\ord(H)=\dim(H[i])$ for all sufficiently large $i$. Suppose $\ord(H)=\ord(G)$ so that $\dim(H[i])=\dim(G[i])$. Because we are in characteristic zero and $G[i]$ is connected, this implies $H[i]=G[i]$ for all $i$. Thus $H=G$; a contradiction.
\end{proof}

The classical Jordan-H\"{o}lder Theorem decomposes a finite group into simple groups. Our variant decomposes a connected $\s$-algebraic group of finite positive order into almost-simple $\s$-algebraic groups.

\begin{theo} \label{theo: Jordan Hoelder for sreduced}
	Let $G$ be a connected $\s$-algebraic group of finite positive order. Then there exists a subnormal series
	$$G=G_0\supseteq G_1\supseteq\ldots\supseteq G_n=1$$
	such that $G_i$ is connected and $G_i/G_{i+1}$ is almost-simple for $i=0,\ldots,n-1$. 
\end{theo}
\begin{proof}
	We proceed by induction on $\ord(G)$. If $\ord(G)=1$, then any proper $\s$-closed subgroup of $G$ has order $0$ by Lemma \ref{lemma: order drops}. Thus the theorem holds with $n=1$.
	
Assume $\ord(G)>1$.
	If every proper normal $\s$-closed subgroup of $G$ has order $0$ the theorem also holds with $n=1$. So we may assume that there exists a proper normal $\s$-closed subgroup of $G$ of positive order. Among all such subgroups let $N$ be of maximal order. Because $N^o$ is a characteristic subgroup of $N$ (\cite[Prop. 6.17]{Wibmer:almostsimple}) and $N$ is normal in $G$, it follows that $N^o$ is normal in $G$.	
	
	We claim that $G/N^o$ is almost-simple. A proper normal $\s$-closed subgroup of $G/N^o$ is of the form $N'/N^o$ for some proper normal $\s$-closed subgroup $N'$ of $G$ containing $N^o$ (\cite[Theorem~5.9]{Wibmer:almostsimple}). By the choice of $N$ we have $\ord(N')\leq\ord(N)=\ord(N^o)$, where the last equality uses \cite[Cor. 6.13]{Wibmer:almostsimple}. Thus $\ord(N')=\ord(N^o)$ and $\ord(N'/N^o)=\ord(N')-\ord(N^o)=0$. This shows that $G/N^o$ is almost-simple.
	
	Set $G_1=N^o$. Then $G_1$ is connected, normal in $G$ and $G_0/G_1$ is almost-simple. As $\ord(G_1)<\ord(G)$ (Lemma \ref{lemma: order drops}) it follows from the induction hypothesis there exists a subnormal series $G_1\supseteq G_2\supseteq\ldots\supseteq G_n=1$ with $G_i$ connected and $G_i/G_{i+1}$ almost-simple for $i=1,\ldots, n-1$.
	So the subnormal series $G=G_0\supseteq G_1\supseteq\ldots\supseteq G_n=1$ has the desired properties.	
%
\end{proof}

\subsection{The proof for diagonalizable $\s$-algebraic groups of $\s$-dimension zero}

A $\s$-algebraic group is \emph{diagonalizable} if it is isomorphic to a $\s$-closed subgroup of $\Gm^n$ for some $n\geq 1$. These groups and their torsors are well understood. Abbreviating $g^\alpha=g^{\alpha_0}\s(g)^{\alpha_1}\ldots\s^m(g)^{\alpha_m}$ for  $\alpha=\alpha_0+\alpha_1\s+\ldots+\alpha^m\s^m\in \ZZ[\s]$ and $g^\beta=g_1^{\beta_1}\ldots g_n^{\beta_n}$ for $\beta=(\beta_1,\ldots,\beta_n) \in \ZZ[\s]^n$, every $\s$-closed subgroup of $\Gm^n$ is of the form
$$G=\{g\in\Gm^n|\ g^{\beta_1}=1,\ldots,g^{\beta_m}=1 \},$$
for some $\beta_1,\ldots,\beta_m\in\ZZ[\s]^n$. Every $G$-torsor is of the form $$X=\{x\in\Gm^n |\ x^{\beta_1}=a_1,\ldots,x^{\beta_m}=a_m\}$$ for some $a_1,\ldots,a_m\in k^\times$. So $X$ is described by a system of $m$ multiplicative difference equations in $n$ difference variables. We need to show that this system has a solution in some $k^{[d]}$ using \pt{} (b), which only states that \emph{one} multiplicative difference equation in \emph{one} difference variable has a solution in some $k^{[d]}$. In other words, we need to reduce the general case to the case of one-relator subgroups of $\Gm$. Before going into the details of the reduction proof, we need to provide the foundational properties of diagonalizable $\s$-algebraic groups.

\subsubsection{Diagonalizable $\s$-algebraic groups}

In this section we discuss the basic properties of diagonalizable $\s$-algebraic groups, as relevant for the proof of our main results.
Diagonalizable algebraic groups are equivalent to finitely generated abelian groups. Diagonalizable $\s$-algebraic groups are equivalent to abelian groups $M$ equipped with an endomorphism $\s\colon M\to M$ such that $M$ is finitely $\s$-generated, i.e., there exists a finite subset $B$ of $M$ such that $B,\s(B),\s^2(B),\ldots$ generates $M$ as an abelian group. Let $\ZZ[\s]$ denote the univariate polynomial ring in the variable $\s$ aver the integers. To specify an endomorphism of an abelian group $M$ is equivalent to defining a $\ZZ[\s]$-module structure on $M$. Moreover, $M$ is finitely $\s$-generated if and only $M$ is a finitely generated as $\ZZ[\s]$-module. As a matter of convenience, we mostly use the language of $\ZZ[\s]$-modules, rather than classical difference algebra jargon.

For $\alpha=\alpha_0+\alpha_1\s+\ldots+\alpha^m\s^m\in \ZZ[\s]$ and $g\in R^\times$, where $R$ is a \ks-algebra, we set
$$g^\alpha=g^{\alpha_0}\s(g)^{\alpha_1}\ldots\s^m(g)^{\alpha_m}.$$
Then $(gh)^\alpha=g^\alpha h^\alpha$ for $g,h\in R^\times$ and $(g^\alpha)^\beta=g^{\alpha\beta}$ for $\alpha,\beta\in \ZZ[\s]$. Thus $R^\times$ is a $\ZZ[\s]$-module under $(\alpha, g)\mapsto g^\alpha$. For a $\ZZ[\s]$-module $M$ we denote with $\Hom_{\ZZ[\s]}(M,R^\times)$ the abelian group of morphisms of $\ZZ[\s]$-modules from $M$ to $R^\times$.
\begin{lemma}
	Let $M$ be a finitely generated $\ZZ[\s]$-module. The functor $D_\s(M)$ from the category of \ks\=/algebras to the category of groups, given by $D_\s(M)(R)=\Hom_{\ZZ[\s]}(M,R^\times)$, is a $\s$-algebraic group over $k$. It is represented by $k\{D(M)\}=k[M]$, the group algebra of $M$ over $k$, equipped with the natural extension of $\s$. 
\end{lemma}	
\begin{proof}
	Let $k[M]$ denote the group algebra of the abelian group $M$ over $k$. As $M$ is abelian, $k[M]$ is a commutative $k$-algebra. We define $\s\colon k[M]\to k[M]$ by $\s(\sum_i{\lambda_im_i})=\sum_i\s(\lambda_i)\s(m_i)$ for $\lambda_i\in k$ and $m_i\in M$. This makes $k[M]$ a \ks-algebra. If $B\subseteq M$ generates $M$ as a $\ZZ[\s]$-module, then $B\subseteq M\subseteq k[M]$ generated $k[M]$ as a \ks-algebra. Thus $k[M]$ is finitely $\s$-generated over $k$. Moreover,
	$$\Hom_{k,\s}(k[M],R)\simeq\Hom_{\ZZ[\s]}(M,R^\times)$$
	functorially in $R$.	
\end{proof}

\begin{defi}
	A $\s$-algebraic group is \emph{diagonalizable} if it is isomorphic to $D_\s(M)$ for some finitely generated $\ZZ[\s]$-module $M$.	
\end{defi}

\begin{ex}
	Specifying a morphism $\ZZ[\s]^m\to R^\times$ of $\ZZ[\s]$-modules, is equivalent to choosing $n$ elements from $R^\times$ (the images of the $n$ basis vectors). It follows that $D_\s(\ZZ[\s]^n)=[\s]_k\Gm^n$. 
\end{ex}	

\begin{ex}
	Let $\alpha\in\ZZ[\s]$ and let $G$ be the $\s$-closed subgroup of $\Gm$ given by $G=\{g\in\Gm|\ g^\alpha=1\}$. To specify an element of $\Gm(R)=R^\times$ is equivalent to specifying a morphism $\ZZ[\s]\to R^\times$ of $\ZZ[\s]$-modules. The element of $R^\times$ lies in $G(R)$ if and only if $\alpha$ lies in the kernel of the corresponding morphism. Thus $G\simeq D_\s(\ZZ[\s]/(\alpha))$.
\end{ex}

Let $M$ be a finitely generated $\ZZ[\s]$-module. We define $\rank_{\ZZ[\s]}(M)$ as the dimension of $M\otimes_{\ZZ[\s]}\mathbb{Q}(\s)$ as a $\mathbb{Q}(\s)$-vector space. This agrees with the maximal number of $\ZZ[\s]$-linearly independent elements in $M$. We define $\rank_\ZZ(M)$ as the dimension of $M\otimes_\ZZ\mathbb{Q}$ as a $\mathbb{Q}$-vector space. This is understood to be simply $\infty$ in case it is not finite. We say that $M$ has no $\ZZ$-torsion if $nm=0$ implies $n=0$ or $m=0$ for $n\in\ZZ$ and $m\in M$. We will need the following basic lemma connecting the two different ranks of $M$.

\begin{lemma} \label{lemma: ranks}
	Let $M$ be a finitely generated $\ZZ[\s]$-module. If $\rank_{\ZZ[\s]}(M)=0$, then $\rank_{\ZZ}(M)<\infty$.
\end{lemma}	
\begin{proof}
	Because $\rank_{\ZZ[\s]}(M)=0$ and $M$ is finitely generated, we see that there exists a non-zero $\alpha\in\ZZ[\s]$ such that $\alpha m=0$ for all $m\in M$. So $M$ is a finitely generated $\ZZ[\s]/(\alpha)$-module. It follows that $M\otimes_\ZZ\mathbb{Q}$ is a finitely generated module over $(\ZZ[\s]/(\alpha))\otimes_\ZZ\mathbb{Q}=\mathbb{Q}[\s]/(\alpha)$. As $\mathbb{Q}[\s]/(\alpha)$ is a finite dimensional $\mathbb{Q}$-vector space, it follows that also $M\otimes_\ZZ\mathbb{Q}$ is a finite dimensional $\mathbb{Q}$-vector space. Thus $\rank_{\ZZ}(M)<\infty$.
\end{proof}	 

The following proposition collects all the properties of diagonalizable $\s$-algebraic groups that we will need.

\begin{prop} \label{prop: all about diagonalizable}\mbox{}
	\begin{enumerate}
		\item The (contravariant) functor $M\rightsquigarrow D_\s(M)$ from the category of finitely generated $\ZZ[\s]$\=/modules to the category of diagonalizable $\s$-algebraic groups is an equivalence of categories.	
		\item A sequence $0\to M'\to M\to M''\to 0$ of finitely generated $\ZZ[\s]$-modules is exact if and only if the sequence $1\to D_\s(M'')\to D_\s(M)\to D_\s(M')\to 1$ of $\s$-algebraic groups is exact. 
		\item A $\s$-algebraic group is diagonalizable if and only if it is isomorphic to a $\s$-closed subgroup of $\Gm^n$ for some $n\geq 1$.
		\item Quotients and $\s$-closed subgroups of diagonalizable $\s$-algebraic groups are diagonalizable.
		\item For a finitely generated $\ZZ[\s]$-module $M$ we have $\sdim(D_\s(M))=\rank_{\ZZ[\s]}(M)$ and $\ord(D_\s(M))=\rank_\ZZ(M)$.
		\item The $\s$-algebraic group $D_\s(M)$ is connected if and only if $M$ has no $\ZZ$-torsion. 
	\end{enumerate}
\end{prop}
\begin{proof}
	For a $\s$-algebraic group $G$ let $X(G)$ denote the abelian group of all morphisms $G\to \Gm$ of $\s$\=/algebraic groups. Then $X(G)$ is a $\ZZ[\s]$-module with action of $\s$ obtained by composition with $\s\colon \Gm\to \Gm,\ g\mapsto \s(g)$. We will show that $X$ defines a quasi-inverse to $D_\s$.
	
	Recall that an element $a$ in a $k$-Hopf algebra is group-like if $\Delta(a)=a\otimes a$. For example, the variable $y\in k\{y,y^{-1}\}=k\{\Gm\}$ is group-like. A morphism $G\to \Gm$ of $\s$-algebraic groups is determined by the image of $y$ under the dual map $k\{y,y^{-1}\}=k\{\Gm\}\to k\{G\}$ and this is a group-like element of $k\{G\}$. Conversely, a group-like element $a\in k\{G\}$ defines a morphism $\chi\colon G\to \Gm$ of $\s$-algebraic groups via $\chi(g)=g(a)$ for $g\in G(R)=\Hom_{k,\s}(k\{G\},R)$ for any \ks-algebra $R$. The group-like elements of $k\{G\}$ are stable under $\s$ and this action of $\s$ corresponds to the action of $\s$ on $X(G)$. We can therefore identify $X(G)$ (as a $\ZZ[\s]$-module) with the set of all group-like elements of $k\{G\}$. 
	
	If $M$ is a finitely generated $\ZZ[\s]$-module, the elements of $M$ are group-like elements of $k[M]$. Moreover, since group-like elements are $k$-linearly independent (\cite[Lemma 2.2]{Waterhouse:IntroductiontoAffineGroupSchemes}), it follows that $M$ agrees with the set of group-like elements of $k[M]$.
	Thus $X(D_\s(M))\simeq M$ and it follows that $X$ is a quasi-inverse to $D_\s$. This implies (i).
	
	Let $0\to M'\to M\to M''\to 0$ be a sequence of finitely generated $\ZZ[\s]$-modules. First assume it is exact. 
	The dual map $k[M'']\to k[M]$ to $D_\s(M'')\to D_\s(M)$ is just the $k$-linear extension of $M''\to M$. So, since $M\to M''$ is surjective, also $k[M]\to k[M'']$ is surjective. Therefore, $D_\s(M'')\to D_\s(M)$ is a $\s$-closed embedding. Similarly, since $M'\to M$ is injective, also $k[M']\to k[M]$ is injective. So $D_\s(M)\to D_\s(M')$ is a quotient map. For a \ks-algebra $R$ we have
	\begin{align*}
		\ker(D_\s(M)\to D_\s(M'))(R)& =\ker\big(\Hom_{\ZZ[\s]}(M,R^\times)\to\Hom_{\ZZ[\s]}(M',R^\times)\big) \\
		&=\{g\in \Hom_{\ZZ[\s]}(M,R^\times)|\ g|_{M'}=1\}\simeq \Hom_{\ZZ[\s]}(M'',R^\times)=D_\s(M'')
	\end{align*}
	Thus the sequence of $\s$-algebraic groups is exact.
	
	Conversely, if $1\to D_\s(M'')\to D_\s(M)\to D_\s(M')\to 1$ is exact, then, as $k[M']\to k[M]$ is injective and $k[M]\to k[M'']$ is surjective, it follows that $M'\to M$ is injective and $M\to M''$ is surjective. Thus  
	$D_\s(M'')\to D_\s(M)$ is a $\s$-closed embedding and $D_\s(M)\to D_\s(M')$ is a quotient map. The exactness at $D_\s(M)$ means that for every \ks-algebra $R$, the map 
	\begin{equation} \label{eq: bij for exact}
		\Hom_{\ZZ[\s]}(M'',R^\times)\to \{g\in \Hom_{\ZZ[\s]}(M,R^\times)|\ g|_{M'}=1\}
	\end{equation} is bijective. Let $m_0$ be in the kernel of $M\to M''$. 
	Then, using (\ref{eq: bij for exact}), we see that $m_0$ gets mapped to $1\in R^\times$ under every $g\in \Hom_{\ZZ[\s]}(M,R^\times)$ with $g|_{M'}=1$ for any \ks-algebra $R$. Choosing $R=k[M/M']$ and $g\colon M\to R^\times$ the map that sends $m$ to the equivalence class of $m$ in $M/M'$, it follows that $m_0$ lies in $M'$. This shows that  $0\to M'\to M\to M''\to 0$ is exact and completes the proof of (ii).
	
	If $H$ is a $\s$-closed subgroup of a diagonalizable $\s$-algebraic group $D_\s(M)$, then we have a surjective morphism $k[M]\to k\{H\}$ of \ks-Hopf algebras. As $k[M]$ is generated by group-like elements as a $k$-vector space, it follows that also $k\{H\}$ is generated by group-like elements as a $k$-vector space. From the discussion in the proof of (i), we see that $k\{H\}=k[M'']$, where $M''$ is the $\ZZ[\s]$-module of all group-like elements of $k\{H\}$. Thus $H$ is diagonalizable.
	
	If $G=D_\s(M)$ is a diagonalizable $\s$-algebraic group, we can find an $n\geq 1$ and a surjective morphism $\ZZ[\s]^n\to M$ of $\ZZ[\s]$-modules. According to (ii) this corresponds to a $\s$-closed embedding $G=D_\s(M)\to D_\s(\ZZ[\s]^n)=\Gm^n$. Conversely, if $G$ is a $\s$-closed subgroup of $\Gm^n$, then $G$ is diagonalizable because, as seen above, $\s$-closed subgroups of diagonalizable $\s$-algebraic groups are diagonalizable. This proves (iii).
	
	A quotient of a diagonalizable $\s$-algebraic group, is a quotient by a diagonalizable $\s$-closed subgroup, so the quotient is diagonalizable by (ii). This completes the proof of (iv). 
	
	Let $M$ be a finitely generated $\ZZ[\s]$-module and fix a surjective morphism $\ZZ[\s]^n\to M$ with kernel $N$. As discussed above, this corresponds to a $\s$-closed embedding of $G=D_\s(M)$ into $\Gm^n$. Explicitly, we have $$D_\s(M)=\{(g_1,\ldots,g_n)\in\Gm^n|\ g_1^{\alpha_1}\ldots g_n^{\alpha_n}=1 \ \forall \ (\alpha_1,\ldots,\alpha_n)\in N\}. $$ 
	For $i\geq 0$, the $i$-th order Zariski closure $G[i]$ of $G$ in $\Gm^n$ is then the closed subgroup of $(\Gm^n)^{i+1}=\Gm^{n(i+1)}$ defined by all equations corresponding to tuples in $N$ whose components are polynomials of degree at most $i$. So if $\ZZ[\s]_{\leq i}$ denotes the $\ZZ$-submodule of $\ZZ[\s]$ of all polynomials of degree at most $i$, then $G[i]$ is the diagonalizable algebraic group corresponding to the $\ZZ$-module $\ZZ[\s]_{\leq i}^n/\ZZ[\s]_{\leq i}^n\cap N$. Note that if $B\subseteq M$ is the image of the standard basis of $\ZZ[\s]^n$ in $M$, then $\ZZ[\s]_{\leq i}^n/\ZZ[\s]_{\leq i}^n\cap N$ identifies with $\langle B,\s(B),\ldots,\s^i(B)\rangle$, the $\ZZ$-submodule of $M$ generated by $B,\s(B),\ldots,\s^i(B)$. In particular, $\dim(G[i])=\rank_{\ZZ}(\langle B,\ldots,\s^i(B)\rangle)$.
	
	First assume that $\rank_{\ZZ[\s]}(M)=0$. Then $\rank_\ZZ(M)$ is finite by Lemma \ref{lemma: ranks} and bounds $\dim(G[i])=\rank_{\ZZ}(\langle B,\ldots,\s^i(B)\rangle)$ for all $i$. Thus $\sdim(G)=0$ and $$\ord(G)=\max_i\dim(G[i])=\max_i\rank_{\ZZ}(\langle B,\ldots,\s^i(B)\rangle)=\rank_\ZZ(M).$$
	
	If $d=\rank_{\ZZ[\s]}(M)>0$, then we can find an exact sequence $0\to \ZZ[\s]^d\to M\to N\to 0$ with $\rank_{\ZZ[\s]}(N)=0$. By (ii) this corresponds to an exact sequence $1\to D_\s(N)\to D_\s(M)\to \Gm^d\to 1$. Therefore $D_\s(M)/D_\s(N)=\Gm^d$. As $\rank_{\ZZ[\s]}(N)=0$, we have $\sdim(D_\s(N))=0$ (as argued above). Therefore $\sdim(D_\s(M))=\sdim(\Gm^d)=d=\rank_{\ZZ[\s]}(M)$. Moreover, we have $\rank_\ZZ(M)=\infty$ so that $\ord(D_\s(M))=\rank_\ZZ(M)$ also in this case. This proves (v).
	
	We know that $G$ is connected if and only if $G[i]$ is connected for all $i\geq 0$ (\cite[Lemma~6.12]{Wibmer:almostsimple}). But a diagonalizable algebraic group corresponding to a finitely generated $\ZZ$-module is connected if and only if the $\ZZ$-module has no ($\ZZ$-)torsion (using our characteristic zero assumption). Thus $G$ is connected if and only $\langle B,\ldots,\s^i(B)\rangle$ has no torsion for all $i\geq 0$ and this is equivalent to $M$ having no $\ZZ$-torsion.
\end{proof}	

\begin{ex}
	For $G=\{g\in\Gm|\ g^\alpha=1\}$ with $\alpha\in\ZZ[\s]$ non-zero, we have $\ord(G)=\deg(\alpha)$ because $(\ZZ[\s]/(\alpha))\otimes_\ZZ\mathbb{Q}=\mathbb{Q}[\s]/(\alpha)$ is a $\mathbb{Q}$-vector space of dimension $\deg(\alpha)$.
\end{ex}

\begin{lemma} \label{lemma: Gn diagonalizable}
	If $G$ is a diagonalizable $\s$-algebraic group, then $G_n$ is a diagonalizable $\s^n$-algebraic group for every $n\geq 1$.
\end{lemma}	
\begin{proof}
	If $G$ is a $\s$-closed subgroup of $\Gm^m$, then $G_n$ is a $\s^n$-closed subgroup of $(\Gm^m)^n$. Now use Proposition~\ref{prop: all about diagonalizable} (iii).
\end{proof}

\subsubsection{Reduction to one-relator subgroups of $\Gm$}

We will need the following result about ideals in $\ZZ[\s]$.

\begin{lemma} \label{lemma: exists onerelator ideal}
	Let $I$ be an ideal of $\ZZ[\s]$ such that $I\cap \ZZ=\{0\}$. Then there exists an $\alpha\in \ZZ[\s]$ of positive degree such that $I\subseteq (\alpha)$.
\end{lemma}	
\begin{proof}
	We may assume that $I$ is not the zero ideal.
	The prime ideals of $\ZZ[\s]$ are 
	\begin{itemize}
		\item the zero ideal,
		\item ideals of the form $(p)$ with $p$ a prime number,
		\item ideals of the form $(\alpha)$, where $\alpha\in\ZZ[\s]$ is irreducible and of positive degree,
		\item ideals of the form $(p,\alpha)$, where $p$ is a prime number and $\alpha\in\ZZ[\s]$ is a monic polynomial, irreducible mod $p$ (\cite[Chapter II, \S 1, Example H]{Mumford:TheRedBookfOfVarietiesAndSchemes}).
	\end{itemize}	
	Suppose that all prime ideals minimal above $I$ are of the form $(p)$ or $(p,\alpha)$. Then their intersection, which agrees with the radical of $I$, contains a non-zero integer (the product of all the relevant $p$'s). Thus also $I$ contains a non-zero integer; a contradiction.
	
	Therefore, at least one of the prime ideals minimal above $I$ is of the form $(\alpha)$. In particular, $I\subseteq (\alpha)$.
\end{proof}

\begin{lemma} \label{lemma: ex of onerelator subgroup}
	Let $\alpha\in \ZZ[\s]$ have positive degree. Then $G=\{g\in \Gm|\ g^\alpha=1\}$ is almost-simple if and only if $\alpha\in \ZZ[\s]$ is irreducible.
\end{lemma}
\begin{proof}
	Assume that $G$ is almost-simple and that $\alpha$ factors as $\alpha=\beta\gamma$. Without loss of generality, we assume that $\beta$ has positive degree. As $\{g\in\Gm|\ g^\beta=1\}$ is a $\s$-closed subgroup of $G$ of order $\deg(\beta)>0$, it follows that the morphism $\ZZ[\s]/(\alpha)\to \ZZ[\s]/(\beta)$ is an isomorphism, i.e., $(\alpha)=(\beta)$. This is only possible if $\alpha=\pm\beta$. Thus $\alpha$ is irreducible.
	
	Assume that $\alpha$ is irreducible and let $N$ be a proper $\s$-closed subgroup of $G$. By Proposition~\ref{prop: all about diagonalizable} the inclusion $N\leq G$ corresponds to a surjection $\ZZ[\s]/(\alpha)\to \ZZ[\s]/I$ for some ideal $I$ of $\ZZ[\s]$ with $(\alpha)\subsetneqq I$. Suppose $I\cap\ZZ=\{0\}$. Then, by Lemma \ref{lemma: exists onerelator ideal}, there exists a polynomial $\beta\in \ZZ[\s]$ with $I\subseteq (\beta)$. So $(\alpha)\subsetneqq(\beta)$. This contradicts the irreducibility of $\alpha$. Thus $I\cap\ZZ\neq\{0\}$ and therefore $\rank_{\ZZ}(\ZZ[\s]/I)=0$. Hence $\ord(N)=0$ as desired. 
\end{proof}

We now work towards showing that in fact every almost-simple $\s$-closed subgroup of $\Gm$ is of the form described in Lemma \ref{lemma: ex of onerelator subgroup}. We call a $\s$-closed subgroup $G$ of $\Gm$ a \emph{one-relator subgroup} if there exists a non-zero $\alpha\in \ZZ[\s]$ such that $G=\{g\in\Gm|\ g^\alpha=1\}$.

\begin{lemma} \label{lemma: exists one-relator}
	Let $G$ be a $\s$-closed subgroup of $\Gm$ of finite positive order. Then there exists a $\s$-closed subgroup $H$ of $G$ such that $H$ is a one-relator subgroup of $\Gm$ and $\ord(H)>0$.  
\end{lemma}	
\begin{proof}
	By Proposition \ref{prop: all about diagonalizable} we have $G=D_\s(\ZZ[\s]/I)$ for some ideal $I$ of $\ZZ[\s]$. Moreover, $I\cap \ZZ=\{0\}$ because otherwise we would have $\ord(G)=\rank_{\ZZ}(\ZZ[\s]/I)=0$. We know from Lemma \ref{lemma: exists onerelator ideal} that there exists an $\alpha\in \ZZ[\s]$ of positive degree such that $I\subseteq (\alpha)$. The surjection $\ZZ[\s]/I\to \ZZ[\s]/(\alpha)$ corresponds to the inclusion of $H=\{h\in\Gm|\ h^\alpha=1\}$ into $G$.
\end{proof}

\begin{cor} \label{cor: is one-relator}
	Let $G$ be an almost-simple $\s$-closed subgroup of $\Gm$ of finite positive order. Then $G$ is a one-relator subgroup of $\Gm$.  
\end{cor}	
\begin{proof}
	This is clear from Lemma \ref{lemma: exists one-relator}.
\end{proof}

%

Our next goal is to show that a $\s$-field satisfying \pt{} is uniformly strongly $G$-trivial for every almost-simple diagonalizable $\s$-algebraic group of order one. To this end, we have to take a closer look at the almost-simple $\s$-closed subgroups of $\Gm$ of order one.

\begin{lemma} \label{lemma: almost simple sclosed subgr of Gm of order 1}
	Let $G$ be an almost-simple $\s$-closed subgroup of $\Gm$ of order one. Then $G=\{g\in\Gm|\ g^{\alpha}\s(g)^{\beta}=1\}$ for some integers $\alpha,\beta$ with $\beta\neq 0$ and $\gcd(\alpha,\beta)=1$. 
\end{lemma}
\begin{proof}
	We know from Corollary \ref{cor: is one-relator} that $G$ is a one-relator subgroup and from Lemma \ref{lemma: ex of onerelator subgroup} we know that $G$ is defined by an irreducible polynomial $\beta\s+\alpha$ of degree one. Thus $\beta\neq 0$ and $\gcd(\alpha,\beta)=1$. 
\end{proof}	

We now work towards showing that if $k$ satisfies \pt{}, then $k$ is uniformly strongly $G$-trivial for $G$ a $\s$-algebraic group of the form described in Lemma \ref{lemma: almost simple sclosed subgr of Gm of order 1}. 

\begin{lemma} \label{lemma: special long group}
	Let $\alpha,\beta$ be integers with $\beta\neq 0$ and $\gcd(\alpha,\beta)=1$ and let $G$ be the $\s$-closed subgroup of $\Gm^n$ defined by the equations $$y_1^{\alpha}y_2^{\beta}=1,\ y_2^{\alpha}y_3^{\beta}=1,\ldots, y_{n-1}^\alpha y_n^\beta=1,\ y_n^\alpha\s(y_1)^\beta=1.$$ If $k$ satisfies \ptweak{} (or \pt{}), then $k$ is (uniformly) $G$-trivial.	
\end{lemma}
\begin{proof}
	By \cite[Example 4.2]{BachmayrWibmer:TorsorsForDifferenceAlgebraicGroups}, every $G$-torsor is isomorphic to a $\s$-closed $\s$-subvariety $X$ of $\Gm^n$ defined by equations
	\begin{equation} \label{eq: defining torsor}
		y_1^{\alpha}y_2^{\beta}=a_1,\ y_2^{\alpha}y_3^{\beta}=a_2,\ldots, y_{n-1}^\alpha y_n^\beta=a_{n-1},\ y_n^\alpha\s(y_1)^\beta=a_n
	\end{equation}
	for some $a_1,\ldots,a_n\in k^\times$. If $\alpha=0$, then $\beta=\pm1$ and (\ref{eq: defining torsor}) has a (unique) solution in $k^n$ by Lemma \ref{lemma: inversive}. We can thus assume that $\alpha\neq 0$.
	By assumption, the equation
	\begin{equation} \label{eq: eq for b1}
		y_1^{(-1)^{n+1}\alpha^n}\s(y_1)^{\beta^n}=a_n^{\beta^{n-1}}a_{n-1}^{-\alpha\beta^{n-2}}a_{n-2}^{\alpha^2\beta^{n-3}}\ldots a_2^{(-1)^{n-2}\alpha^{n-2}\beta} a_1^{(-1)^{n-1}\alpha^{n-1}}
	\end{equation}
	has a solution $b_1\in k^{[d]}$ for some $d\geq 1$.  Moreover, if $k$ satisfies \pt{}, then $[d]$ can be chosen uniformly for all $a_1,\ldots,a_n$. Note that $b_1$ is automatically invertible. For $i=2,\ldots,n$ choose $b_i\in k^{[d]}$ such that
	\begin{equation} \label{eq: for bi}
		b_i^{\alpha^{n-i+1}}=a_i^{\alpha^{n-i}}a_{i+1}^{-\alpha^{n-i-1}\beta}a_{i+2}^{\alpha^{n-i-2}\beta^2}\ldots a_n^{(-1)^{n-i}\beta^{n-i}}\s(b_1)^{(-1)^{n-i+1}\beta^{n-i+1}}.
	\end{equation}
	This is possible because $k$ is algebraically closed. The choice of $b_i$ is not unique but any other possible choice is of the form $\zeta b_i$, where $\zeta=(\zeta_1,\ldots,\zeta_d)\in k^{[d]}$ with $\zeta_i\in k$ an $\alpha^{n-i+1}$-th root of unity, i.e., $\zeta_i^{\alpha^{n-i+1}}=1$. 
	Note that formula (\ref{eq: for bi}) also holds for $i=1$ by (\ref{eq: eq for b1}).
	Then 
	$$(b_i^\alpha b_{i+1}^\beta)^{\alpha^{n-i}}=a_i^{\alpha^{n-i}}$$
	for $i=1,\ldots,n-1$. Thus $b_i^\alpha b_{i+1}^\beta$ and $a_i$ agree up to multiplication with an $\alpha^{n-i}$-th root of unity in every component. Note that the choice of $b_{i+1}$ is only unique up to multiplication with an $\alpha^{n-i}$-th root of unity in every component. Moreover, as $\gcd(\alpha^{n-i},\beta)=1$, the map $\zeta\mapsto \zeta^\beta$ is a bijection on the set of $\alpha^{n-i}$-th roots of unity in $k$. This shows that we can adjust all the roots of unity in the choice of $b_2,\ldots,b_n$ such that $b_i^\alpha b_{i+1}^\beta=a_i$ for $i=1,\ldots,n-1$. More precisely, $b_1$ is already fixed and we adjust $b_2$ such that $b_1^\alpha b_2^\beta=a_1$. Then we adjust $b_3$ such that $b_2^\alpha b_3^\beta=a_2$. We continue like this until we adjust $b_n$ such that $b_{n-1}^\alpha b_n^\beta=a_n$. Thus $b=(b_1,\ldots,b_n)$ satisfies the first $n-1$ equations in (\ref{eq: defining torsor}). Evaluating (\ref{eq: for bi}) for $i=n$ shows that $b$ also satisfies the last equation. Therefore $b\in X(k^{[d]})\neq\emptyset$ as desired.
\end{proof}

\begin{lemma} \label{lemma: special short group}
	Let $\alpha,\beta$ be integers with $\beta\neq 0$ and $\gcd(\alpha,\beta)=1$ and  set $G=\{g\in\Gm|\ g^{\alpha}\s(g)^{\beta}=1\}$. If $k$ satisfies \ptweak{} (or \pt{}), then $k$ is (uniformly) strongly $G$-trivial.
\end{lemma}
\begin{proof}
	Let $n\geq 1$. We have to show that $(k,\s^n)$ is (uniformly) $G_n$-trivial for every $n\geq 1$. Note that $G_n$ is the $\s^n$-closed subgroup of $\Gm^n$ defined by the equations
	$$y_1^{\alpha}y_2^{\beta}=1,\ y_2^{\alpha}y_3^{\beta}=1,\ldots, y_{n-1}^\alpha y_n^\beta=1,\ y_n^\alpha\s^n(y_1)^\beta=1.$$
	Because $(k,\s^n)$ satisfies \ptweak{} (or \pt{}) by Lemma \ref{lem: pt for powers of s}, the claim follows from Lemma \ref{lemma: special long group}.
\end{proof}

\begin{lemma} \label{lemma: get into Gm}
	Let $G$ be a non-trivial diagonalizable $\s$-algebraic group. Then there exists a proper $\s$-closed subgroup $N$ of $G$ such that $G/N$ is isomorphic to a $\s$-closed subgroup of $\Gm$.
\end{lemma}
\begin{proof}
	As $G$ is diagonalizable, we can assume that $G$ is a $\s$-closed subgroup of some $\Gm^n$. As $G$ is non-trivial, at least one of the projections $G\to\Gm$ will have non-trivial image. The kernel $N$ has the required properties.
	%
	%
\end{proof}

\begin{lemma} \label{lemma: proof for almostsimple diagonalizable order one}
	Let $G$ be an almost-simple diagonalizable $\s$-algebraic group of order one. If $k$ satisfies \ptweak{} (or \pt{}), then 
	$k$ is (uniformly) strongly $G$-trivial.
\end{lemma}	
\begin{proof}
	By Lemma \ref{lemma: get into Gm}, there exists an exact sequence
	$1\to N\to G\to H\to 1$ with $N$ a proper $\s$\=/closed subgroup of $G$ and $H$ a $\s$-closed subgroup of $\Gm$. Because $N$ has order zero, we know that $k$ is uniformly strongly $N$-trivial by Proposition \ref{prop: order zero}. Moreover, $H$ is an almost-simple diagonalizable $\s$-algebraic group of order one. Thus, by the induction principle (Proposition \ref{prop:InductionPrinciple} (ii)), we can assume that $G$ is a $\s$-closed subgroup of $\Gm$. So $G=\{g\in\Gm|\ g^{\alpha}\s(g)^{\beta}=1\}$ with $\beta\neq 0$ $\gcd(\alpha,\beta)=1$ by Lemma \ref{lemma: almost simple sclosed subgr of Gm of order 1} and therefore $k$ is (uniformly) strongly $G$-trivial by Lemma \ref{lemma: special short group}.
\end{proof}

With the order one case completed, we need two more lemmas to get to the general case of a diagonalizable $\s$-algebraic group of $\s$-dimension zero. We need the strong $G$-triviality but first we record the $G$-triviality for almost-simple diagonalizable $\s$-algebraic groups.

\begin{lemma} \label{lemma: almost-simple diagonalizable}
	Let $G$ be an almost-simple diagonalizable $\s$-algebraic group of finite positive order. If $k$ satisfies \ptweak{} (or \pt{}), then 
	$k$ is (uniformly) $G$-trivial.
\end{lemma}	
\begin{proof}
	By Lemma \ref{lemma: get into Gm}, there exists an exact sequence $1\to N\to G\to H\to 1$ with $H$ a non-trivial $\s$-closed subgroup of $\Gm$. As $N$ has order zero, $k$ is (uniformly) strongly $N$-trivial (Proposition~\ref{prop: order zero}). By the induction principle (Proposition \ref{prop:InductionPrinciple} (i) in this case), we can thus assume that $G$ is a $\s$-closed subgroup of $\Gm$. By Lemma \ref{cor: is one-relator} every almost-simple $\s$-closed subgroup of $\Gm$ is a one-relator subgroup. So $G$ is of the form $G=\{g\in\Gm|\ g^\alpha=1\}$ for some $\alpha\in \ZZ[\s]$. By  \cite[Example 4.2]{BachmayrWibmer:TorsorsForDifferenceAlgebraicGroups} every $G$-torsor is of the form $X=\{x\in\Gm|\ x^\alpha=a\}$ for some $a\in k^\times$. So $k$ is (uniformly) $G$-trivial by \ptweak{} (c) (or \pt{} (c)). 
\end{proof}

\begin{lemma} \label{lem: almost simple diagonalizable}
	Let $G$ be an almost-simple diagonalizable $\s$-algebraic group of finite positive order. If $k$ satisfies \ptweak{} (or \pt{}), then $k$ is (uniformly) strongly $G$-trivial.
\end{lemma}
\begin{proof}
	We will prove this by induction on the order of $G$. To be precise, our induction hypothesis is that every difference field satisfying \ptweak{} (or \pt{}) is (uniformly) strongly $H$-trivial for every almost-simple diagonalizable $\s$-algebraic group $H$ with $1\leq\ord(H)<\ord(G)$. The base case $\ord(G)=1$ is Lemma~\ref{lemma: proof for almostsimple diagonalizable order one}.

	For the induction step we have to show that $(k,\s^n)$ is (uniformly) $G_n$-trivial for every $n\geq 1$. Note that $G_n$ is a diagonalizable $\s^n$-algebraic group (Lemma \ref{lemma: Gn diagonalizable}) with $\ord(G_n)=\ord(G)$ (Lemma~ \ref{lem: sdim and order preserved}). Moreover, $(k,\s^n)$ satisfies \ptweak{} (or \pt{}) by Lemma \ref{lem: pt for powers of s}. So if $G_n$ is almost-simple, then the claim follows from Lemma \ref{lemma: almost-simple diagonalizable}. We can therefore assume that $G_n$ is not almost-simple. Note that $G_n$ is connected, because $G$ is connected (Lemma \ref{lem: connectedness preserved}). By Theorem \ref{theo: Jordan Hoelder for sreduced} there exists a subnormal series $G_n=H_0\supseteq H_1\supseteq \ldots\supseteq H_m=1$ with $H_i$ connected and $H_i/H_{i+1}$ almost-simple for $i=0,\ldots, m-1$. Because the quotients $H_i/H_{i+1}$ are diagonalizable (Proposition~\ref{prop: all about diagonalizable} (iv)) and $\ord(H_i/H_{i+1})<\ord(G_n)=\ord(G)$, it follows from the induction hypothesis that $(k,\s^n)$ is (uniformly) strongly $H_i/H_{i+1}$-trivial for $i=0,\ldots,m-1$. Therefore Corollary \ref{cor: strong induction for tower} shows that $(k,\s^n)$ is (uniformly) strongly $G_n$-trivial.
\end{proof}

\begin{prop} \label{prop: proof for diagonalizable}
	Let $G$ be a diagonalizable $\s$-algebraic group of $\s$-dimension zero. If $k$ satisfies \ptweak{} (or \pt{}), then $k$ is (uniformly) strongly $G$-trivial.
\end{prop}	
\begin{proof}
If $\ord(G)=0$, the claim follows from Proposition \ref{prop: order zero}. So we may assume $\ord(G)>0$. Furthermore, we can assume that $G$ is connected (Corollary \ref{cor: reduce to connected}). Now the claim follows by combining Theorem~\ref{theo: Jordan Hoelder for sreduced}, Lemma \ref{lem: almost simple diagonalizable} and Corollary \ref{cor: strong induction for tower}.
\end{proof}

\subsection{The proof for $\s$-algebraic vector groups of $\s$-dimension zero}
	
	\label{subsection: abelian}

In this subsection we show that a $\s$-field satisfying \ptweak{} (or \pt{}) is (uniformly) strongly $G$-trivial for every $\s$-closed subgroup $G$ of some $\Ga^n$ of $\s$-dimension zero.

\begin{defi}
A $\s$-algebraic group is a \emph{$\s$-algebraic vector group} if is is isomorphic to a $\s$-closed subgroup of $\Ga^n$ for some $n\geq 1$. 	
\end{defi}


In the sequel it will often be helpful to be able to assume that $G$ is $\s$-reduced. If $X$ is a $\s$-variety, there exists a unique largest $\s$-closed $\s$-subvariety $X_\sred$ of $X$ that is $\s$-reduced. It satisfies $k\{X_{\sred}\}=k\{X\}_{\sred}$. If $G$ is a $\s$-algebraic group and $k$ is inversive, then $G_\sred$ is a $\s$-closed subgroup of $G$ (\cite[Cor. 2.9]{Wibmer:almostsimple}).

\begin{lemma} \label{lemma: reduce to sreduced}
	Let $k$ be an inversive $\s$-field and let $G$ be a $\s$-algebraic group over $k$. If $k$ is (uniformly) $G_\sred$-trivial, then $k$ is (uniformly) $G$-trivial.
\end{lemma}
\begin{proof}
	By \cite[Cor. 2.8 (ii)]{Wibmer:almostsimple} the functor $X\to X_\sred$ commutes with products. Thus $X_\sred$ is a $G_\sred$-torsor. A morphism $k\{X_\sred\}\to\kn$ yields a morphism $k\{X\}\to \kn$ by composing with $k\{X\}\to k\{X_\sred\}$.
\end{proof}

\begin{lemma} \label{lemma: reduce to sredued for strongly}
	Let $k$ be an inversive $\s$-field and let $G$ be a $\s$-algebraic group over $k$. If $k$ is (uniformly) strongly $G_{\sred}$-trivial, then $k$ is (uniformly) strongly $G$-trivial.
\end{lemma}
\begin{proof}
	Let $n\geq 1$. By assumption, for every $(G_{\sred})_n$-torsor $Y$, there exists a $d\geq 1$ such that $Y_d(k,\s^{nd})\neq\emptyset$. Moreover, $d$ does not depend on the choice of $Y$ in the uniform case.
	
	Let $X$ be a $G_n$-torsor. By \cite[Cor. 2.9 (ii)]{Wibmer:almostsimple} the functor $X\to X_{\s^n\text{-}\operatorname{red}}$ commutes with products. Thus $X_{\s^n\text{-}\operatorname{red}}$ is a $(G_n)_{\s^n\text{-}\operatorname{red}}$-torsor. Clearly, $\I((G_n)_{\s^n\text{-}\operatorname{red}})=\I((G_{\sred})_n)$. So $(G_n)_{\s^n\text{-}\operatorname{red}}=(G_{\sred})_n$ and by assumption  $(X_{\s^n\text{-}\operatorname{red}})_d(k,\s^{nd})\neq\emptyset$. So there exists a morphism of $k$-$\s^{nd}$-algebras $k\{X_{\s^n\text{-}\operatorname{red}}\}\to k$ which we can compose with $k\{X\}\to k\{X_{\s^n\text{-}\operatorname{red}}\}$ to obtain a morphism $k\{X\}\to k$ of $k$-$\s^{nd}$-algebras. So $X_d(k,\s^{nd})\neq\emptyset$.
\end{proof}

Before tackling general $\s$-algebraic vectors groups of $\s$-dimension zero, we treat the $\s$-closed subgroups of $\Ga$.

	\begin{lemma} \label{lemma: subgroups of Ga}
 Let $G$ be a $\s$-closed subgroup of $\Ga$. 
		If $k$ satisfies \ptweak{} (or \pt{}), then $k$ is (uniformly) $G$-trivial.
	\end{lemma}
	\begin{proof}
		Since $k$ is uniformly $\Ga$-trivial (Lemma \ref{lemma: algebraic case}) we may assume that $G$ is a proper $\s$-closed subgroup of $\Ga$. Moreover, by Lemma \ref{lemma: reduce to sreduced} we may assume that $G$ is $\s$-reduced. The proper $\s$-closed subgroups of $\Ga$ are all defined by a linear difference equation (\cite[Cor. A.3]{DiVizioHardouinWibmer:DifferenceAlgebraicRelations}.) I.e.,  there exists a linear $\s$-equation $\mathcal{L}=\s^m(y)+\lambda_{m-1}\s^{m-1}(y)+\ldots+\lambda_0 y$ over $k$ such that $G=\{g\in \Ga|\ \mathcal{L}(g)=0\}$. Because $G$ is $\s$-reduced (and using that $k$ is inversive (Lemma \ref{lemma: inversive})) we see that $\lambda_0\neq 0$.

		By \cite[Example 5.4 ]{BachmayrWibmer:TorsorsForDifferenceAlgebraicGroups} every $G$-torsor is isomorphic to a $G$-torsor of the form $X=\{x\in \A^1|\ \mathcal{L}(x)=a\}$ for some $a\in k$. Consider the linear $\s$-equation $\s(Y)=AY$ where 
		$$A=\left[\begin{array}{ccccc}
		0 & 1 & 0 & \cdots & 0 \\
		0 & 0 & 1 & \cdots & 0 \\
		\vdots & &  &   & \vdots \\
		0 & 0 & \cdots & 1 & 0 \\	 
		-\lambda_0 & -\lambda_1 & \cdots & -\lambda_{m-1} & a \\
		0 & 0 & \cdots & 0 & 1 	
		\end{array}\right]\in\Gl_{m+1}(k).
		$$ 	
		By assumption \ptweak{} (b) there exists $n\geq 1$ and $Y\in\Gl_{m+1}(\kn)$ with $\s(Y)=AY$.
		  
		Since $Y\in\Gl_{m+1}(\kn)$ there exists at least one column
		$$y=\left[\begin{array}{c}
		y_0 \\
		\vdots \\
		y_m
		\end{array}
		\right]\in (\kn)^{m+1}
		$$
		of $Y$
		with $y_m\neq 0$. As $\s(y_m)=y_m$, and the solutions are a $(\kn)^\s$-vector space, we can assume $y_m=1$. For such a $y$ we have $\mathcal{L}(y_0)=a$, i.e., $y_0\in X(\kn)$. Moreover, if the stronger condition \pt{} (b) holds, then $n$ does not depend on $A$ and thus it does not depend on the choice of $X$. 		
\end{proof}

	\begin{prop} \label{prop: vector groups of sdim 0 are strongly trivial}
 Let $G$ be a $\s$-algebraic vector group of $\s$-dimension zero. 
		If $k$ satisfies  \ptweak{} (or \pt{}), then $k$ is (uniformly) strongly $G$-trivial.
	\end{prop}
	\begin{proof}
		We proceed by induction on the order of $G$. The case $\ord(G)=0$ is handled by Proposition~\ref{prop: order zero}. So let us assume that $\ord(G)>0$. By Corollary \ref{cor: reduce to connected} we may assume that $G$ is connected. Let $n\geq 1$. Then $G_n$ is a $\s^n$-algebraic vector group (cf. Example \ref{ex: varieties and Xn}). Say $G_n\leq \Ga^m$ (as $\s^n$\=/algebraic groups). As $\ord(G_n)=\ord(G)>0$ (Lemma~ \ref{lem: sdim and order preserved}), at least one of the projections $G_n\to \Ga$ has to be non-trivial. 
%
%
		 So we have an exact sequence $1\to N\to G_n\to H\to 1$ of $\s^n$-algebraic groups with $H$ a non-trivial $\s^n$-closed subgroup of $\Ga$.
		 Thus $N$ is a proper $\s^n$-closed subgroup of $G_n$. As $G$ is connected, also $G_n$ is connected (Lemma \ref{lem: connectedness preserved}) and therefore $\ord(N)<\ord(G_n)=\ord(G)$ (Lemma \ref{lemma: order drops}).
		  So, by the induction hypothesis and Lemma \ref{lem: pt for powers of s} it follows that $(k,\s^n)$ is (uniformly) strongly $N$-trivial. Moreover, $(k,\s^n)$ is (uniformly) $H$-trivial by Lemma \ref{lemma: subgroups of Ga} (and Lemma \ref{lem: pt for powers of s}). So it follows from Proposition \ref{prop:InductionPrinciple}(i) that $(k,\s^n)$ is (uniformly) $G_n$-trivial as desired. \end{proof}

\subsection{The proof for Zariski-dense $\s$-closed subgroups of simple algebraic groups}

\label{subsec: Zariski dense in simple}

Note that \ptweak{} (b) (or \pt{} (b)) means that $k$ is (uniformly) $\Gl_m^\s$-trivial for all $m\geq 1$. The key to showing that this implies that $k$ is (uniformly) $G$-trivial for any Zariski-dense $\s$-closed subgroup $G$ of a simple algebraic group, is to first show that it implies that $k$ is (uniformly) $\G^\s$-trivial for any closed subgroup $\G\leq\Gl_m$ defined over $k^\s$.

Recall that a $\s$-ring $R$ is called \emph{$\s$-simple} if $\{0\}$ and $R$ are the only $\s$-ideals of $R$.

\begin{lemma} \label{lemma: structure of R}
	Assume that $k$ is algebraically closed and let $R$ be a $\s$-simple \ks-algebra such that $R$ is finitely generated as a $k$-algebra and integral over $k$. Then there exists an $r\geq 1$ such that $R$ is 
	isomorphic as a \ks-algebra to $k^r$ equipped with $\s\colon k^r\to k^r, (a_1,\dots,a_r)\mapsto (\s(a_r),\s(a_1),\dots,\s(a_{r-1}))$ and $k\to k^r,\ \lambda\mapsto (\lambda,\ldots,\lambda)$.
\end{lemma}
\begin{proof}
	Because $R$ is finitely generated as a $k$-algebra and integral over $k$, we see that $R$ is a finite dimensional $k$-vector space. Furthermore, because the nilradical of $R$ is a $\s$-ideal, it must be trivial. So $R$ is reduced. Since $k$ is algebraically closed, it follows that $R$ is isomorphic as a $k$-algebra to $k^r$ for some $r\geq 1$. So we identify $R$ with $k^r$. Because the kernel of $\s\colon k^r\to k^r$ is a $\s$-ideal, it must be trivial. So $\s\colon k^r\to k^r$ is injective and therefore permutes the primitive idempotent elements $e_1,\ldots,e_r$ of $k^r$. Any subset of $e_1,\ldots,e_r$ stable under $\s$ generates a $\s$-ideal. Therefore $\s$ must permute $e_1,\ldots,e_r$ as cycle of length $r$. Without loss of generality we can assume that $\s(e_1)=e_2,\ \s(e_2)=e_3,\ldots,\s(e_n)=e_1$. This shows that $\s\colon k^r\to k^r$ has the required form.
\end{proof}

As we will use some Galois cohomology in the upcoming proofs, we briefly recall the definitions. (See \cite[Section 3, k]{Milne:AlgebraicGroups} or \cite{Serre:GaloisCohomology} for a fuller treatment). Let $k'/k$ be a Galois field extension with Galois group $\Gamma_{k'/k}$ and let $\G$ be an algebraic group over $k$.  A continuous map $\chi\colon \Gamma_{k'/k}\to \G(k')$, where $\G(k')$ has the discrete topology, is a \emph{cocycle} if $\chi(\gamma_1\gamma_2)=\chi(\gamma_1)\gamma_1(\chi(\gamma_2))$ for all $\gamma_1,\gamma_2\in \Gamma_{k'/k}$. Two cocycles $\chi_1,\chi_2$ are equivalent if there exists a $g\in\G(k')$ such that $\chi_1(\gamma)=g^{-1}\chi_2(\gamma)\gamma(g)$ for all $\gamma\in\Gamma$.
The Galois cohomology set $H^1(k'/k,\G)$ is the set of equivalence classes of cocycles. It has a distinguished element, the equivalence class consisting of the \emph{principal cocycles}, i.e., the cocycles of the form $\chi(\gamma)=g^{-1}\gamma(g)$ for some $g\in\G(k')$. Let $\Sigma_{k'/k}(\G)$ denote the set of isomorphism classes of right $\G$-torsors $\X$ with $\X(k')\neq\emptyset$. For $x\in \X(k')$ and $\gamma\in \Gamma_{k'/k}$ there exists a unique $\chi(\gamma)\in\G(k')$ such that $\gamma(x)=x\chi(\gamma)$. Then $\chi\colon\Gamma_{k'/k}\to\G(k')$ is a cocycle and its equivalence class does not depend on the choice of $x\in \X(k')$. This yields a well-defined map $\Sigma_{k'/k}(\G)\to H^1(k'/k,\G)$ which turns out to be a bijection under which the isomorphism class of the trivial $\G$-torsor corresponds to the distinguished element of $H^1(k'/k,\G)$.

\begin{lemma}\label{lemma: H1 trivial}
  Assume that $k$ is algebraically closed and let $\G$ be an algebraic group over $k^\s$. Then there exists an integer $s\geq 1$ such that for every cocycle
   $\Gamma_{\overline{k^\s}/k^\s}\to \G(\overline{k^\s})$
   the 
   induced cocycle $\Gamma_{\overline{k^\s}/k^{\s^s}}\to \Gamma_{\overline{k^\s}/k^\s}\to \G(\overline{k^\s})$
    is principal.
\end{lemma}
\begin{proof}
It is well known (Theorem \cite[Theorem 2.1.12]{Levin:difference}) that the algebraic closure $\overline{k^\s}$ of $k^\s$ in $k$ coincides with the subfield of all periodic elements of $k$. It follows that the absolute Galois group $\Gamma_{\overline{k^\s}/k^\s}$ of $k^\s$ is topologically generated by $\s\colon \overline{k^\s}\to \overline{k^\s}$.  
Therefore $k^\s$ is a field of type (F) in the sense of \cite[Chapter III, Section 4.2]{Serre:GaloisCohomology} and it follows from \cite[Chapter III, Section 4.3, Theorem 4]{Serre:GaloisCohomology} that $H^1(\overline{k^\s}/k^\s,\G)$ is finite. Thus the set of isomorphism classes of $\G$-torsors is finite and so there exists an extension $k^{\s^s}/k^\s$ such that $\X(k^{\s^s})\neq\emptyset$ for every $\G$-torsor $\X$.
The diagram
$$ 
\xymatrix{
	\Sigma_{\overline{k^\s}/k^\s}(\G) \ar[r] \ar[d] & \Sigma_{\overline{k^\s}/k^{\s^s}}(\G_{k^{\s^s}}) \ar[d] \\
	H^1(\overline{k^\s}/k^\s,\G) \ar[r] & H^1(\overline{k^\s}/k^{\s^s},\G_{k^{\s^s}})	
}
$$
where the top horizontal arrow is given by base change from $k^\s$ to $k^{\s^s}$ and the bottom horizontal arrow is induced by composing a cocycle with $\Gamma_{\overline{k^\s}/k^{\s^s}}\to \Gamma_{\overline{k^\s}/k^\s}$ is commutative. By choice of $s$, the top horizontal map is trivial. Therefore also the bottom horizontal map is trivial. This implies the claim of the lemma.
\end{proof}


With these two lemmata at hand, we can strengthen conditions \ptweak{}(b) and \pt{}(b) from solutions in $\Gl_m$ to solutions in closed subgroups $\G\leq \Gl_m$ defined over $k^\s$:

\begin{prop} \label{prop: stronger P3b}
  If $k$ satisfies \ptweak{} then the following stronger version of \ptweak{}(b) also holds:
    For every pair of integers $m,d \geq 1$, every closed subgroup $\G\leq\Gl_m$ defined over $k^\s$ and every $A\in\G(k)$, there exists an integer $n\geq 1$ such that the equation $\s^d(Y)=AY$ has a solution in $\G(\kn)$. Moreover, if $k$ satisfies \pt{}, then $n$ can be chosen uniformly for all $A \in \G(k)$.  
\end{prop}
\begin{proof}
Let $m,d \geq 1$ be integers and $\G\leq \Gl_m$ a closed subgroup defined over $k^\s$. If \pt{}(b) holds, then we use \pt{}(b$_d$') from Lemma \ref{lemma: on pt (b) was remark} and obtain an $l\in \mathbb{N}$ such that for all $A\in \G(k)$ there exists a $Y\in \Gl_m(k)$ with $\s^{dl}(Y)=\s^{(l-1)d}(A)\ldots\s^d(A)AY$. If instead \ptweak{} holds, then such an $l$ also exists but it may depend on $A$. 
 We fix an $A\in \G(k)$ and abbreviate $B=\s^{(l-1)d}(A)\ldots\s^d(A)A$. Note that $B \in \G(k)$ because $\G\leq\Gl_m$ is defined over $k^\s$. 
 Let $s\geq 1$ be as in the statement of Lemma \ref{lemma: H1 trivial} but applied to the difference field $(k,\s^{dl})$.
 We claim that there exists a $\hat Y\in \G(k)$ with $\s^{sdl}(\hat Y)=\s^{(s-1)dl}(B)\ldots\s^{dl}(B)B\hat Y$. Note that this implies the claim of the proposition by Lemma~\ref{lemma: on pt (b)} because $\s^{(s-1)dl}(B)\ldots\s^{dl}(B)B=\s^{(ls-1)d}(A)\ldots\s^d(A)A$ and $ls$ is uniform in $A$ if \pt{} holds.
 
%
%

Define ${\tilde \s}=\s^{dl}$ and let $T=(T_{ij})$ be the matrix of coordinate functions on $\Gl_m$. 
Consider the coordinate ring $k[\Gl_m]=k[T,1/\det(T)]$ of $\Gl_m$ over $k$ as a $k$-${\tilde \s}$-algebra via ${\tilde \s}(T)=BT$. 
Let $\mathbb{I}(\G)\subseteq k[\Gl_m]$ be the defining ideal of $\G$ and let $f=f(T)\in \mathbb{I}(\G)$ have constant coefficients. Then $\tilde \s(f(T))(g)=f(\tilde\s(T))(g)=f(BT)(g)=f(Bg)=0$ for all $g\in \G(k)$ because $Bg\in\G(k)$. So $\tilde \s(f)\in \mathbb{I}(\G)$. Because $\G$ is defined over $k^\s$ this shows that $\mathbb{I}(\G)\subseteq k[\Gl_m]$ is a ${\tilde \s}$-ideal.

%

Thus we can find a maximal ${\tilde \s}$-ideal $\mathfrak{m}\subseteq k[\Gl_m]$ that contains $\mathbb{I}(\G)$. The quotient $R=k[\Gl_m]/\mathfrak{m}$ is then a $\tilde \s$-simple $k$-${\tilde \s}$-algebra. By construction, the image $\tilde Y$ in $\Gl_m(R)$ of $T$ is contained in $\G(R)$ and satisfies $\tilde \s(\tilde Y)=B\tilde Y$. Recall, from the first paragraph, that we have a $Y\in \Gl_m(k)$ with $\tilde \s(Y)=BY$. Since two solutions matrices only differ by a constant matrix, there exists a $C\in \Gl_m(R^{\tilde \s})$ with $\tilde Y=YC$ and we conclude that $R=k[\tilde Y, 1/\det(\tilde Y)]=k[C,1/\det(C)]$ is generated as a $k$-algebra by the elements in $R^{\tilde \s}$. As $R$ is $\tilde \s$-simple, $R^{\tilde \s}$ is a field. Moreover, it is an algebraic extension of $k^{\tilde \s}$. (This is well-known in case $k^{\tilde \s}$ is algebraically closed \cite[Lemma 1.8]{SingerPut:difference} and follows similarly in the general case. It also follows from \cite[Prop. 2.11]{Wibmer:Chevalley}). Note, however, that $R^{\tilde \s}$ need not be contained in $k$ because $R$ need not be an integral domain. In any case, $R$ is generated as a $k$-algebra by elements that are integral over $k$ and therefore $R$ itself is integral over $k$. By Lemma \ref{lemma: structure of R} we have 
$R= k^r$ with $\tilde \s\colon k^r\to k^r, (a_1,\dots,a_r)\mapsto (\tilde \s(a_r),\tilde \s(a_1),\dots,\tilde \s(a_{r-1}))$ for some $r\geq 1$. Then $R^{\tilde \s}=\{(a,\tilde\s(a),\dots, \tilde \s^{r-1}(a) ) \mid a \in k^{{\tilde \s}^r}\}$. So there exists a matrix $D \in\Gl_m(k^{{\tilde \s}^r})$ with $C=(D,\tilde \s(D),\dots,\tilde \s^{r-1}(D))$. Thus $\tilde Y \in \G(R)$ can be rewritten as $\tilde Y=YC=(YD,Y\tilde \s(D),\dots,Y\tilde \s^{r-1}(D))$. Because $\tilde Y\in\G(R)$, we conclude that $Y\tilde \s^i(D) \in \G(k)$ for $i=0,\ldots,{r-1}$. This implies $D^{-1}\tilde \s^i(D)=(YD)^{-1}Y\tilde \s^i(D) \in \G(k^{{\tilde \s}^r})$.

The Galois group 
$\Gamma_{k^{{\tilde\s}^r}/k^{\tilde\s}}$ of $k^{{\tilde \s}^r}$ over $ k^{\tilde \s}$ is generated by $\tilde{\s}\colon k^{{\tilde \s}^r}\to k^{{\tilde \s}^r}$. (However, the order of $\tilde{\s}\in \Gamma_{k^{{\tilde\s}^r}/k^{\tilde\s}}$ could be less than $r$). In any case, we can define a map $\chi\colon \Gamma_{k^{{\tilde\s}^r}/k^{\tilde\s}}\to \G( k^{\tilde \s^r})$ by $\chi(\tilde \s^i)=D^{-1}\tilde \s^i(D)$ for $i=0,\ldots,{r-1}$. It is straight forward to check that $\chi$ is a cocycle. Note that $\chi$ may not be principal because we don't know if $D$ lies in $\G( k^{\tilde \s^r})$. By the choice of $s$, the induced cocycle $\overline{\chi}\colon \Gamma_{\overline{k^{{\tilde\s}}}/k^{\tilde\s}}\to \Gamma_{k^{{\tilde\s}^r}/k^{\tilde\s}}\xrightarrow{\chi} \G( k^{\tilde \s^r})\to  \G(\overline{k^{\tilde \s}})$ becomes principal, when restricted to $\Gamma_{\overline{k^{{\tilde\s}}}/k^{\tilde\s^s}}\leq\Gamma_{\overline{k^{{\tilde\s}}}/k^{\tilde\s}}$. In particular, there exists a matrix $E\in \G(\overline{k^{\tilde \s}})$ such that $E^{-1}\tilde\s^s(E)=D^{-1}\tilde\s^s(D)$. But then $\tilde\s^s(DE^{-1})=DE^{-1}$ and so $DE^{-1}\in\Gl_m(k^{\tilde \s^s})$.
Set $\hat Y=YDE^{-1}\in \Gl_m(k)$. Because $YD\in\G(k)$ and $E\in\G(k)$ we have $\hat Y\in \G(k)$. Since $\tilde\s(Y)=BY$, we have $\tilde\s^s(Y)=\tilde\s^{(s-1)}(B)\ldots\tilde\s(B)BY$ and because  $DE^{-1}\in\Gl_m(k^{\tilde \s^s})$, $\hat Y$ must satisfy the same linear $\tilde\s^s$-equation. So $\hat Y\in\G(k)$ and $\tilde\s^s(\hat Y)=\tilde\s^{(s-1)}(B)\ldots\tilde\s(B)B\hat Y$ as desired.
%
%
\end{proof}

\begin{lemma} \label{lemma: group of constants}
	Let $\G\leq \Gl_m$ be a closed subgroup defined over $k^\s$ and consider the $\s$-algebraic group $G=\{g\in\G|\ \s(g)=g\}$. If $k$ satisfies \ptweak{} (or \pt{}) then $k$ is (uniformly) strongly $G$-trivial.
\end{lemma}
\begin{proof}
	For $n\geq 1$ the $\s^n$-algebraic group $G_n=\{g\in\G|\ \s^n(g)=g\}$ is of the same form as $G$. So, using Lemma \ref{lem: pt for powers of s}, it suffices to show that $k$ is $G$-trivial. 
	
	According to \cite[Cor. 6.3]{BachmayrWibmer:TorsorsForDifferenceAlgebraicGroups} (with $d=1$ and $\psi$ the identity map) every $G$-torsor is isomorphic to a $G$-torsor of the form 
	$X'=\{x'\in\G|\ \s(x')=x'A'\}$
	for some $A'\in\G(k)$, where $G$ is acting by $(g,x)\mapsto gx$. Let $X=\{x\in\G|\ \s(x)=Ax\}$ where $A=A'^{-1}$ and $G$ is acting on $X$ by $(g,x)\mapsto xg^{-1}$. Then $X'\to X,\ x'\mapsto x'^{-1}$ is an isomorphism of $G$-torsors. The stronger versions of assumption \ptweak{} (b) and \pt{} (b) in Proposition \ref{prop: stronger P3b} (with $d=1$) imply $X(\kn)\neq\emptyset$ for some $n\geq 1$. If $k$ satisfies \pt{}, then $n$ can be chosen uniformly for all $G$-torsors.
\end{proof}

Recall that a non-commutative algebraic group of positive dimension (over an algebraically closed field of characteristic zero) is called \emph{(almost-)simple} if every proper normal closed subgroup is trivial (respectively finite). The simple algebraic groups are sometimes also referred to as the almost-simple algebraic groups of adjoint type. 

\begin{prop} \label{prop: Zariski dense in simple}
Let $G$ be a proper Zariski-dense $\s$-reduced $\s$-closed subgroup of a simple algebraic group $\G$. If $k$ satisfies \ptweak{} (or \pt{}) then $k$ is (uniformly) strongly $G$-trivial.
\end{prop}
\begin{proof}
	The idea of the proof is to show that for some $d\geq 1$ the difference algebraic group $G_d$ is isomorphic to one of the difference algebraic groups treated in Lemma \ref{lemma: group of constants}.
	
	By \cite[Prop. A.19]{DiVizioHardouinWibmer:DifferenceAlgebraicRelations} (Cf. the proof of  \cite[Prop. 7.10)]{ChatzidakiHrushovskiPeterzil:ModelTheoryofDifferenceFieldsIIPeriodicIdelas}.) we know that 
	$$G=\{g\in\G|\ \s^n(g)=\varphi(g)\}$$
	for some $n\geq 1$ and some isomorphism $\varphi\colon\G\to {^{\s^n}\!\G}$ of algebraic groups over $k$.
	Let us fix an embedding of $\G$ into $\Gl_m$ such that $\G$ is defined over $\mathbb{Q}$ inside $\Gl_m$. (This is possible because the simple algebraic groups are defined over $\mathbb{Q}$.) 
	As in the proof of \cite[Theorem A.20]{DiVizioHardouinWibmer:DifferenceAlgebraicRelations} it follows that there exists a matrix $h\in\G(k)$ and an integer $e\geq 1$ such that
	$\s^{ne}(g)=hgh^{-1}$ for all $g\in G$. 
	
	Consider the equation $\s^{ne}(Y)=hY$. The stronger versions of \ptweak{}(b) and \pt{}(b) obtained in Proposition \ref{prop: stronger P3b} (applied with $d=ne$) together with Lemma \ref{lemma: on pt (b)} both imply that there exists an integer $l\geq 1$ such that the equation $\s^{nel}(Y)=\s^{(l-1)ne}(h)\ldots\s^{ne}(h)hY$ has a solution $u\in\G(k)$. From $\s^{ne}(g)=hgh^{-1}$  it follows that $\s^{nel}(g)=\s^{(l-1)ne}(h)\ldots h g (\s^{(l-1)ne}(h)\ldots h)^{-1}$. Therefore 
	$\s^{nel}(u^{-1}gu)=u^{-1}gu$ for all $g\in G$. Now
	$g\mapsto u^{-1}gu$ is an automorphism of $\G$ that maps $G$ isomorphically to a $\s$-closed subgroup $G'$ of $\G$. Indeed,
	$$G'=\{g\in \G|\ \s^n(g)=v\varphi(g)v^{-1}\},$$
	where $v=\s^n(u^{-1})\varphi(u)\in\G(k)$.

	Let $T=(T_{ij})$ denote the $m\times m$-matrix of coordinate functions on $\G$. Then the $\s$-coordinate ring $k\{G'\}$ of $G'$ equals $$k\{G'\}=k[\G^n]=k[T,1/\det(T),\s(T),1/\det(\s(T)),\ldots,\s^{n-1}(T),1/\det(\s^{n-1}(T))]$$ with $\s$ determined by $\s(\s^{n-1}(T))=v\varphi(T)v^{-1}$.  
	As $\s^{nel}(g)=g$ for all $g\in G'$ we have $\s^{nel}(T)=T$. Thus $\s^{nel}(\s^i(T))=\s^i(T)$ for $i=0,\ldots,n-1$. This shows that for $d=nel$ the $\s^d$-algebraic group $G'_d$ is isomorphic to
	the $\s^d$-algebraic group $H=\{h\in\G^n|\ \s^d(h)=h\}$.
	Therefore, also $G_d$ is isomorphic to $H$. As $(k,\s^d)$ satisfies \ptweak{} or \pt{}, respectively, by Lemma \ref{lem: pt for powers of s}, we know from Lemma \ref{lemma: group of constants} that $H$ is (uniformly) strongly $(k,\s^d)$-trivial. So $G_d$ is (uniformly) strongly $(k,\s^d)$-trivial, and therefore $G$ is (uniformly) strongly $k$-trivial (Lemma \ref{lem Gn trivial implies G trivial}).		
\end{proof}

\subsection{The proof for $\s$-algebraic groups of $\s$-dimension zero}	

In this subsection we show that if $k$ satisfies \ptweak{} (or \pt{}) then $k$ is (uniformly) strongly $G$\=/trivial for every $\s$-algebraic group $G$ with $\sdim(G)=0$. The idea for the proof is to reduce to the three special cases treated in the three previous subsections (diagonalizable $\s$-algebraic groups, $\s$-algebraic vector groups and Zariski-dense $\s$-reduced $\s$-closed subgroups of simple algebraic groups).

\begin{lemma} \label{lemma: embedd into Ga Gm or simple}
	Assume that $k$ is algebraically closed. Let $G$ be a connected $\s$-algebraic group of positive order. Then there exists a proper normal $\s$-closed subgroup $N$ of $G$ such that $G/N$ is isomorphic to a Zariski-dense $\s$-closed subgroup of $\Ga$, $\Gm$ or a simple algebraic group.	
\end{lemma}
\begin{proof}
	First of all, note that if $\G$ is a connected algebraic group of positive dimension, then there exists a proper normal closed subgroup $\N$ of $\G$ such that $\G/\N$ is isomorphic to $\Ga$, $\Gm$ or a simple algebraic group. To see this, simply choose a proper normal closed subgroup $\N'$ of maximal dimension. If $\G/\N'$ is abelian, then it follows from the structure theory of abelian algebraic groups 
	\cite[Cor. 16.15]{Milne:AlgebraicGroups}
	that $\G/\N'$ is isomorphic to $\Ga$ or $\Gm$. If $\G/\N'$ is non-commutative, then $\G'=\G/\N'$ is almost-simple and therefore $\G'/Z(\G')$ is simple.
	
	Now let $\G$ be an algebraic group such that $G$ embeds into $\G$ as a Zariski-dense $\s$-closed subgroup. Then $\G$ is connected and has positive dimension. As argued above, there exists a proper normal closed subgroup $\N$ of $\G$ such that $\G/\N$ is isomorphic to a simple algebraic group, $\Ga$ or $\Gm$.
	
	Now $N=\N\cap G$ is a normal $\s$-closed subgroup of $G$. Suppose $\N\cap G=G$, then $G\subseteq\N$. So $\N=\G$, because $\G$ is Zariski-dense in $\G$. Therefore $N$ is properly contained in $G$ and we have an embedding $G/N\to\G/\N$. As $G$ is Zariski-dense in $\G$, we see that $G/N$ is Zariski-dense in $\G/\N$.
\end{proof}

\begin{prop} \label{prop: sdim zero}
	Let $G$ be a $\s$-algebraic group of $\s$-dimension zero. If $k$ satisfies \ptweak{} (or \pt{}) then $k$ is (uniformly) strongly $G$-trivial.
\end{prop}
\begin{proof}
	We will prove this by induction on $\ord(G)$.
	The case when $G$ has order zero is treated in Proposition \ref{prop: order zero}. So we can assume $\ord(G)>0$.
	By Corollary \ref{cor: reduce to connected} we may assume that $G$ is connected. 
	According to Lemma \ref{lemma: embedd into Ga Gm or simple} there exists a proper normal $\s$-closed subgroup $N$ of $G$ such that $G/N$ is isomorphic to a Zariski-dense $\s$-closed subgroup of $\Ga$, $\Gm$ or a simple algebraic group. 
	As $\ord(N)<\ord(G)$, we can use the induction hypothesis together with Proposition \ref{prop:InductionPrinciple} (ii) to reduce to the case that $G$ is a Zariski-dense $\s$-closed subgroup of $\Ga$, $\Gm$ or a simple algebraic group. The case when $G$ is a $\s$-closed subgroup of $\Ga$ follows from Proposition \ref{prop: vector groups of sdim 0 are strongly trivial} and the case when $G$ is a $\s$-closed subgroup of $\Gm$ follows from Proposition \ref{prop: proof for diagonalizable}. So it only remains to treat the case when $G$ is a Zariski-dense $\s$-closed subgroup of a simple algebraic group. If $G_\sred$ is a proper $\s$-closed subgroup of $G$, then we can use the induction hypothesis and Lemma \ref{lemma: reduce to sredued for strongly} to conclude. Thus we can assume that $G=G_\sred$, i.e., $G$ is a proper $\s$-reduced $\s$-closed subgroup of a simple algebraic group. This case is treated in Proposition \ref{prop: Zariski dense in simple}.	
\end{proof}

\subsection{Conclusion}

\label{sec: positive sdim}

In this subsection we complete the proof of \ptweak{}$\Rightarrow$\poweak{} by showing that a $\s$-field satisfying \ptweak{} is strongly $G$-trivial for every $\s$-algebraic group $G$.

Recall that we defined the meaning of almost-simple for $\s$-algebraic groups of $\s$-dimension zero and positive order in Definition \ref{defi: almost-simple sdim zero}. Following \cite[Def. 7.10]{Wibmer:almostsimple} we extend the definition to $\s$-algebraic groups of positive $\s$-dimension as follows. Assume that $k$ is algebraically closed and inversive. A $\s$-algebraic group over $k$ of positive $\s$-dimension is \emph{almost-simple} if it is perfectly $\s$-reduced and every proper normal $\s$-closed subgroup has $\s$-dimension zero.

\begin{lemma} \label{lemma: proof for almost simple pos dim}
	Let $G$ be an almost-simple $\s$-algebraic group with $\sdim(G)>0$. If $k$ satisfies \ptweak{} (or \pt{}), then $k$ is (uniformly) strongly $G$-trivial.
\end{lemma}
\begin{proof}
	 By \cite[Cor. 8.12]{Wibmer:almostsimple}, there exists a normal $\s$-closed subgroup $N$ of $G$ with $\sdim(N)=0$ and an algebraic group $\G$ such that $G/N\cong [\s]_k\G$. 
	As $k$ is (uniformly) strongly $N$-trivial by Proposition \ref{prop: sdim zero} and also uniformly strongly $[\s]_k\G$-trivial by Lemma \ref{lemma: algebraic case}, it follows that $k$ is (uniformly) strongly $G$-trivial by Proposition \ref{prop:InductionPrinciple}(ii).
\end{proof}

\begin{lemma} \label{lemma: Gsred is connnected}
	If $G$ is a connected $\s$-algebraic group, then also $G_\sred$ is connected.
\end{lemma} 
\begin{proof}
	As $G$ is connected, the zero ideal of $k\{G\}$ is prime. Therefore $\I(G_\sred)=\bigcup_{n\geq 1}\s^{-n}(\{0\})$ is a prime ideal. Thus $G_\sred$ is connected.
\end{proof}

Every $\s$-algebraic group $G$ with $\sdim(G)>0$ contains a unique smallest $\s$-closed subgroup $G^{so}$ with $\sdim(G^{so})=\sdim(G)$ (\cite[Lemma 6.18]{Wibmer:almostsimple}), the \emph{strong identity component} of $G$. The $\s$-algebraic group $G$ is \emph{strongly connected} if $G=G^{so}$, i.e., for every proper $\s$-closed subgroup $H$ of $G$ we have $\sdim(H)<\sdim(G)$. 
Note that $G^{so}$ is strongly connected.

\begin{theo}\label{thm main}
Let $G$ be a $\s$-algebraic group. If $k$ satisfies \ptweak{} (or \pt{}) then $k$ is (uniformly) strongly $G$-trivial.
\end{theo}
\begin{proof}	
	By Lemma \ref{cor: reduce to connected} we can assume that $G$ is connected. By Lemmas \ref{lemma: reduce to sredued for strongly} and \ref{lemma: Gsred is connnected} we can also assume that $G$ is $\s$-reduced. So $G$ is $\s$-integral and in particular perfectly $\s$-reduced. Furthermore, by Proposition \ref{prop: sdim zero} we may assume that $\sdim(G)>0$. As $G$ is perfectly $\s$-reduced, it follows from \cite[Cor. 6.26]{Wibmer:almostsimple} that $G^{so}$ is normal in $G$ and we can consider the exact sequence $1\to G^{so}\to G\to G/G^{so}\to 1$. Because $\sdim(G/G^{so})=0$, we know from Proposition \ref{prop: sdim zero} that $k$ is (uniformly) strongly $G/G^{so}$-trivial. By Proposition \ref{prop:InductionPrinciple} (ii) we can thus assume that $G$ is strongly connected.
	Then by \cite[Theorem 7.13]{Wibmer:almostsimple}, there exists a subnormal series $$G=G_0\supseteq G_1\supseteq \dots \supseteq G_n=1 $$ of strongly connected $\s$-closed subgroups such that the quotients $G_i/G_{i+1}$ are almost-simple for $i=0,\ldots,n-1$. So it follows from Lemma \ref{lemma: proof for almost simple pos dim} and Corollary \ref{cor: strong induction for tower} that $k$ is (uniformly) strongly $G$-trivial.
%
%
%
%
%
%
%
\end{proof}

We are finally in a position to complete the proof of our main results (Theorems \ref{theo: main}) and \ref{theo: main 2nd version}) from the introduction.

\medskip

\begin{proof}[Proof of Theorems \ref{theo: main} and \ref{theo: main 2nd version}]
Let $k$ be a $\s$-field (of characteristic zero). The equivalences \poweak{}$\Leftrightarrow$\pwweak{} and \po{}$\Leftrightarrow$\pw{} are in Proposition \ref{prop: P2 implies P1}. The implications \poweak{}$\Rightarrow$\ptweak{} and \po{}$\Rightarrow$\pt{} are in Proposition~\ref{prop: po implies pt}. The implications \ptweak{}$\Rightarrow$\ptweak{} and \pt{}$\Rightarrow$\po{} follow from Theorem~\ref{thm main}.
\end{proof}

%
%

\section{Torsors over $\s$-closed $\s$-fields}
\label{sec: Torsors over sclosed sfields}

In this short section, we show that $\s$-closed $\s$-fields satisfy condition (P1$^u$).

Recall that a $\s$-field $k$ is \emph{$\s$-closed} if every system of algebraic difference equations over $k$ that has a solution in a $\s$-field extension of $k$, already has a solution in $k$. Note that, due to the basis theorem (\cite[Theorem 2.5.11]{Levin:difference}), the system of algebraic difference equations in this definition can be assumed to be finite. The model theory of $\s$-closed $\s$-field has been studied extensively (see e.g., \cite{ChatzidakisHrushovski:ModelTheoryOfDifferenceFields}). In the model theoretic context, $\s$-closed $\s$-fields are usually referred to as models of ACFA. The following simple lemma translates the definition of $\s$-closed $\s$-fields into the language of \ks-algebras.

\begin{lemma} \label{lemma: sclosed}
	A $\s$-field $k$ is $\s$-closed if and only if for every finitely $\s$-generated \ks-algebra $R$ such that there exists a $\s$-ideal of $R$ that is prime, there exists a morphism $R\to k$ of \ks-algebras.
\end{lemma}	
\begin{proof}
	First assume that $k$ is $\s$-closed. Let $R$ be a finitely $\s$-generated \ks-algebra $R$ such that there exists a prime $\s$-ideal $\p$ of $R$. Then $\p^*=\{r\in R|\ \exists\ n\geq 1:\ \s^n(r)\in \p\}$ is a prime ideal of $R$ with $\s^{-1}(\p)=\p$. Therefore $R/\p^*$ is a $\s$-ring with $\s\colon R/\p^*\to R/\p^*$ injective. Thus the residue field $k(\p^*)$ of $\p^*$ is naturally a $\s$-field and the canonical map $R\to k(\p)$ is a morphism of \ks-algebras.
	Writing $R=k\{y_1,\ldots,y_n\}/I$ for some $\s$-ideal $I$ of $k\{y_1,\ldots,y_n\}$, we see that the morphism $R\to k(\p^*)$ corresponds to a solution of $I$ in $k(\p^*)$. As $k$ is $\s$-closed, this shows that there exists a solution of $I$ in $k$. This solution defines a morphism $R=k\{y_1,\ldots,y_n\}/I\to k$ of \ks-algebras.
	
	Conversely, let $F\subseteq k\{y_1,\ldots,y_n\}$ define a system of algebraic difference equations over $k$ that has a solution in some $\s$-field extension $k'$ of $k$. Then $R=k\{y_1,\ldots,y_n\}/[F]$ is a finitely $\s$-generated \ks-algebra and the solution of $F=0$ in $k'$ defines a morphism $R=k\{y_1,\ldots,y_n\}/[F]\to k'$ of \ks-algebras. The kernel of this morphism is a prime $\s$-ideal. Thus, by assumption, there exists a morphism $R=k\{y_1,\ldots,y_n\}/[F]\to k$ of \ks-algebras. The image of the $\s$-generators defines a solution of $F=0$ in $k$. Therefore, $k$ is $\s$-closed.
	%
%
\end{proof}

%
As not every finitely $\s$-generated \ks-algebra contains a prime $\s$-ideal, it is crucial that the assumption about the existence of a prime $\s$-ideal is included in Lemma \ref{lemma: sclosed}.

\begin{lemma} \label{lemma: sclosed implies snclosed}
	If $(k,\s)$ is a $\s$-closed $\s$-field, then $(k,\s^n)$ is a $\s^n$-closed $\s^n$-field for every $n\geq 1$.
\end{lemma}
\begin{proof}
	This is well-known (Corollary to Lemma 1.12 in \cite{ChatzidakisHrushovski:ModelTheoryOfDifferenceFields}) and can also quickly be seen using the constructions from Section \ref{subsection: From s to sn}: Let $R$ be a finitely $\s^n$-generated $k$-$\s^n$-algebra such that there exists a prime $\s^n$-ideal of $R$. As in the proof of Lemma \ref{lemma: sclosed}, the prime $\s^n$-ideal defines a morphism $R\to k'$ of $k$-$\s^n$-algebras into a $\s^n$-field extension $k'$ of $k$. Then ${_1R}\to{_1(k')}$ is a morphism of \ks-algebras and because ${_1(k')}$ is an integral domain (as $\s$-closed $\s$-fields are algebraically closed), we see that the kernel is a prime $\s$-ideal of ${_1R}$. Thus, by Lemma \ref{lemma: sclosed}, there exists a morphism ${_1R}\to k$ of \ks-algebras. The composition $R\to {_1R}\to k$ then is a morphism of $k$-$\s^n$-algebras. So $(k,\s^n)$ is $\s^n$-closed by Lemma \ref{lemma: sclosed}.
\end{proof}	

\begin{prop} \label{prop: sclosed implies P1}
	Let $k$ be a $\s$-closed $\s$-field and $G$ a $\s$-algebraic group over $k$. Then there exists an $n\geq 1$ such that $X(\kn)\neq\emptyset$ for every $G$-torsor $X$. In other words, a $\s$-closed $\s$-field satisfies the equivalent conditions of Theorem \ref{theo: main}.
\end{prop}	
\begin{proof}
	We first show that $k$ is uniformly strongly $G$-trivial for every connected $\s$-algebraic group $G$. So let $n\geq 1$ and let $X$ be a $G_n$-torsor. We then have, in particular, an isomorphism $k\{X\}\otimes_k k\{X\}\simeq k\{G\}\otimes_kk\{X\}$ of $k\{X\}$-algebras. Fixing a morphism $k\{X\}\to k'$ of $k$-algebras into some field extension $k'$ of $k$, we can base change this isomorphism to obtain an isomorphism $k\{X\}\otimes_k k'\simeq k\{G\}\otimes_kk'$. Because $G$ is connected (and we are in characteristic zero), $k\{G\}$ is an integral domain (\cite[Lemma 6.7]{Wibmer:almostsimple}). Moreover, $k\{G\}\otimes_kk'$ is an integral domain (e.g., using that $k$ is algebraically closed) and therefore $k\{X\}$ is an integral domain. So the zero ideal of $k\{X\}$ is a prime $\s^n$-ideal of $k\{X\}$. Because $k$ is $\s^n$-closed (Lemma \ref{lemma: sclosed implies snclosed}), there exists a morphism $k\{X\}\to k$ of $k$-$\s^n$-algebras. Thus $X$ is trivial and $k$ is uniformly strongly $G$-trivial as claimed.	
	
	Because $k$ is algebraically closed and inversive, it follows from Corollary \ref{cor: reduce to connected} that $k$ is uniformly strongly $G$-trivial for every $\s$-algebraic group $G$ over $k$. In particular, $k$ is uniformly $G$-trivial for every $\s$-algebraic group $G$ over $k$ as claimed.  	
	%
	%
\end{proof}	

While Proposition \ref{prop: sclosed implies P1} is a statement about $\s$-closed $\s$-fields, we can nevertheless deduce from it a statement about arbitrary $\s$-fields (of characteristic zero). Note that if $K$ is a $\s$-field extension of $k$ and $n\geq 1$, then $\kn=(k,\s^n)_1$ is a \ks-subalgebra of $K_1=(K,\s^n)_1$ and $K_1$ is an $n$-fold product of field extension of $k$.

\begin{cor}
	Let $G$ be a $\s$-algebraic group. Then there exists a $\s$-field extension $K$ of $k$ and an integer $n\geq 1$ such that  $X(K_1)\neq\emptyset$ for every $G$-torsor $X$. 
\end{cor}	
\begin{proof}
	Every $\s$-field can be embedded into a $\s$-closed $\s$-field (\cite[Theorem 1.1]{ChatzidakisHrushovski:ModelTheoryOfDifferenceFields}). So let $K$ be a $\s$-closed $\s$-field containing $k$. By Proposition \ref{prop: sclosed implies P1} there exists an $n\geq 1$ such that $Y(K^{[n]})\neq\emptyset$ for every $G_K$-torsor $Y$. As $X_K$ is a $G_K$-torsor, we have in particular, $X(K_1)=X_K(K^{[n]})\neq\emptyset$.
\end{proof}

One can also show directly that a $\s$-closed $\s$-field satisfies \pt{}, which, by Theorem \ref{theo: main}, is equivalent to \po{}. In fact, a $\s$-closed $\s$-field satisfies
 \pt{} (b) and \pt{} (c) with $n=1$. To see this, it suffices to show that the relevant difference equations have a solution in a $\s$-field extension of $k$. For (b), if $A\in\Gl_n(k)$ and $k(T_{ij})=k(\Gl_n)$ is the function field of $\Gl_n$, we can simply define $\s\colon k(\Gl_n)\to k(\Gl_n)$ by $\s(T)=AT$. Then clearly $T$ is a solution of $\s(Y)=AY$ in some $\s$-field extension of $k$. For (c) we can invoke the classical existence theorem from difference algebra (\cite[Theorem 7.2.1]{Levin:difference}) which implies that any non-constant difference polynomial over $k$ has a solution in some $\s$-field extension of $k$. 

\section{An application to Picard-Vessiot theory}
\label{sec: Applicationt to PV}

Difference algebraic groups arise naturally as Galois groups 
of linear differential equations depending on a discrete parameter. The solution rings in this context are called $\s$-Picard-Vessiot rings. They play a role similar to the splitting field in classical Galois theory. Torsors for difference algebraic groups arise when studying isomorphisms between two $\s$-Picard-Vessiot rings for the same linear differential equation. In this section we apply our main results (Theorems \ref{theo: main} and \ref{theo: main 2nd version} to improve the known uniqueness results for $\s$-Picard-Vessiot rings.

We briefly recall the setup from \cite{DiVizioHardouinWibmer:DifferenceGaloisTheoryOfLinearDifferentialEquations}. A \emph{$\ds$-ring} is a commutative ring $R$ together with a derivation $\de\colon R\to R$ and a ring endomorphism $\s\colon R\to R$ such that $\de(\s(r))=\hslash\s(\de(r))$ for all $r\in R$ for some fixed unit $\hslash\in R^\de$. If moreover $R$ is a field, it is called a $\ds$-field. If $K$ is a $\ds$-field, then a $K$-algebra $R$ with the structure of a $\ds$-ring such that $K\to R$ is compatible with $\de$ and $\s$ is called a \emph{$K$-$\ds$-algebra}. A \emph{$\de$-ideal} is an ideal $\ida\subseteq R$ such that $\de(\ida)\subseteq \ida$. If $\{0\}$ and $R$ are the only $\de$-ideals of $R$, the $R$ is called \emph{$\de$-simple}.

The \emph{$\de$-constants} of a $\ds$-ring $R$ are defined as $R^{\de}=\{r\in R \mid \de(r)=0\}$. Note that $R^{\de}$ is a $\s$-ring, and in particular, if $K$ is a $\ds$-field, then $K^{\de}$ is a $\s$-field. 

In what follows we work over a $\ds$-field $K$ of characteristic zero and let $k=K^\de$ be the $\s$-field of $\de$-constants. We consider a linear differential equation $\de(y)=Ay$ for some matrix $A\in K^{m\times m}$.

\begin{defi} \label{def: sPVring}
	A \emph{$\s$-Picard-Vessiot ring}
	for $\de(y)=A y$ is a $K$-$\ds$-algebra $R$ such that
	\begin{itemize}
		\item there exists a matrix $Y\in\Gl_m(R)$ with $\de(Y)=AY$ and $R=K\{Y,1/\det(Y)\}$,
		\item $R$ is $\de$-simple,
		\item $R^\de=k$. 
	\end{itemize} 
\end{defi}
The $\s$-Galois group $G$ of $R/K$ is defined as the functor from the category of \ks-algebras to the category of groups given by 
$$G(S)=\Aut^{\ds}(R\otimes_k S/K\otimes_k S)$$
for any \ks-algebra $S$, i.e., we consider automorphisms of $R\otimes_k S$ that are the identity on $K\otimes_k S$ and commute with both $\s$ and $\delta$, where $R\otimes_k S$ is considered as a $\ds$-ring with $\de$ being the trivial derivation on $S$, i.e., $\de(s)=0$ for $s\in S$. 
It is a $\s$-algebraic group over $k$ (\cite[Prop. 2.5]{DiVizioHardouinWibmer:DifferenceGaloisTheoryOfLinearDifferentialEquations}).
%

In general, a $\s$-Picard-Vessiot ring for a given differential equation $\de(y)=Ay$ is not unique (up to isomorphism). Under the assumption that the $\s$-field $k=K^\de$ is $\s$-closed, it is known that a $\s$-Picard Vessiot ring is unique up to powers of $\s$, i.e., if $R_1$ and $R_2$ are $\s$-Picard Vessiot rings for the same differential equation $\de(y)=Ay$, then there exists an integer $n\geq 1$ (possibly depending on $R_1$ and $R_2$) such that $(R_1,\de,\s^n)$ and $(R_2,\de,\s^n)$ are isomorphic as $k$-$\de\s^n$-algebras (\cite[Cor. 1.17]{DiVizioHardouinWibmer:DifferenceGaloisTheoryOfLinearDifferentialEquations}). We improve this as follows:
%

\begin{theo}
	Let $K$ be a $\delta\sigma$-field of characteristic zero and let $A\in K^{m\times m}$.
	\begin{enumerate}
		\item If $k=K^\de$ satifies the equivalent conditions of Theorem \ref{theo: main 2nd version}, then for any two $\s$-Picard-Vessiot rings $R$ and $R'$ for $\delta(y)=Ay$, there exists an $n\geq 1$ such that $R$ and $R'$ are isomorphic as $K$-$\delta\sigma^n$-algebras.
		\item If $k=K^\de$ satifies the equivalent conditions of Theorem \ref{theo: main} (e.g., $k$ is $\s$-closed), then there exists an $n\geq 1$ such that any two $\s$-Picard-Vessiot rings for $\delta(y)=Ay$ are isomorphic as $K$-$\delta\sigma^n$-algebras.
	\end{enumerate}	
%
%
\end{theo}
\begin{proof}  
	Let $R$ and $R'$ be $\s$-Picard-Vessiot rings for $\de(y)=Ay$ and let $G$ be the $\s$-Galois group of $R/K$. Consider the functor $X$ from the category of \ks-algebras to the category of sets given by
	$X(S)=\operatorname{Isom}^{\delta\sigma}_{K\otimes_k S}(R\otimes_k S,R'\otimes_k S)$ for every \ks-algebra $S$. So $X(S)$ is the set of $K\otimes_k S$\=/$\delta\s$\=/isomorphisms from $R\otimes_k S$ to  $R'\otimes_k S$, where $S$ is considered as a constant $\de$-ring.
	In the first step of the proof of 	
	 \cite[Theorem 6.4]{BachmayrWibmer:TorsorsForDifferenceAlgebraicGroups} is is shown that $X$ is a $G$-torsor.
	 Thus, if \poweak{} holds, then there exists an $n\geq 1$ such that $X(\kn)\neq\emptyset$, i.e., there exists an isomorphism 
	 \begin{equation} \label{eq: isom}
	 R\otimes_k\kn\to R'\otimes_k\kn
	 \end{equation}
	 of $K\otimes_k\kn$-$\delta\s$-algebras. The projection $\kn\to k$ onto the last factor is a morphism of $k$\=/$\s^n$\=/algebras. Thus if we base change (\ref{eq: isom}) via $\kn\to k$ we obtain an isomorphism $R\to R'$ of $K$-$\de\s^n$-algebras. This proves (i).
	 
	 If $k$ satisfies \po{}, then we can find an $n$ that is uniform in $X$, i.e., independent of $R'$. The dependence of $n$ on $R$ can then easily be eliminated: Let $R_1$ and $R_2$ be two $\s$-Picard-Vessiot rings for $\de(y)=Ay$. Then $R_1$ is isomorphic to $R$ as a $K$-$\de\s^n$-algebra and $R_2$ is isomorphic to $R$ as a $K$-$\de\s^n$-algebra. Thus $R_1$ and $R_2$ are isomorphic as $K$-$\de\s^n$-algebras.
\end{proof}


\bibliographystyle{alpha}
 \bibliography{bibdata}
\end{document}